\documentclass[11pt,british,reqno]{amsart}

\usepackage[T1]{fontenc}

\usepackage{babel,eucal,url,amssymb,enumerate,booktabs,amscd,graphics}

\theoremstyle{plain}
\newtheorem{lemma}{Lemma}[section]
\newtheorem{definition}[lemma]{Definition}
\newtheorem{proposition}[lemma]{Proposition}
\newtheorem{corollary}[lemma]{Corollary}
\newtheorem{theorem}[lemma]{Theorem}
\newtheorem{remark}[lemma]{Remark}

\newtheorem*{ack}{Acknowledgements}

\DeclareMathOperator{\Ric}{Ric}

\newcommand{\Lie}[1]{\operatorname{\textsl{#1}}}
\newcommand{\lie}[1]{\operatorname{\mathfrak{#1}}}

\newcommand{\un}{\lie{u}}
\newcommand{\Gtwo}{\ifmmode{{\rm G}_2}\else{${\rm G}_2$}\fi}

 \newcommand{\cyclic}{\mathop{\kern0.9ex{{+}\kern-2.2ex\raise-.28ex\hbox{\Large\hbox
 {$\circlearrowright$}}}}}

\def\sideremark#1{\ifvmode\leavevmode\fi\vadjust{\vbox to0pt{\vss
 \hbox to 0pt{\hskip\hsize\hskip1em
 \vbox{\hsize2.5cm\tiny\raggedright\pretolerance10000
 \noindent #1\hfill}\hss}\vbox to8pt{\vfil}\vss}}}%

\input xy
\xyoption{all}

\newfont{\eusm}{eusm10 scaled \magstep1}
\newfont{\eusmiii}{eusm10 scaled \magstep3}

\newcommand{\comp}{\makebox[7pt]{\raisebox{1.5pt}{\tiny $\circ$}}}

\newcommand{\trace}{\mathop{\rm trace}}

\title{Homogeneous nearly K\"ahler manifolds}

\author{J.~C.~Gonz{\'a}lez~D{\'a}vila}
\address[J.~C.~Gonz{\'a}lez~D{\'a}vila]{Department of Fundamental Mathematics\\
  University of La Laguna\\ 38200 La Laguna, Tenerife, Spain}
\email{jcgonza@ull.es}

\author{F.~Mart\'\i n~Cabrera}
\address[F.~Mart\'\i n~Cabrera]{Department of Fundamental Mathematics\\ CSIC Associated Unity\\
  University of La Laguna\\ 38200 La Laguna, Tenerife, Spain}
\email{fmartin@ull.es}

\begin{document}
\maketitle

\begin{abstract}{\indent}
The structure of nearly K{\"a}hler manifolds was studied by Gray in
several papers, mainly in \cite{G1}. More recently, a relevant
progress on the subject has been done by Nagy. Among other results,
he proved that a complete strict nearly K{\"a}hler manifold is
locally a Riemannian product of homogeneous nearly K{\"a}hler
spaces, twistor spaces over quaternionic K{\"a}hler manifolds and
six-dimensional nearly K{\"a}hler manifolds, where the homogeneous
nearly K{\"a}hler factors are also 3-symmetric spaces. In the
present paper,  using the lists of $3$-symmetric spaces given by
Wolf\&Gray,   we display the exhaustive list of irreducible
 simply connected homogeneous strict nearly K\"ahler manifolds. For such manifolds, we
 give details  relative to the intrinsic torsion and  the Riemannian  curvature. Additionally, we
 determine the canonical fibration for those with special
 algebraic torsion.

\end{abstract}

\begin{figure}[b]  
 \hspace{-20mm}
  \hspace{-60mm}  {\footnotesize minimal connection}
\\
\hspace{-91mm}
 \end{figure}

\tableofcontents

\section{Introduction}\indent \label{uno}
Nagy proved in \cite{Nagy2} that every complete, simply connected
 nearly K\"ahler manifold $M$ is a Riemannian product
 $M_{0}\times M_{1}\times \dots \times M_{k},$ where $M_{0}$ is
K\"ahler and $M_{i},$ for each $i\in\{1,\dots, k\},$ is an
irreducible strict nearly K\"ahler manifold belonging to the
following list: homogeneous nearly K\"ahler manifolds, twistor
spaces over positive quaternionic K\"ahler manifolds and
six-dimensional nearly K\"ahler manifolds. More concretely, from
this paper and using also a result of  Butruille \cite{Bu}, relative
to six-dimensional homogeneous nearly K\"ahler manifolds, one
obtains the following classification:

Any simply connected, complete, irreducible and strict nearly
K\"ahler manifold $(M^{2n},g,J)$ belongs to one of the following
classes:
\begin{enumerate}
\item[{\rm (i)}] \textit{Homogeneous NK Type I}: The holonomy
representation of the minimal connection $\nabla^{\Lie{U}(n)}$ of
the nearly K\"ahler structure is real irreducible and for $n =3,$
$(M,g,J)$ is moreover holomorphically isometric to $\Lie{S}^{6} =
\Lie{G}_{2}/\Lie{SU}(3)$ equipped with the nearly K\"ahler structure
as $3$-symmetric space. \item[{\rm (ii)}] {\em Homogeneous NK Type
II}:  The holonomy representation  is complex
  $\mathrm{Hol}(\nabla^{\Lie{U}(n)})$-irre\-du\-ci\-ble and there exists a $\nabla^{\Lie{U}(n)}$-parallel
      decomposition $\mathrm{T}M = {\mathcal E}\oplus
J{\mathcal E}.$

\item[{\rm (iii)}] \textit{Homogeneous NK Type III}:  $(M,g,J)$ has
special algebraic torsion (see Section \ref{three}) and its
horizontal distribution ${\mathcal H}$ is complex
$\mbox{Hol}(\nabla^{\Lie{U}(n)})$-reducible.

\item[{\rm (iv)}] \textit{Homogeneous NK Type IV}: $(M,g,J)$ has
special algebraic torsion,  the base space of the twistor fibration
is symmetric and its horizontal distribution ${\mathcal H}$ is real
$\mbox{Hol}(\nabla^{\Lie{U}(n)})$-irreducible.
 \item[{\rm (v)}] Six-dimensional non-homogeneous
nearly K\"ahler manifolds.
 \item[{\rm (vi)}] Twistor spaces of
non-symmetric positive quaternionic K\"ahler manifolds.
\end{enumerate}

\noindent Here the homogeneous nearly K\"ahler manifolds of Types
I, III and IV are defined with light modifications from the
corresponding classes given by Nagy in \cite{Nagy2}. The
definitions here presented seem to be more natural. Furhermore, we
will see that they are more naturally agree with the lists of
$3$-symmetric spaces given by Wolf\&Gray in \cite{WG} and by Gray
in \cite{G}. The class homogeneous NK Type III includes the
corresponding original one given by Nagy. However, the set of
non-Einstein manifolds included in  the class homogeneous NK Type
IV is slightly smaller than the set of manifolds included in the
corresponding Nagy's class.

Using Proposition \ref{p3-sym} ((i) equivalent to (v)) and
\cite[Theorem 2.1]{Nagy2} originally due to Cleyton\&Swann
\cite{Cleyton1}, Nagy proves that nearly K\"ahler manifolds
contained in the classes (I)-(IV), for dimension higher than six,
are in fact homogeneous and moreover, they are $3$-symmetric
spaces where $J$ is the canonical almost complex structure.
Afterwards,  Butruille completes this result proving that
six-dimensional homogeneous strict nearly K\"ahler are also
$3$-symmetric spaces.

It is still open the problem exposed by O'Brian and Rawnsley in
  \cite[p. 57]{O'Brian-Raw} about the existence or not of six-dimensional
non-homogeneous nearly K\"ahler manifolds (see \cite{Bur-Gu-Raw}).
Since any positive quaternionic K\"ahler manifold of dimension
four is symmetric \cite{Hitchin}, six-dimensional non-homogeneous
nearly K\"ahler manifolds must admit real irreducible holonomy
representation of the minimal connection.

The only known examples of positive quaternionic K\"ahler
manifolds are the so-called {\em Wolf-spaces}, which are all
symmetric and the only homogeneous examples due to Alekseevskii
(see \cite[Ch. 14]{Be}). This motives the LeBrun-Salamon Conjecture: {\em
Every positive quaternionic K\"ahler manifold is a Wolf-space}.
Then the existence of twistor spaces in (vi) corresponds with the
existence of counter-examples of the conjecture. According with
\cite{Nagy2}, such twistor spaces are non-homogeneous nearly
K\"ahler. Therefore,  it follows that {\em any simply connected
strict homogeneous nearly K\"ahler manifold is a compact naturally
reductive $3$-symmetric space equipped with its canonical complex
structure}.

In the present paper, using the lists of $3$-symmetric spaces given by Wolf\&Gray in
\cite{WG} (see also \cite{G}), we display the exhaustive list of irreducible
 simply connected  homogeneous strict nearly K\"ahler manifolds in each one of these classes, NK Type I-IV. Concretely, we prove the following.
 \begin{theorem}\label{tmaincla} If $(M,g,J)$ is
 a simply connected homogeneous strict nearly K\"ahler manifold,
  then $M$ decomposes as Riemannian product whose factors belong
  to one of the following classes of homogeneous strict nearly K\"ahler manifolds:
\begin{enumerate}
\item[{\rm (i)}] homogeneous NK Type I: compact irreducible
$3$-symmetric spaces of Type $A_{3}IV$ $($Table {\rm \ref{tab:AIV})}
or holomorphically isometric to $\Lie{Spin}(8)/[\Lie{SU}(3)/{\mathbb
Z}_{3}];$
\item[{\rm (ii)}] homogeneous NK Type II: compact irreducible
$3$-symmetric spaces of Type $C_{3}$ $($Table {\rm \ref{tab:VI})} or
holomorphically isometric to $\Lie{Spin}(8)/\Lie{G}_{2};$
\item[{\rm (iii)}] homogeneous NK Type III: compact irreducible $3$-symmetric spaces of Type
$A_{3}II$ $($Table {\rm \ref{tab:AII});}
\item[{\rm (iv)}] homogeneous NK Type IV: compact
irreducible $3$-symmetric spaces of Type $A_{3}III$ {\rm (}Table
{\rm \ref{tab:AIII}).}
\end{enumerate}
\end{theorem}

\noindent Moreover, we give details  relative to the Riemannian  curvature and
 determine the canonical fibration for those ones with special
 algebraic torsion (NK Types III and IV). It is shown  that homogeneous NK of Type III  are twistor
spaces of irreducible Hermitian symmetric spaces of compact type, meanwhile those of Type IV  are twistor spaces of irreducible
non-Hermitian symmetric spaces.

Finally, we would like to note that we also give a list of
irreducible simply connected  strict nearly K{\"a}hler manifolds
which are Einstein. To complete such a list, it remains to find
those which are six-dimensional, non-homogeneous and
$\mathrm{Hol}(\nabla^{\Lie{U}(n)})$-holonomy real irreducible,
where  $\mathrm{Hol}(\nabla^{\Lie{U}(n)})$ is the holonomy group
of the minimal connection $\nabla^{\Lie{U}(n)}.$ The existence of these last manifolds
is still an open problem as it was above mentioned.

\section{Preliminaries}\indent
An {\em almost Hermitian manifold} $(M,g,J)$ is a $2n$-dimensional
Riemannian manifold $(M,g)$ equipped with an almost complex
structure $J$ compatible with the metric. That is, $J^{2} = -Id$ and
$g(JX,JY) = g(X,Y),$ for all vector fields $X,Y$ on $M.$ The
presence of an almost Hermitian structure is equivalent to say that
$M$ is equipped with a $\Lie{U}(n)$-structure (its frame bundle
admits a reduction to $\Lie{U}(n)).$ Let $\nabla$ be the Levi-Civita
connection of $g.$ On the set of Hermitian connections, metric
connections $\widetilde{\nabla}$ such that $\widetilde{\nabla}J =
0,$ we consider the unique connection $\nabla^{\Lie{U}(n)}$ given by
$\nabla^{\Lie{U}(n)} = \nabla + \xi$, where $\xi \in \mbox{T}^* M
\otimes \un(n)^\perp,$ i.e. $\xi J + J \xi = 0.$
$\nabla^{\Lie{U}(n)}$ is called the
 \textit{minimal $\Lie{U}(n)$-connection}
 and $\xi$ is referred to as the
  \textit{intrinsic torsion}
of the $\Lie{U}(n)$-structure, which can be computed explicitly as
(see~\cite{Yano:torsionHermitianstru})
\begin{equation} \label{torsion:xi}
  \xi_X Y = - \tfrac12 J\left( \nabla_X J \right) Y.
  \end{equation}
Then, $\xi=0$ if and only if $J$ is parallel with respect  to the
Levi-Civita connection, that is, M is K\"ahlerian.

An almost Hermitian manifolds $(M,g=\langle \cdot , \cdot
\rangle,J)$  is called
 \textit{nearly K\"ahler}
 if $J$ satisfies $(\nabla_{X}J)X = 0,$ for all vector
fields $X$ on $M,$ or equivalently, the intrinsic torsion is
totally skew-symmetric. On nearly K\"ahler manifolds, $\xi$ is
parallel with respect to the minimal connection
$\nabla^{\Lie{U}(n)}$ (see \cite{GM,Kir}) and we have
\begin{equation}\label{curvatura}
\langle R_{X\,Y}Z,W\rangle - \langle R_{X\,Y}JZ,JW\rangle =
4\langle\xi_{X}Y,\xi_{Z}W\rangle,
\end{equation}
where $R$ is the Riemannian curvature of $M$ with the sign
convention $R_{X\,Y} = \nabla_{[X,Y]} - [\nabla_{X},\nabla_{Y}].$
For studying the curvature $R$ is useful to consider the usual Ricci
curvature tensor  $\Ric$ associated to the metric structure,
$\Ric(X,Y) = \langle \Ric X,Y\rangle = \langle R_{X \,e_{i}}Y,
e_{i}\rangle,$ and another tensor $\Ric^*$, called the {\it
$\ast$-Ricci curvature tensor}, associated to the almost Hermitian
structure and defined by $\Ric^* (X,Y) = \langle \Ric^{*}X,Y\rangle
= \langle R_{X  \, e_i} JY , Je_i\rangle,$ where $\{e_{1},\dots ,e_{2n}\}$ is a local frame field. For sake of simplicity
the summation convention will be used throughout the exposition.
However, the sum will be explicitly written when a risk of ambiguity
appears.
 For nearly K{\"a}hler manifolds the tensor $\Ric^{*}$ is symmetric and
\[
\Ric(J\cdot,J\cdot) = \Ric(\cdot,\cdot),\;\;\;\; \Ric^*(J\cdot,J\cdot) = \Ric^*(\cdot,\cdot).
\]
From (\ref{curvatura}), the difference $r = \Ric -\Ric^*$
satisfies
\[
r(X,Y) = \langle rX,Y\rangle =
4\langle\xi_{X}e_{i},\xi_{Y}e_{i}\rangle = -4\trace\; \xi_{X}\comp
\xi_{Y}.
\]
Then, it is a positive semidefinite symmetric form. Moreover, the following formula holds \cite{G1,Wat}
 \begin{equation}\label{rr}
 \sum_{i,j =1}^{2n}\langle re_{i},e_{j}\rangle(\langle
R_{X\,e_{i}}Y,e_{j}\rangle - 5\langle
R_{X\,e_{i}}JY,Je_{j}\rangle) = 0,
 \end{equation}
for all $X,$ $Y$ in $\mathrm{T}M$. The tensor $r$ has other
strong geometric properties: It commutes with $J$ and
$\nabla^{\Lie{U}(n)}r = 0$ (see \cite{Nagy1}). Then its
eigenvalues are constants.

If for all $m\in M$ and $u\in \mathrm{T}_{m}M$ with $u\neq 0$ we
have $\nabla_{u}J \neq 0,$ or equivalently $\xi_{u}\neq 0,$ we say
that the nearly K\"ahler manifold $(M,g,J)$ is {\em strict}. Thus
a strict nearly K\"ahler manifold could be to be considered as the
opposite of a K\"ahler manifold. Here, $r$ becomes into a positive
symmetric form and its eigenvalues are all different from zero.
For the complete case, a strict nearly K\"ahler manifold has
positive Ricci curvature, hence it is compact with a finite
fundamental group \cite[Theorem 1.1]{Nagy1}. Also in \cite{Nagy1}
it is proved that any nearly K\"ahler manifold might be decomposed
locally (or globally, if it is complete and simply connected) into
the Hermitian product of a K\"ahler manifold and a strict nearly
K\"ahler manifold. Thus, the study of nearly K\"ahler manifolds
can be reduced to the strict case. Note that irreducible complete
and simply connected non-K\"ahler nearly K\"ahler manifolds are
strict and so, they are compact.

We also recall that the first Chern class \cite{G1} of $(M,J)$ is
represented by the closed two-form $\gamma_{1}$ defined by
$8\pi\gamma_{1} = \langle CX,JY\rangle$ for all $X,Y$ in
$\mathrm{T}M$. Here $C$ denotes the symmetric endomorphism
$\Ric-5\Ric^{*}.$ It is $\nabla^{\Lie{U}(n)}$-parallel and commutes
with $J.$ As $\gamma_{1}$ is a closed form it is obtained
 \begin{equation}\label{Cprop}
 C\xi_{X}Y + \xi_{X}CY + \xi_{CX}Y = 0.
 \end{equation}
Hence, if $C$ has single eigenvalue, then $C =0$ and $\Ric =
\frac{5}{4}r.$

\vspace{1mm}

\section{Irreducible nearly K\"ahler manifolds}\indent
 \label{three}
  In this section, we expose  some results
given by Nagy \cite{Nagy2}, including  further details,  and
another new results.
\begin{definition} \cite{Nagy2} {\rm Let $(M,g,J)$ be a complete strict
nearly K\"ahler manifold. It is said to have}  special algebraic
torsion {\rm if there exists a $\nabla^{\Lie{U}(n)}$-parallel
orthogonal decomposition $TM = {\mathcal V}\oplus {\mathcal H}$ such
that  ${\mathcal V}$, ${\mathcal H}$ are stable by $J$,
$\xi_{\mathcal V}{\mathcal V} = 0$ and $\xi_{\mathcal H}{\mathcal H}
= {\mathcal V}.$ }
\end{definition}
It is not hard to check that  the condition  $\xi_{\mathcal
H}{\mathcal H} = {\mathcal V}$ can be replaced by  $\xi_{\mathcal
H}{\mathcal H} \subseteq {\mathcal V}$. Since ${\mathcal V}$ is
$\nabla^{\Lie{U}(n)}$-parallel and $\xi_{\mathcal V}{\mathcal V} =
0,$ it is an integrable distribution. If $(M,g,J)$ has special
algebraic torsion,  then there exists a Riemannian manifold
$(N,h)$ and a Riemannian submersion with totally geodesic fibers
$\pi:M\to N,$ called the {\em canonical} fibration, whose vertical
distribution equals ${\mathcal V}.$ With respect to the induced
metric and almost complex structures, each fiber is a Hermitian
symmetric space of compact type \cite[Proposition 4.2]{Nagy2}.
Moreover, taking a new metric  $\bar{g}$ on $M$ given by $\bar{g}
= 2g_{\mathcal V}\oplus g_{\mathcal H},$ where $g_{\mathcal V}$
and $g_{\mathcal H}$ denote the restrictions of the metric $g$ to
the distributions ${\mathcal V}$ and ${\mathcal H},$ respectively,
and an almost complex structure $\bar{J}$ by setting
$\bar{J}_{\mathcal V} = -J$ and $\bar{J}_{\mathcal H} = J,$ Nagy
proved in \cite[Proposition 4.3]{Nagy2} that $(M,\bar{g},\bar{J})$
is a simply connected K\"ahler manifold of positive Ricci
curvature. Hence nearly K\"ahler manifolds with special algebraic
torsion are {\em twistor spaces} (complex manifolds fibrering over
a real manifold such that the fibres are complex submanifods
\cite{Bur-Gu-Raw}) over Riemannian manifolds, which  are also
compact and, because the fibers are connected, simply connected.

Along this section we will assume that  $(M^{2n},g =\langle \cdot
, \cdot \rangle,J)$ is a connected, complete and irreducible (as Riemannian manifold) strict
nearly K\"ahler manifold. If it has special algebraic torsion, the fibre
of the canonical fibration is irreducible \cite[see Theorem
4.1]{Nagy2} and recall that the base space $N$ has to be
irreducible, too. Moreover, $(M^{2n},g,J)$ is (\cite[see Sections
5 and 6]{Nagy2}):
\begin{enumerate}
\item[{\rm (i)}] homogeneous nearly K\"ahler if and only if $N$ is
a symmetric space. \item[{\rm (ii)}] non-homogeneous space if and
only if $N$ is a non-symmetric positive quaternionic K\"ahler
manifold. In this case,  $M$ is the twistor bundle on $N$
consisting of local almost complex structures compatible with the
quaternionic structure and $\dim \, \mathcal{V} = 2$.
\end{enumerate}
According with \cite[Theorem 3.1]{Nagy2}, the notion of special
algebraic torsion can be characterised in terms of holonomy:
$(M,g,J)$ has special  algebraic torsion if and only if the holonomy
representation of the minimal connection $\nabla^{\Lie{U}(n)}$ is
complex reducible. \vspace{1mm}

For the complex irreducible case, the following possibilities may
occur:
\begin{enumerate}
\item[{\rm (i)}] the $\nabla^{\Lie{U}(n)}$-holonomy representation is (real) irreducible;
\item[{\rm (ii)}] there exists an irreducible $\nabla^{\Lie{U}(n)}$-parallel distribution ${\mathcal V}$ of $M$ such that $TM =
{\mathcal V} \oplus J{\mathcal V}.$
\end{enumerate}
Nearly K\"ahler manifolds $(M,g,J)$ satisfying (i), at least for
dimension higher than six, or (ii) are homogeneous nearly K\"ahler.
In fact,  they are $3$-symmetric spaces, $J$ is their canonical
complex structure and correspond in the classification given in
Section \ref{uno} with the {\em homogeneous NK Type I}, including
$\Lie{S}^{6} = \Lie{G}_{2}/\Lie{SU}(3),$ and {\em homogeneous NK
Type II}, respectively.
\begin{proposition}\label{complexirr}
On strict nearly K\"ahler manifolds $(M,g,J)$ with complex
irreducible $\nabla^{\Lie{U}(n)}$-holonomy representation, the first
Chern class of $(M,J)$ vanishes, $r$ is parallel and $g$ is positive
Einstein metric.
\end{proposition}
\begin{proof}
Let $\mathcal{E}_\lambda$ be an eigenbundle of $r$ corresponding
to some eigenvalue $\lambda >0.$ Since $r$ is
$\nabla^{\Lie{U}(n)}$-parallel, $\mathcal{E}_{\lambda}$ is a
$\nabla^{\Lie{U}(n)}$-parallel distribution. Hence, if the
$\nabla^{\Lie{U}(n)}$-holonomy representation is real irreducible,
one obtains that $\mathcal{E}_{\lambda}=\mathrm{T}M.$ If the
$\nabla^{\Lie{U}(n)}$-holonomy representation is not real
irreducible, taking into account that $J{\mathcal E}_{\lambda} =
{\mathcal E}_{\lambda},$ it follows that ${\mathcal
E}_{\lambda}\cap{\mathcal V}$ is a non-zero
$\nabla^{\Lie{U}(n)}$-parallel distribution and so, ${\mathcal
E}_{\lambda}$ must be all $\mathrm{T}M.$ In both cases, $\lambda$
is the unique eigenvalue of $r.$ Now, the result follows using the
next Lemma \ref{lei}, a consequence from (\ref{rr}) (see
\cite{G1}, \cite{Nagy1}).
\end{proof}

\begin{lemma}\label{lei} If the tensor $r$ on a strict nearly K\"ahler manifold $(M,g,J)$ has exactly one eigenvalue, then $(M,g)$ is a positive Einstein manifold. Moreover, the first Chern class of $(M,J)$ vanishes and $r$ is parallel.
\end{lemma}
\begin{remark}{\rm Also $r$ is parallel on nearly K\"ahler manifolds such that $r$ has exactly two eigenvalues, one of which is zero \cite[Theorem 4.13]{G1}. In particular, $r$ is parallel for $\dim M \leq 8.$}
\end{remark}

In \cite{Nagy2} the following lemma was proved for $\mathcal{L} =
\mathrm{T} M$. Here we rebuilt the argument  for a more general
situation.
\begin{lemma} \label{trespuntodoscor11}
For a nearly Kähler manifold, let  $\mathcal{L} \subseteq
\mathrm{T} M$  be a $\nabla^{U(n)}$-parallel distribution which is
$\mathrm{Hol}(\nabla^{U(n)})$-reducible. If one has  an
$\nabla^{U(n)}$- parallel orthogonal decomposition $\mathcal{L}
=\mathcal{E} \oplus \mathcal{F}$
 and $X$
 in $\mathcal{F}$, $Y$ in $\mathrm{T} M$ and $V, W$ in
${\mathcal{E}}$, then
\begin{equation} \label{trespuntounolem11bis}
R^{\Lie{U}(n)}(X,Y,V,W)=  4 \left(  \langle [\xi_V, \xi_W
]X,Y\rangle -  \langle \xi_X Y, \xi_V W \rangle \right),
\end{equation}
where $R^{\Lie{U}(n)}$ denotes the curvature tensor of the minimal
connection $\nabla^{\Lie{U}(n)}$.

Furthermore, if the  decomposition $\mathcal{L} =\mathcal{E}
\oplus \mathcal{F}$ is stable by $J$, then we have the following
consequences of the previous identity:

 {\rm (i)}
$\xi_{X_1} \xi_{V_1} V_2 =0$,  $\xi_{V_1} \xi_{X_1} X_2 =0$;
\hspace{1cm}  {\rm (ii)}   $\xi_{X_1} \xi_{X_2} X_3 \in
\mathcal{F}$,
 $\xi_{V_1} \xi_{V_2} V_3 \in \mathcal{E}$;

{\rm (iii)}   $\xi_{V_1} \xi_{V_2} X_1 \in \mathcal{F}$,
$\xi_{X_1} \xi_{X_2} V_1 \in \mathcal{E}$,

\noindent for all $X_1,X_2,X_3$ in $\mathcal{F}$ and $V_1,V_2,V_3$
in ${\mathcal{E}}$.

In particular, if we consider $\mathrm{T}M = \mathcal{L} \oplus
\mathcal{L}^\perp$, where $\mathcal{L}$ is stable by $J$ and
$\mathcal{L}^\perp$ is its  orthogonal complement, then
$$
\xi_{X_1} \xi_{X_2} X_3, \xi_{V_1} \xi_{V_2} X_1 \in \mathcal{L},
\qquad\xi_{V_1} \xi_{V_2} V_3,  \xi_{X_1} \xi_{X_2} V_1 \in
\mathcal{L}^\perp,
$$
for all $X_1,X_2,X_3 \in \mathcal{L}$ and  $V_1,V_2,V_3 \in
\mathcal{L}^\perp$.

\end{lemma}
\begin{proof}
 Let us recall the relation (see \cite[page 237]{G1}) :
\begin{equation} \label{trespuntouno}
 R^{\Lie{U}(n)} (X,Y,Z,T) =R(X,Y,Z,T)- 2\langle \xi_XY, \xi_ZT
\rangle +
 \langle \xi_X Z, \xi_YT\rangle -\langle \xi_X T, \xi_Y Z\rangle.
\end{equation}

The first  Bianchi identity is given by
$R(X,Y,V,W)+R(Y,V,X,W)+R(V,X,Y,W)=0$. Also we have
$R^{\Lie{U}(n)}(Y,V,X,W)= \langle R^{\Lie{U}(n)}(Y,V)X, W\rangle=
0$. Note that $R^{\Lie{U}(n)}(Y,V)X \in \mathcal{F}$ and $W \in
\mathcal{E}$. Therefore, using (\ref{trespuntouno}),  it is
obtained
$$
\begin{array}{lr}
R(Y,V,X,W)=  2 \langle \xi_Y V, \xi_X W\rangle  - \langle \xi_Y
X,\xi_{V}W\rangle + \langle \xi_Y W,\xi_{V} X\rangle.
\end{array}$$
An analogous formula is obtained for $R(V,X,Y,W)$. Such a formula is
given by
$$
 R(V,X,Y,W)= - R(Y,W,X,V) =  - 2 \langle \xi_Y W, \xi_X V\rangle +
\langle \xi_Y X,\xi_{W}V\rangle - \langle \xi_Y V,\xi_{W} X\rangle.
$$
Finally, from these last two formulas, the first Bianchi identity
and equation (\ref{trespuntouno}), the required equation follows.

 Next let us see (i). From  Equation \eqref{trespuntounolem11bis} we have
 $$R^{\Lie{U}(n)}(X_1,Z,V_1,V_2)- 4 \langle
[\xi_{V_1}, \xi_{V_2}]X_1,Z\rangle = 4  \langle  \xi_{X_1} \xi_{V_1}
V_2 ,Z \rangle,
$$
for all $Z$ in $\mathrm{T}M$. Now taking
$R^{\Lie{U}(n)}(JX_1,JZ,V_1,V_2)=R^{\Lie{U}(n)}(X_1,Z,V_1,V_2)$
 and the algebraic properties of the
tensor $\xi$ into account, it follows
$$
4  \langle  \xi_{X_1}  \xi_{V_1} V_2, Z \rangle = - 4  \langle
\xi_{JX_1} \xi_{V_1} V_2, JZ \rangle =  4  \langle  \xi_{X_1}
\xi_{V_1} V_2 , Z \rangle =0.
$$
 Reversing the roles of $\mathcal{E}$ and $\mathcal{F}$,  we obtain that
$\xi_{V_1} \xi_{X_1} X_2=0$. \vspace{2mm}

Now we see (ii). Using (i) we have  $\langle \xi_{V_1} \xi_{X_1}
X_2, X_3 \rangle =0$, for all   $X_3$ in $\mathcal{F}$. Therefore,
$\langle \xi_{X_3} \xi_{X_1} X_2,V_1 \rangle =0$. Thus $
(\xi_{X_3} \xi_{X_1} X_2)_{\mathcal{L}} \in \mathcal{F}$, where
$Z_{\mathcal{L}}$ denotes the projection of $Z$ on $\mathcal{L}$.
Similarly, we will obtain   $(\xi_{V_1} \xi_{V_2}
V_3)_{\mathcal{L}} \in \mathcal{E}$. Now we apply the facts here
proved to $\mathrm{T}M = \mathcal{L} \oplus \mathcal{L}^\perp$,
where $\mathcal{L}^\perp$ is the orthogonal complement of
$\mathcal{L}$. Therefore,
$$
\xi_{X_3} \xi_{X_1} X_2 = (\xi_{X_3} \xi_{X_1} X_2)_{\mathrm{T}M}
\in \mathcal{L}, \qquad \xi_{V_1} \xi_{V_2} V_3 = (\xi_{V_1}
\xi_{V_2} V_3)_{\mathrm{T}M} \in \mathcal{L}.
$$
Finally,
$$
\xi_{X_3} \xi_{X_1} X_2 =(\xi_{X_3} \xi_{X_1} X_2)_{\mathcal{L}} \in
\mathcal{F}, \qquad \qquad  \xi_{V_1} \xi_{V_2} V_3 = (\xi_{V_1}
\xi_{V_2} V_3)_{\mathcal{L}} \in \mathcal{E}.
$$

 \vspace{2mm}

\noindent For  (iii) we  choose $Y=V_3$ in Equation
\eqref{trespuntounolem11bis}, we have
$$R^{\Lie{U}(n)}(X_1,V_3,V_1,V_2)= 4 \langle [\xi_{V_1}, \xi_{V_2} ]X_1,V_3\rangle + 4  \langle
\xi_{X_1} \xi_{V_1} V_2 , V_3 \rangle = 4 \langle [\xi_{V_1},
\xi_{V_2} ]X_1,V_3\rangle.
$$
Since $R^{\Lie{U}(n)}(X,V_3,V_1,V_2)= \langle
R^{\Lie{U}(n)}(V_1,V_2) X_1, V_3 \rangle = 0$ and, by (ii),
$[\xi_{V_1},\xi_{V_2}]X_1$ is in $\mathcal{L}$, we get that
$[\xi_{V_1},\xi_{V_2}]X_1$ belongs to $\mathcal{F}$. But $[\xi_{V_1}
, \xi_{JV_2}]JX_1=-\xi_{V_1}\xi_{V_2}X_1 - \xi_{V_2}\xi_{V_1} X_1 $
is equally in $\mathcal{F}$.  The remaining identity follows by
reversing the roles of the distributions ${\mathcal{E}}$ and
$\mathcal{F}$.
\end{proof}

 \vspace{2mm}

Because we will need later, we recall the following two results
proved in \cite[see Lemma 3.1, Proposition 4.1 and Lemma
5.1]{Nagy2}.
\begin{lemma} \label{trespuntounolem} For
a nearly K\"ahler manifold with special algebraic torsion,
 we have:
\begin{enumerate}
\item[{\rm (i)}]  $(\nabla^{\Lie{U}(n)}_U R^{\Lie{U}(n)})(V_1,V_2,V_3,V_4) =0;$
\item[{\rm (ii)}] $R^{\Lie{U}(n)} (X,U,V_{1},V_{2})= 4(
\langle [\xi_{V_{1}}, \xi_{V_{2}}]X,U\rangle - \langle \xi_X U,
\xi_{V_{1}} V_{2}\rangle),$
\end{enumerate}
for all $U \in TM,$ $V_1,V_2,V_3,V_4 \in \mathcal{V}$ and $X\in
{\mathcal H}.$
\end{lemma}
\noindent Note that (ii) follows from Equation
\eqref{trespuntounolem11bis} taking $\mathcal{L} = \mathrm{T} M =
\mathcal{V} \oplus \mathcal{H}$.

\vspace{2mm}

\noindent On complete irreducible nearly Kähler manifolds with
special algebraic torsion, the eigenvalues of the tensors $\Ric$,
$r$ and $C$ on $\mathcal{H}$ are ruled by the following result.
\begin{lemma} \label{cincopuntounolem}
The base manifold $(N,h)$ is Einstein with Einstein constant $\mu
>0$ and, furthermore,
$$
\mathrm{Ric}+\frac{r}{4}=\mu \, 1_{\mathcal{H}} = \frac32 r -
\frac14 C
$$
on $\mathcal{H}$. As a consequence the tensors $r$, $C$ and
$\mathrm{Ric}$ have the same eingenbundles on $\mathcal{H}$. If
$\mathcal{E}_\mathcal{H}$ is one such a eingenbundle associated to
the eigenvalue $\lambda(r)$, $\lambda(C)$ and
$\lambda(\mathrm{Ric})$ of the tensors $r$, $C$ and $\mathrm{Ric}$,
respectively, then
\begin{equation} \label{idhor}
4 \mu = 4 \lambda(\mathrm{Ric}) + \lambda(r) =  6 \lambda(r) -
\lambda(C).
\end{equation}
\end{lemma}

Now, we have already  available  the necessary tools to prove next
result.
\begin{proposition}\label{lcomplete}
Let $(M,g,J)$ be a complete, simply connected, irreducible strict
nearly K\"ahler manifold. Then the following conditions are
equivalent:
\begin{enumerate}
\item[{\rm (i)}] $r$ has exactly one eigenvalue; \item[{\rm (ii)}]
the first Chern class of $(M,J)$ vanishes, i.e. $C =0;$
  \item[{\rm
(iii)}] $r$ is $\nabla$-parallel; \item[{\rm (iv)}] $C$ is
$\nabla$-parallel; \item[{\rm (v)}] $(M,g)$ is Einstein.
\end{enumerate}
\end{proposition}
\noindent Besides Lemmas \ref{lei} and \ref{cincopuntounolem},
we shall need the next lemmas for getting its proof.
\begin{lemma}\label{lCparallel}
Let $(M,g,J)$ be strict nearly K\"ahler manifold. If $\nabla C = 0$ then $C =0.$
\end{lemma}
  \begin{proof}
Let $u\in \mathrm{T}M$ such that $Cu = \lambda u,$ for some
$\lambda \in \mathbb{R}.$ Since $\nabla C = 0$ one gets $C\xi_{X}Y
= \xi_{X}CY.$ Then (\ref{Cprop}) implies that $C\xi_{u}X =
-\frac{\lambda}{2}\xi_{u}X.$ Because $(M,g,J)$ is strict there
exists $X\in \mathrm{T}M$ such that $\xi^{2}_{u}X = -c^{2}X,$ for
some $c\neq 0.$ Hence $X$ and $\xi_{u}X$ are eigenvectors of $C$
with eigenvalue $-\frac{\lambda}{2}.$ Repeating the same  process
for $X,$ one gets $C\xi_{X}u = \frac{\lambda}{4}\xi_{X}u.$ Hence
$\xi_{u}X$ also belongs to the eigenspace with eigenvalue
$\frac{\lambda}{4}.$ Then $\lambda$ must be zero and so, $C=0.$
\end{proof}
\begin{lemma}\label{lN}\cite{Nagy2}
Let $(M,g,J)$ be a complete, simply connected, strict nearly
K\"ahler manifold. If $C$ has a single eigenvalue, then it splits
as a Riemannian product whose factors are strict nearly K\"ahler
manifolds such that their corresponding tensors $C$ and $r$ have
exactly one eigenvalue.
\end{lemma}
\noindent {\em Proof of Proposition \ref{lcomplete}.} Using  Lemma
\ref{lei}, we have $({\rm i})\Rightarrow ({\rm ii})$ and $({\rm
i})\Rightarrow ({\rm iii}).$ From Lemma \ref{lN}, $({\rm
ii})\Rightarrow ({\rm i})$ and $({\rm iii})\Rightarrow ({\rm ii})$
follows using \cite[Corollary 4.12]{G1}. From  Lemma
\ref{lCparallel}, $({\rm ii})\Leftrightarrow ({\rm iv})$ is deduced.
$({\rm i})\Rightarrow ({\rm v})$ follows from Lemma \ref{lei}.
Finally, we prove $({\rm v)} \Rightarrow ({\rm ii})$. From
Proposition \ref{complexirr}, we only have to consider the case when
$(M,g,J)$ has special algebraic torsion. Then the result follows
from Lemma \ref{cincopuntounolem}. \hfill $\Box$ \vspace{2mm}

In what follows we focus on irreducible strict nearly K\"ahler
manifolds with special algebraic torsion. The fact that $C$ is
$\nabla^{\Lie{U}(n)}$-parallel and symmetric, together with the
irreducibility of the fiber implies that there exists a real
constant $\lambda$ such that $C=\lambda 1_{\mathcal{V}}$ on
${\mathcal{V}}$ (because of the irreducibility, only one eigenspace
is possible on $\mathcal{V}$). This leads to the following:
\begin{proposition} \label{cincopuntounopro}
Let $(M,g,J)$ be a complete irreducible nearly Kähler manifolds with
special algebraic torsion. Then, for the tensor $C$, we have:
\begin{enumerate}
\item[{\rm (i)}] the restriction of $C$ to $\mathcal{V}$ has only one
eigenvalue $\lambda$;
\item[{\rm(ii)}] the restriction of $C$ to $\mathcal{H}$ has at the most
two eigenvalues: $\lambda_1 = - \frac{\lambda}{2} - \rho$ and
$\lambda_2 =  - \frac{\lambda}{2} + \rho$, for some $\rho \geq 0$.
In particular, if there is one single eigenvalue for $C$  on
$\mathcal{H}$, then it is equal to $- \frac{\lambda}{2}$. On the
other hand, if $\lambda_1$ and $\lambda_2$  are distinct and
associated to the eingenbundles $\mathcal{E}$ and $\mathcal{F}$,
respectively, then $\mathcal{H}=\mathcal{E} \oplus \mathcal{F}$,
$\xi_{\mathcal{E}} \mathcal{E} =\xi_{\mathcal{F}} \mathcal{F}=0$,
$\xi_{\mathcal{E}} \mathcal{F}= \mathcal{V}$, $\xi_{\mathcal{V}}
\mathcal{E} = \mathcal{F}$, $\xi_{\mathcal{V}} \mathcal{F} =
\mathcal{E}$.

\item[{\rm (iii)}]  $C$ has at most three
eigenvalues on $\mathrm{T}M$. In particular, if $C$ has only one
eigenvalue, then $\lambda=0$.
\end{enumerate}
\end{proposition}
\begin{proof} We only need to show (ii). Since $\mathcal{V}$ is an
eigenspace of $C$ with eigenvalue $\lambda,$ the identity (\ref{Cprop}) implies that $S
\xi_V  X  + \xi_V  SX=0$,  whenever $V,X$ are in $\mathcal{V}$ and
$\mathcal{H}$ respectively, where
$S=C+\frac{\lambda}{2}1_{\mathcal{H}}$ on $\mathcal{H}$. Let us
denote by $\mathcal{L}_V^\mathcal{H}$ the projection of the Lie
derivative $\mathcal{L}_V$ on $\mathcal{H}$. The facts that  $S^2$
is $\nabla^{\Lie{U}(n)}$-parallel and ${\mathcal{L}}_V^\mathcal{H}
S^2=0$, for all $V$ in ${\mathcal{V}}$, imply that $S^2$ projects
on a symmetric, $\nabla^h$-parallel endomorphism of $N,$ where
$\nabla^{h}$ denotes the Levi-Civita connection of $(N,h).$ Such
an endomorphism has to be a multiple of identity by the
irreducibility of $N$. Hence $S^2 = \rho^2 1_{\mathcal{H}}$, for
some $\rho \geq 0$. This implies
$$
(C + (\textstyle \frac{\lambda}{2} + \rho ) 1_{\mathcal{H}}) (C +
(\frac{\lambda}{2} - \rho ) 1_{\mathcal{H}}) =0.
$$
Therefore, $C$ has at most two distinct eigenvalues on
$\mathcal{H}$, $\lambda_1 = \textstyle -\frac{\lambda}{2} - \rho$
and $\lambda_2 = \textstyle -\frac{\lambda}{2} + \rho$. In
particular, if $\lambda_1 \neq \lambda_2$, for all  $V\in
\mathcal{V}$ and $X_1 \in \mathcal{E}$, one has
$$
C \xi_V X_1 = - \lambda \xi_V X_1 - \lambda_1 \xi_V X_1= \lambda_2
\xi_V X_1 \in \mathcal{H} \cap \mathcal{E}_{\lambda_2} =
\mathcal{F},
$$
where $\mathcal{E}_{\lambda_2}$ is the eigenbundle associated to the
eigenvalue $\lambda_2$.  Thus, we have $\xi_{\mathcal{V}}
\mathcal{E} \subseteq \mathcal{F}$ and, by a similar proof,
$\xi_{\mathcal{V}} \mathcal{F} \subseteq \mathcal{E}$. From this it
is immediate that $\xi_{\mathcal{E}} \mathcal{E} = \xi_{\mathcal{F}}
\mathcal{F} =0$. In fact, for all $X_1, Y_1 \in \mathcal{E}$ and $V
\in \mathcal{V}$, one has
$$
\langle \xi_{X_1} Y_1 , V \rangle = \langle \xi_{V} X_1 , Y_1
\rangle =0.
$$
\noindent Finally, $\xi_{\mathcal{V}} \mathcal{H} = \mathcal{H}$
implies $\xi_{\mathcal{V}} \mathcal{E} = \mathcal{F}$ and
$\xi_{\mathcal{V}} \mathcal{F} =
 \mathcal{E}$ and $\xi_{\mathcal{E}} \mathcal{F} = \mathcal{V}$.
 \end{proof}

We will give some further descriptions of the possible eigenvalues
of the tensors $r$, $\Ric$ and $C$. For such a purpose, the
following technical result is useful.
 \begin{lemma} \label{cincopuntotreslem}
For $X,Y \in \mathcal{H}$, we have
$$
 r(X,Y) = 8 \langle \xi_{e_i} X,
\xi_{e_i} Y \rangle = 8 \langle \xi_{x_j} X, \xi_{x_j} Y \rangle,
$$
where $\{e_{i}\}$ and  $\{x_{j}\}$ are local orthonormal frames in
$\mathcal{V}$ and $\mathcal{H}$, respectively.

In particular, when there is only one eigenvalue $k$ for $r$ on
$\mathcal{H}$, we have
\begin{equation} \label{ricricast2}
 2 l \dim \,\mathcal{V} = k \dim \,\mathcal{H},
\end{equation}
 where $l$ is the unique
eigenvalue of $r$ on $\mathcal{V}$.

 Moreover, if
there is an orthogonal  $\nabla^{U(n)}$-parallel decomposition
$\mathcal{H} = \mathcal{E} \oplus \mathcal{F}$ stable by $J$ where
$\mathcal{E}$ and $\mathcal{F}$  are   eigenbundles  associated to
the eigenvalues  $k$ and $m$ (not necessarily distinct) of $r$ such
that  $\xi_{\mathcal{E}} \mathcal{E} =\xi_{\mathcal{F}}
\mathcal{F}=0$, then:
$$
\begin{array}{lcl}
r(V,W) & = & 8 \langle \xi_{y_{j_1}} V, \xi_{y_{j_1}} W \rangle =
8 \langle \xi_{z_{j_2}} V, \xi_{z_{j_2}} W
\rangle,\;\;\;\;\; {\mbox for}\;V, W \in \mathcal{V},\\[0.4pc]
r(X_1,Y_1) & = & 8 \langle \xi_{e_i} X_1, \xi_{e_i} Y_1 \rangle =
8 \langle \xi_{z_{j_2}} X_1, \xi_{z_{j_2}} Y_1 \rangle,\;\;\;\;\;
{\mbox for}\;X_1, Y_1 \in
\mathcal{E},\\[0.4pc]
r(X_2,Y_2) & = & 8 \langle \xi_{e_i} X_2, \xi_{e_i} Y_2 \rangle = 8
\langle \xi_{y_{j_1}} X_2, \xi_{y_{j_1}} Y_2 \rangle,\;\;\;\;\;
{\mbox for}\;X_2, Y_2 \in \mathcal{F},
\end{array}
$$
where $\{ e_i \}$,  $\{ y_{j_1} \}$ and $\{ z_{j_2} \}$ are local
orthonormal frames in $\mathcal{V}$, $\mathcal{E}$ and
$\mathcal{F}$, respectively. Furthermore,  we also have
\begin{equation} \label{ricricast}
 l \dim
\,\mathcal{V} = k \dim \,\mathcal{E} = m \dim \,\mathcal{F}.
\end{equation}

 \end{lemma}
\begin{proof} For $X\in \mathcal{H},$
\[
r(X,X) = 4(\sum_{i}\langle \xi_{X}e_{i}, \xi_{X}e_{i}\rangle +
\sum_{j}\langle \xi_{X}x_{j},\xi_{X}x_{j}\rangle).
\]
But, using the fact that $\xi_\mathcal{H}
\mathcal{H}={\mathcal{V}},$ we get
\[
\sum_{i}\langle \xi_{X}e_{i}, \xi_{X}e_{i}\rangle =
\sum_{i,j}\langle \xi_{X}e_{i},x_{j}\rangle\langle x_{j},
\xi_{X}e_{i}\rangle = \sum_{i,j}\langle
\xi_{X}x_{j},e_{i}\rangle\langle e_{i}, \xi_{X}x_{j}\rangle =
\sum_{j}\langle \xi_{X}x_{j},\xi_{X}x_{j}\rangle.
\]
It proves the first equality. In particular, if there is only one
eigenvalue on $\mathcal{H}$, it follows
$$
l \dim \, \mathcal{V} = r(e_i,e_i) = 4 \langle \xi_{x_j} e_i ,
\xi_{x_j} e_i \rangle = \frac12 r(x_j , x_j) = \frac12 k \dim \,
\mathcal{H}.
$$

For the rest, we use the identities $\xi_{\mathcal{E}} \mathcal{E}
=\xi_{\mathcal{F}} \mathcal{F}=0$, $\xi_{\mathcal{E}} \mathcal{F}=
\mathcal{V}$, $\xi_{\mathcal{V}} \mathcal{E} = \mathcal{F}$ and
$\xi_{\mathcal{V}} \mathcal{F} = \mathcal{E}$ and the corresponding
arguments are similar.
\end{proof}

\begin{lemma} \label{trespuntouno11bisbis}
For a nearly Kähler manifold $(M^{2n},g,J)$, if  $\mathcal{L}
\subseteq \mathrm{T}M$  is a stable by $J$ and  complex
$\mathrm{Hol}(\nabla^{\Lie{U}(n)})$-reducible distribution, then
$\mathcal{L}=\mathcal{E} \oplus \mathcal{F}$, where $\mathcal{E}$,
$\mathcal{F}$  are  non-zero  orthogonal
$\nabla^{\Lie{U}(n)}$-parallel distributions such that they are
stable by $J$ and $\xi_{\mathcal{E}} \mathcal{E} \subseteq
\mathcal{E}$.
\end{lemma}
\begin{proof}
 The reducibility of $\mathcal{L}$ implies the existence of a
$\nabla^{\Lie{U}(n)}$-parallel decomposition $\mathcal{L}=
\mathcal{E}_1 \oplus \mathcal{F}_1$ stable by $J$. Let
$\mathcal{F}_0$ be the distribution generated by elements of the
form $(\xi_V W)_{\mathcal{F}_1}$, where $V,W$ belong to
$\mathcal{E}_1$ and the subscript denotes orthogonal projection on
$\mathcal{F}_1$. Using Lemma \ref{trespuntodoscor11} (i) we get
$$
 \xi_X \xi_Y (\xi_V  W)_{\mathcal{E}_1}
+ \xi_X \xi_Y (\xi_V W)_{\mathcal{F}_1} + \xi_X \xi_Y (\xi_V
W)_{\mathcal{L}^\perp} = \xi_X \xi_Y (\xi_V  W) = \xi_X 0 =0,$$
 whenever $X,Y$ are in $\mathcal{F}_1$. Now,
using  Lemma \ref{trespuntodoscor11}, it follows that  the first
summand belongs to $\mathcal{E}_1$, the second is in
$\mathcal{F}_1$ and third in $\mathcal{L}^\perp$. Therefore, each
summand  vanishes, $ \xi_X \xi_Y (\xi_V W)_{\mathcal{E}_1} =\xi_X
\xi_Y (\xi_V W)_{\mathcal{F}_1}= \xi_X \xi_Y (\xi_V
W)_{\mathcal{L}^\perp} =0$. This  implies that
$\xi_{\mathcal{F}_1} \mathcal{F}_0=0$. In fact, if $X \in
\mathcal{F}_1$ and $(\xi_V W)_{\mathcal{F}_1} \in \mathcal{F}_0$,
then we compute
$$
\langle \xi_X (\xi_V W)_{\mathcal{F}_1} , \xi_X (\xi_V
W)_{\mathcal{F}_1} \rangle = - \langle \xi_X \xi_X (\xi_V
W)_{\mathcal{F}_1} , (\xi_V W)_{\mathcal{F}_1} \rangle =  0.
$$
Therefore, $\xi_X (\xi_V W)_{\mathcal{F}_1} =0$.

   As $\mathcal{F}_0$
is contained in $\mathcal{F}_1$, we get that $\xi_{\mathcal{F}_0}
\mathcal{F}_0=0$. Moreover, $\mathcal{F}_0$ is
$\nabla^{\Lie{U}(n)}$-parallel and stable by $J$, In fact, using
$\nabla^{\Lie{U}(n)}\xi=0$, one has

\begin{eqnarray*}
\nabla^{\Lie{U}(n)}_Z (\xi_V W)_{\mathcal{F}_1} & = &
(\xi_{\nabla^{\Lie{U}(n)}_Z V} W)_{\mathcal{F}_1} + (\xi_V
\nabla^{\Lie{U}(n)} W)_{\mathcal{F}_1}\in  \mathcal{F}_0.
\end{eqnarray*}

From all of this, if $\mathcal{F}_0 \neq 0$,   we may take
${\mathcal{E}}=\mathcal{F}_0$. On the other hand, if
$\mathcal{F}_0=0$, then $\xi_{\mathcal{E}_1} \mathcal{E}_1
\subseteq \mathcal{E}_1$ and we set $\mathcal{E}=\mathcal{E}_1$.
Therefore, we have obtained a $\nabla^{\Lie{U}(n)}$-parallel
orthogonal decomposition $\mathcal{L}=\mathcal{E} \oplus
\mathcal{F}$ such that $\xi_{\mathcal{E}} {\mathcal{E}}\subseteq
{\mathcal{E}} $ and $\mathcal{E}, \mathcal{F}$ are non-zero and
stable by $J$.
\end{proof}

\begin{proposition} \label{trespuntouno11probis}
Let $(M^{2n},g,J)$ be an complete, simply connected irreducible
strict nearly K\"ahler manifold with special algebraic torsion.
Then we have:
\begin{enumerate}
\item[{\rm (i)}] The distribution $\mathcal{H}$ is   complex
$\mathrm{Hol}(\nabla^{\Lie{U}(n)})$-reducible if and only if
$\mathcal{H}=\mathcal{E} \oplus \mathcal{F}$, where $\mathcal{E}$,
$\mathcal{F}$  are  non-zero  orthogonal
$\overline{\nabla}$-parallel distributions   such that they are
stable by $J$, $\xi_{\mathcal{E}} \mathcal{E} = \xi_{\mathcal{F}}
\mathcal{F}= 0$ and $\xi_{\mathcal{E}} \mathcal{F} = \mathcal{V}$.
Moreover, $\mathcal{E}$ and $\mathcal{F}$ are real irreducible
$\mathrm{Hol}(\nabla^{\Lie{U}(n)})$-spaces.

\item[{\rm (ii)}] The distribution $\mathcal{H}$ is   complex
$\mathrm{Hol}(\nabla^{\Lie{U}(n)})$-irreducible if and only if
$\mathcal{H}$ is real
$\mathrm{Hol}(\nabla^{\Lie{U}(n)})$-irreducible.
\end{enumerate}
\end{proposition}
\begin{proof}
Let us see part (i). By making use of Lemma
\ref{trespuntouno11bisbis}, if $\mathcal{H}$ is complex
$\mathrm{Hol}(\nabla^{\Lie{U}(n)})$-reducible, then
$\mathcal{H}=\mathcal{E} \oplus \mathcal{F}$, where $\mathcal{E}$,
$\mathcal{F}$  are  non-zero  orthogonal
$\nabla^{\Lie{U}(n)}$-parallel distributions such that they are
stable by $J$ and $\xi_{\mathcal{E}} \mathcal{E} \subseteq
\mathcal{E}$. Since $\xi_{\mathcal{H}} \mathcal{H} = \mathcal{V}$,
we have $\xi_{\mathcal{E}} \mathcal{E} = \{0\}$ and $\mathcal{V} =
\xi_{\mathcal{E}} \mathcal{F} \oplus \xi_{\mathcal{F}}
\mathcal{F}$.

Let us suppose that $\xi_{\mathcal{F}} \mathcal{F} \neq \{0 \}$.
We recall that $\mathcal{V}$ is a stable by $J$ and real
$\mathrm{Hol}(\nabla^{\Lie{U}(n)})$-irreducible, because we are
assuming our manifold $M$ is complete irreducible strict nearly
K\"ahler with special algebraic torsion (see \cite{Nagy2}).
Therefore, $\xi_{\mathcal{F}} \mathcal{F} = \mathcal{V}$. Thus,
the distribution $\mathcal{V} \oplus \mathcal{E}$ and
$\mathcal{F}$ can respectively play the roles of vertical and
horizontal distribution for the special algebraic torsion
condition. In such a case, the fiber $F$ of the canonical
fibration is not irreducible, because $ \mathrm{T} F =\mathcal{V}
\oplus \mathcal{E}$ where $\mathcal{V}$, $\mathcal{E}$ are
non-zero. Hence $M$ would be not irreducible, contradiction. Hence
it follows $\xi_{\mathcal{F}} \mathcal{F} = \{0 \}$. As a
consequence, $\mathcal{V} = \xi_{\mathcal{E}} \mathcal{F}$.

The converse is immediate.

The final assertion of part (i) follows from the fact that $M$ is
irreducible and both distributions, $\mathcal{E}$ and
$\mathcal{F}$ can play the role of vertical distribution.
\vspace{2mm}

Now we prove part (ii). Let us suppose that $\mathcal{H}$ is a
complex irreducible but  real reducible as
$\mathrm{Hol}(\nabla^{\Lie{U}(n)})$-space. Therefore, $\mathcal{H}
= \mathcal{E} \oplus J \mathcal{E}$, where $\mathcal{E}$ is a
non-zero $\nabla^{\Lie{U}(n)}$-parallel distribution.

Let $\{ e_1, \dots e_{\dim \, \mathcal{V}}\}$ be an orthonormal
basis for (real) vectors in $\mathcal{V}$ and  $X \in
\mathcal{E}$. Because $\mathcal{E}$ is
$\nabla^{\Lie{U}(n)}$-parallel,  $R^{\Lie{U}(n)}(e_i, Je_i)X$ is
in $\mathcal{E}$. Now using Equation \eqref{trespuntounolem11bis}
and Lemma \ref{cincopuntotreslem}, we have
$$
\sum_{i} R^{\Lie{U}(n)}(e_i, Je_i)X = 8 \sum_{i} \xi_{e_i}
\xi_{e_i} JX = - r \, JX.
$$
But since $\mathcal{H}$ is complex irreducible and $r$ is
Hermitian,  there is a single positive eigenvalue $k$ of $r$ on
$\mathcal{H}$. Hence $r \, J X = k JX$. Thus $JX$ is in
$\mathcal{E} \cap J \mathcal{E} = \{ 0\}$. This implies $X =0$,
for all $X$ in $\mathcal{E}$. Contradiction.
\end{proof}

\begin{definition}
{\rm  A complete irreducible strict nearly K\"ahler manifold is
\textit{homogeneous NK Type $III$}, if it has special algebraic
torsion and the horizontal distribution $\mathcal{H}$ is complex
$\mathrm{Hol}(\nabla^{\Lie{U}(n)})$-reducible.}
\end{definition}

\begin{proposition}
If  $(M,g,J)$ is homogeneous NK Type $III$,  then $(M,g)$ is a
homogeneous space and the base space of the canonical fibration is
a simply connected, compact and irreducible symmetric space.
\end{proposition}
\begin{proof}
By Proposition \ref{trespuntouno11probis}, we have three different
splittings of special algebraic type $\mathrm{T}M={\mathcal{V}}
\oplus \mathcal{H}$, $\mathrm{T}M= \mathcal{E} \oplus
({\mathcal{V}} \oplus \mathcal{F})$, $\mathrm{T}M=\mathcal{F}
\oplus ({\mathcal{V}} \oplus \mathcal{E})$ each of which being
$\nabla^{\Lie{U}(n)}$-parallel and stable by $J$. It follows by
Lemma \ref{trespuntounolem} (i) that the restriction of the tensor
$\nabla^{\Lie{U}(n)}_U R^{\Lie{U}(n)}$ to either ${\mathcal{V}}$,
$\mathcal{E}$ or $\mathcal{F}$ vanishes. Moreover, because our
distributions are $\nabla^{\Lie{U}(n)}$-parallel we already know,
for all $U,U' \in \mathrm{T} M$, $\;V,W \in \mathcal{V}$, $X_1 \in
\mathcal{E}$ and  $X_2 \in \mathcal{E}$
$$
R^{\Lie{U}(n)} (U,U',V,X_1)=0, \qquad
R^{\Lie{U}(n)}(U,U',V,X_2)=0, \qquad R^{\Lie{U}(n)}(U,U',
X_1,X_2)=0.
$$
Finally, for all $\;V,W \in \mathcal{V}$,  $X_1,Y_1 \in
\mathcal{E}$ and $X_2,Y_2 \in \mathcal{F}$, by Lemma
\ref{trespuntounolem} (ii), one has
\begin{eqnarray*}
 R^{\Lie{U}(n)}(V,W, X_1,Y_1) & = & -  R^{\Lie{U}(n)}(X_1,V, W,  Y_1)- R^{\Lie{U}(n)}(W, X_1,V, Y_1)
  \\
  &=& - 4 \langle \xi_{V} W, \xi_{X_1} Y_1 \rangle - 4 \langle \xi_{X_1} V, \xi_{W} Y_1 \rangle
   - 4 \langle \xi_{W} X_1, \xi_{V} Y_1 \rangle
   \\
   & = &  4 \langle \xi_{V} X_1, \xi_{W} Y_1 \rangle
   - 4 \langle \xi_{W} X_1, \xi_{V} Y_1 \rangle.
\end{eqnarray*}
Similarly,  one has also
\begin{eqnarray*}
 R^{\Lie{U}(n)}(V,W, X_2,Y_2) & = &   4 \langle \xi_{V} X_2, \xi_{W} Y_2 \rangle
   - 4 \langle \xi_{W} X_1, \xi_{V} Y_1 \rangle,
   \\
 R^{\Lie{U}(n)}(X_1,Y_1, X_2,Y_2)    & = &  4 \langle \xi_{X_1} X_2, \xi_{Y_1}
Y_2 \rangle
   - 4 \langle \xi_{Y_1} X_2, \xi_{X_1} Y_2 \rangle.
\end{eqnarray*}
Since $\nabla^{\Lie{U}(n)} \xi=0$, all of this implies
$\nabla^{\Lie{U}(n)}R^{\Lie{U}(n)} =0$. Therefore,
$\nabla^{\Lie{U}(n)}$ is an Ambrose-Singer connection. Hence,
taking into account that $(M,g)$ is a connected, simply connected
and complete Riemannian manifold, it is a homogenous space
\cite{TV}. The claim that the base space of the canonical
fibration is symmetric follows by  comparison between curvature
tensors of the total and the base space (see \cite{Be}). In fact,
if $(N,h)$ is the base manifold with Levi-Civita connection
$\nabla^h$ associated to the metric $h$, the corresponding
Riemannian curvature $R^h$ and
 $X^{\mathcal{H}}$ denotes the horizontal lift of $X \in \mathfrak{X}(N)$, then
$$
0 = \langle (\nabla^{\Lie{U}(n)}_{X'^{\mathcal{H}}}
R^{\Lie{U}(n)})_{X^{\mathcal{H}}\, Y^{\mathcal{H}}}
Z^{\mathcal{H}}, Z'^{\mathcal{H}} \rangle = \langle
((\nabla^h_{X'} R^h)_{X \, Y} Z)^{\mathcal{H}}, Z'^{\mathcal{H}}
\rangle,
$$
for all $X,X',Y,Z,Z'\in \mathfrak{X}(N)$. This implies $\nabla^h
R^h=0$. Then, because $N$ is simply connected, it follows that
$(N,h)$ is symmetric.
\end{proof}

\begin{definition}
{\rm  A complete irreducible strict nearly K\"ahler manifold is
\textit{homogeneous NK Type $IV$}, if it has special algebraic
torsion, the base space of the twistor fibration is symmetric  and
the horizontal distribution $\mathcal{H}$ is real
$\mathrm{Hol}(\nabla^{\Lie{U}(n)})$-irreducible.}
\end{definition}

From the fact that the base manifold is symmetric it follows
$\nabla^{\Lie{U}(n)} R^{\Lie{U}(n)}=0$ \cite{Nagy2}. Therefore, a
homogeneous NK Type $IV$ manifold is homogeneous as it is
expected.

\begin{remark}
{\rm In the case that $(M^{2n},g,J)$ is a complete irreducible
non-homogeneous nearly Kähler manifolds with special algebraic
torsion $($a twistor space of a non-symmetric positive
quaternionic Kähler manifold \cite{Nagy2}$)$, by Proposition
\ref{trespuntouno11probis} (ii),  the horizontal distribution
$\mathcal{H}$ is real
$\mathrm{Hol}(\nabla^{\Lie{U}(n)})$-irreducible. }
\end{remark}

Finally, we give further details relative to the eigenvalues of
the tensors $r$, $\Ric$ and $C$.

\begin{proposition} \label{satdimensions} Let $(M,g,J)$ be a complete irreducible nearly Kähler manifolds with
special algebraic torsion. Then we have:

\noindent $\mathrm{(i)}$ If $\mathcal{H}$ is a complex $\mathrm{Hol}
(\nabla^{\Lie{U}(n)})$-reducible space, then $\mathcal{H} =
\mathcal{E} \oplus \mathcal{F}$, where $\mathcal{E}$ and
$\mathcal{F}$ are stable by $J$, $\xi_{\mathcal{E}} \mathcal{E} =
\xi_{\mathcal{F}} \mathcal{F}= 0$  and  non-zero $\mathrm{Hol}
(\nabla^{\Lie{U}(n)})$-spaces  such that  consisting of eigenvectors
of the tensors $r$, $C$ and $\mathrm{Ric}$ associated to eigenvalues
$($not necessarily distinct$)$  which satisfy the following table
\vspace{2mm}

\begin{center}
\begin{tabular}{ccccc}
\toprule
Eigenvalue  &  $r$ &  $\Ric$ &  $C$ & Eigenbundle \\[1mm]
\midrule
$\lambda $ & $l$& $\frac14 (l+2k +2m)$&  $  2(2l - k -m)$ &
${\mathcal{V}}$
\\[1mm]
 $\lambda_1$ & $k$ & $\frac14 (2l + k  +2m)$& $ 2(- l + 2k -m)$& $\mathcal{E}$\\[1mm]
$\lambda_2$ & $m$ & $\frac14 (2l + 2k +m)$& $ 2(-l -k + 2m)$ & $\mathcal{F}$\\[1mm]
\bottomrule
\end{tabular}
\end{center}
\vspace{2mm}

\noindent
 The eigenvalues $l$,
$k$ and $m$ are positive constants satisfying Equation {\rm
$($\ref{ricricast}$)$}. The Einstein constant of the base manifold
is given by $2 \mu = l + k + m$.

 \vspace{1mm}

\noindent $\mathrm{(ii)}$ If $\mathcal{H}$ is a complex irreducible
$\mathrm{Hol} (\nabla^{\Lie{U}(n)})$-space or complex reducible as
in the previous case but with $k=m$, then the tensors $r$, $C$ and
$\Ric$ have two eigenvalues $($not necessarily distinct$)$ and the
bundles $\mathcal{V}$ and $\mathcal{H}$ consist of eigenvectors
associated to the first and second eigenvalue, respectively, as it
is indicated in the following table. \vspace{2mm}

\begin{center}
\begin{tabular}{ccccc}
\toprule
Eigenvalue  &  $r$ &  $\Ric$ &  $C$ & Eigenbundle \\[1mm]
\midrule $\lambda $ & $l$&  $\frac{1}{4} (l+4k) $& $4(l-k)$ &
${\mathcal{V}}$
\\[1mm]
$\lambda_1 $ & $k$ & $\frac{1}{4}(2l +3k)$& $2(k-l)$ &
$\mathcal{H}$\\[1mm]
\bottomrule
\end{tabular}
\end{center}
\vspace{2mm}

\noindent The eigenvalues  $l$ and $k$ are positive constants
satisfying Equation {\rm $($\ref{ricricast2}$)$}. The Einstein
constant $\mu$ of the base manifold is given by  $2\mu= l + 2k$.
\end{proposition}

\begin{proof} The distribution $\mathcal{V}$ is a real  irreducible $\mathrm{Hol}
(\nabla^{\Lie{U}(n)})$-space. In fact, if $\mathcal{V} =
\mathcal{V}_1 \oplus \mathcal{V}_2$, where $\mathcal{V}_i$ are
$\nabla^{\Lie{U}(n)}$-parallel, then we also have $\nabla_{V} W =
\nabla^{\Lie{U}(n)}_{V} W \in \mathcal{V}_i$, for all $V \in
\mathcal{V}$ and $W \in \mathcal{V}_i$, $i=1,2$. This implies that
the distributions $\mathcal{V}_i$ are parallel for the Levi Civita
connection $\nabla$  on a maximal leaf of $\mathcal{V}$ (we recall
that the leaf is a totally geodesic submanifold of $M$). Hence
$\mathcal{V}$ is not irreducible with respect to such a Levi Civita
connection, contradiction.

Since $\mathcal{V}$ is a real   irreducible $\mathrm{Hol}
(\nabla^{\Lie{U}(n)})$-space and $C$ is a
$\nabla^{\Lie{U}(n)}$-parallel tensor, then $\mathcal{V}$ is
contained in a single eingenbundle $\mathcal{E}_\lambda$ of $C$,
$\mathcal{V} \subseteq \mathcal{E}_\lambda$. In other case, for some
eigenvalue $\lambda$, we would have $\mathcal{E}_\lambda \cap
\mathcal{V} \neq \mathcal{V} $ and $\mathcal{V}$ would not be
$\mathrm{Hol} (\nabla^{\Lie{U}(n)})$-irreducible.

By Proposition  \ref{trespuntouno11probis}, if $\mathcal{H}$ is
complex reducible as a $\mathrm{Hol} (\nabla^{\Lie{U}(n)})$-space,
then $\mathcal{H} = \mathcal{E} \oplus \mathcal{F}$  in the
conditions written in the statement. Therefore, $\mathcal{E}$ and
$\mathcal{F}$ are real irreducible as $\mathrm{Hol}
(\nabla^{\Lie{U}(n)})$-spaces. This implies that there is only one
eigenvalue for the tensor $C$ on each of them. Thus, the situation
is  that $C$ has at the most three eigenvalues, namely, $\lambda$,
$\lambda_1$ and $\lambda_2$ (not necessarily distinct) such that
$\mathcal{V}$, $\mathcal{E}$ and $\mathcal{F}$ consist of
eigenvectors for $\lambda$, $\lambda_1$ and $\lambda_2$,
respectively.  Moreover, $\lambda = - (\lambda_1 + \lambda_2)$.

 Now we recall the expression for the Ricci curvature of $(M,g)$
obtained by Nagy in \cite{Nagy1}. If $\mathcal{E}_i, 1 \le i \le
p$ are the eigenspaces of  $r$ corresponding to the eigenvalues
$\lambda_i$ then
\begin{equation} \label{NagyRicci}
\mathrm{Ric}(X,Y)=\frac{\lambda_i}{4} \langle X,Y \rangle
+\frac{1}{\lambda_i} \sum_{s=1}^{p}\lambda_s\langle  r^s(X),Y
\rangle,
\end{equation}
for all $X$and $Y$ in $V_i$, where the tensors $r^s, 1 \le s \le p$
are defined by
$$
 \langle r^s(X),Y \rangle
=-\mathrm{trace}_{\mathcal{E}_s}(\nabla_XJ)(\nabla_YJ) = 4 \sum_{j_s=1}^{\dim
\mathcal{E}_s} \langle \xi_X e_{j_s}
 , \xi_Y e_{j_s} \rangle.
$$
(Note that formula (\ref{NagyRicci}) is also true for a
decomposition $\mathrm{T}
 M = \mathcal{E}_1 \oplus \dots \oplus \mathcal{E}_p$ into eigenbundles $\mathcal{E}_i$
 corresponding to $\lambda_i$, not necessarily distinct, such that
 the bundles $\mathcal{E}_i$ are $\nabla^{\Lie{U}(n)}$-parallel and stable by $J$.) Then,
 using Lemma \ref{cincopuntotreslem} and   the expression for
the Ricci tensor given by (\ref{NagyRicci}),
 we have that the eigenvalues
of $\mathrm{Ric}$ corresponding to $\mathcal{V}$, $\mathcal{E}$ and
$\mathcal{F}$, respectively, are given by
\begin{eqnarray*}
\textstyle \mathrm{Ric}(V,W) & = & \textstyle \frac{l}4 \langle
V,W \rangle + \frac{1}{l} (  \frac{k}{2} r(V,W)  + \frac{m}2
r(V,W) ) = \frac{l +
2k + 2m}4 \langle V,W \rangle,\\
\textstyle \mathrm{Ric}(X_1,Y_1) & = & \textstyle  \frac{k}4
\langle X_1,Y_1\rangle + \frac{1}{k} (  \frac{l}{2} r(X_1,Y_1) +
\frac{m}2 r(X_1,Y_1) ) = \frac{2l + k + 2m}4 \langle X_1,Y_1
\rangle,
\\
\textstyle \mathrm{Ric}(X_2,Y_2) & = & \textstyle \frac{m}4
\langle X_2,Y_2\rangle + \frac{1}{m} (  \frac{l}{2} r(X_2,Y_2)  +
\frac{k}2 r(X_2,Y_2) ) = \frac{2l + 2k + m}4 \langle X_2,Y_2
\rangle.
 \end{eqnarray*}

Now taking the identity (\ref{idhor}) and  $\lambda = - (\lambda_1 +
\lambda_2)$ into account, we have
\begin{gather*}
\lambda = 4l - 2k - 2m =  8\mu - 6 (k+m), \quad  \lambda_1 = -2l +
4k - 2m = 6k -4\mu , \\
 \lambda_2 = -2l - 2k + 4m = 6m - 4\mu.
 \end{gather*}
 Hence $2 \mu = l + k + m$. From all of this the table given in (i) follows. In particular, if
$k = m$, the table given in (ii) is obtained.

 \vspace{2mm}

For (ii). If  $\mathcal{H}$ is $\mathrm{Hol}
(\nabla^{\Lie{U}(n)})$-irreducible, then $\mathcal{H} \subseteq
\mathcal{E}_{\lambda_1}$, where $\lambda_1$ is an  eigenvalue for
$C$. Now, by Proposition \ref{cincopuntounopro}, it follows that
$\lambda_1 = - \frac{\lambda}{2}$. Therefore, $ C = \lambda
\langle \pi_{\mathcal{V}} \, \cdot , \pi_{\mathcal{V}} \, \cdot
\rangle - \frac{\lambda}{2} \langle \pi_{\mathcal{H}} \, \cdot ,
\pi_{\mathcal{H}} \, \cdot \rangle, $ where $\pi_{\mathcal{V}}$
and $\pi_{\mathcal{H}}$ are the projections of $\mathrm{T} M $ on
$\mathcal{V}$ and $\mathcal{H}$, respectively. The same argument
is valid to claim that $r$ has only one eigenvalue $l$
 on $\mathcal{V}$ and another only one $k$ on $\mathcal{H}$,
where $l$ and $k$ are positive constants.

Applying the expression for the Ricci tensor given by
(\ref{NagyRicci}) and Lemma \ref{cincopuntotreslem}, we obtain:
\begin{eqnarray*}
\mathrm{Ric}(V,W) & = & \textstyle \frac{l}{4} \langle V,W \rangle
+\frac{k}{l}   r(V,W)  = \frac{l+4k}{4} \langle V,W \rangle, \;
\mbox{ for } V,W  \in \mathcal{V};
 \\
\mathrm{Ric}(X,Y) & = & \textstyle \frac{k}{4} \langle X,Y \rangle
+\frac{1}{k} ( \frac{l}{2} r( X,Y)     +  \frac{k}{2} r(X,Y) ) =
\frac{2l+3k}{4}   \langle X,Y \rangle, \mbox{ for } X,Y \in
\mathcal{H}.
\end{eqnarray*}
To compute the eigenvalue $\lambda$ of $C$ and the Einstein constant
$\mu$ of the base manifold, we use the identity (\ref{idhor}) and
obtain $ \textstyle 4 \mu = 4 \frac{2l+3k}{4}  + k$.  From this the
required expressions for $\mu$ and $\lambda$ follow.
\end{proof}

\begin{remark}{\rm In both cases, {\rm (i)} and {\rm (ii)} in the previous Proposition,
 when the total manifold
is Einstein,  it is satisfied that $\mu= \frac{3}{2} k$ and the
tensors $r$, $\Ric$ and $C$ have a single eigenvalue satisfying
the already known relations (Proposition \ref{lcomplete})
 given by the following table}
\begin{center}
\begin{tabular}{ccccc}
\toprule
Eigenvalue  &  $r$ & $\Ric$ &  $C$ & Eigenbundle \\[1mm]
\midrule
$\lambda $ & $k$&  $\frac{5}{4} k $&  $0$ & $\mathrm{T} M$ \\[1mm]
\bottomrule
\end{tabular}
\end{center}
\vspace{3mm}

{\rm
 Finally, note that from Equations \eqref{ricricast2}  and \eqref{ricricast},
 Proposition \ref{satdimensions}, the tables given in
 \cite[Proposition 5.2 and Remark 5.1]{Nagy2} are deduced. However, we do not require distinct eigenvalues.
  Also some details of the table for three eigenvalues must be corrected in \cite{Nagy2}. In fact, if
 $2d = \dim \mathcal{V}$,  $2d_1 = \dim \mathcal{E}$ and  $2d_2 = \dim
 \mathcal{F}$, for homogeneous NK Type III spaces one obtains

 \begin{center}
\begin{tabular}{ccccc}
\toprule
$Eigenvalue$  &  $r$ &  $\Ric$ &  $C$ & $Eigenbundle$ \\[1mm]
\midrule $\lambda $ & $\frac{d_1}{d}k$ & $\frac12 (\frac{d_1}{2d} +
1 +  \frac{d_1}{d_2})k $& $ 2  (\frac{2d_1}{d} - 1 -
\frac{d_1}{d_2})k$ & ${\mathcal{V}}$
\\[1mm]
 $\lambda_1$ & $k$ & $\frac12 (\frac{d_1}{d} +  \frac12   + \frac{d_1}{d_2})k$&
 $ 2( - \frac{d_1}{d} + 2 -\frac{d_1}{d_2})k$& $\mathcal{E}$\\[1mm]
$\lambda_2$ & $\frac{d_1}{d_2}k$ & $\frac12 (\frac{d_1}{d} +  1  +
\frac{d_1}{2d_2})k$
& $ 2( - \frac{d_1}{d} -1 + \frac{2d_1}{d_2})k $ & $\mathcal{F}$\\[1mm]
\bottomrule
\end{tabular}
\end{center} }
\end{remark}
\vspace{2mm}

\section{Riemannian 3-symmetric spaces}\indent
We recall that a connected $2n$-dimensional Riemannian manifold $(M,g)$ is called a
$3$-{\em symmetric space} \cite{G} if it admits a family of
isometries $\{\theta_{p}\}_{p\in M}$ of $(M,g)$ satisfying
\begin{enumerate}
\item[{\rm (i)}] $\theta^{3}_{p} = I,$
\item[{\rm (ii)}] $p$ is an isolated fixed point of $\theta_{p},$
\item[{\rm (iii)}] the tensor field $\Theta$ defined by $\Theta =
(\theta_{p})_{*p}$ is of class $C^{\infty},$
\item[{\rm (iv)}] $\theta_{p*}\comp J = J \comp \theta_{p*},$
\end{enumerate}
where $J$ is the {\em canonical almost complex structure} associated
with the family $\{\theta_{p}\}_{p\in M}$ given by $J =
\frac{1}{\sqrt{3}}(2\Theta +I).$

\textit{Riemannian $3$-symmetric spaces} are characterised by a
triple $(M = G/K,\sigma, \langle \cdot , \cdot\rangle)$ satisfying
the following conditions:
\begin{enumerate}
\item[{\rm (1)}] $G$ is a connected Lie group and $\sigma$ is an
automorphism of $G$ of order $3,$
 \item[{\rm (2)}] $K$ is a closed
subgroup of $G$ such that $G_{o}^{\sigma} \subseteq K\subseteq
G^{\sigma},$ where $G^{\sigma} = \{x\in G\mid \sigma(x) = x\}$ and
$G_{o}^{\sigma}$ denotes its identity component,
 \item[{\rm (3)}]
$\langle \cdot , \cdot\rangle$ is
an $\mathrm{Ad}(K)$- and
$\sigma$-invariant inner product
on ${\mathfrak m} = {\rm Ker}\;\phi,$ where $\phi$ is
the endomorphism of ${\mathfrak g}$ given by $\phi = 1 + \sigma + \sigma^{2}.$
\end{enumerate}
We shall say that $G/K$ is a \textit{$3$-symmetric coset space}
if there exists $\sigma$ satisfying (1) and (2). Here and in the
sequel, $\sigma$ and its differential $\sigma_{*}$ on ${\mathfrak
g}$ and on the complexification ${\mathfrak g}_{\mathbb C}$ of ${\mathfrak g}$ are denoted by the same
letter $\sigma.$ Then it is a reductive homogeneous space with
reductive decomposition ${\mathfrak g} = {\mathfrak m}\oplus
{\mathfrak k},$ where the Lie algebra ${\mathfrak k}$ of $K$ is
${\rm Im}\;\phi = {\mathfrak g}^{\sigma} = \{ X\in {\mathfrak
g}\mid \sigma X = X\}.$ The inner product $\langle \cdot ,
\cdot\rangle$ induces a $G$-invariant Riemannian metric $g$ on $M
= G/K$ and $(G/K,g)$ becomes into a Riemannian $3$-symmetric
space, where the symmetry $\theta_{o}$ at the origin $o\in M$ is
given by $\theta_{o}(gK) = \sigma(g)K,$ for all $g\in G.$
 Since the canonical almost complex structure
$J$ on $G/K$ is $G$-invariant, $3$-symmetric spaces are almost
Hermitian homogeneous spaces determined by the
$\mathrm{Ad}(K)$-invariant automorphism on ${\mathfrak m},$ which
we denote by the same letter, given by
\begin{equation}\label{J}
J = \frac{1}{\sqrt{3}}(2\sigma_{\mid {\mathfrak m}} +
Id_{\mathfrak m}).
\end{equation}
Under the canonical identification of ${\mathfrak m}$ with
$\mathrm{T}_{o}\, G/K,$ we have the following \cite{G}.
\begin{equation}\label{JJ}
[JX,JY]_{\mathfrak k} = [X,Y]_{\mathfrak k},\;\;\;\;\;
[JX,Y]_{\mathfrak m} = -J[X,Y]_{\mathfrak m}.
\end{equation}

In addition to the minimal $\Lie{U}(n)$-connection of the
canonical almost Hermitian structure $(J,g)$ of a Riemannian
$3$-symmetric space, we may consider the {\em canonical
connection} (see \cite{Nomizu}) of the $3$-symmetric coset space
$G/K$ adapted to the reductive decomposition ${\mathfrak g} =
{\mathfrak m}\oplus {\mathfrak k}$ and furthermore, the {\em
canonical connection of the $3$-symmetric space}, treated as a
regular $s$-manifold (see \cite{Kow}). Next we prove that these
three connections coincide.
\begin{proposition}\label{lcanonical}
The minimal connection $\nabla^{\Lie{U}(n)}$ of the
$\Lie{U}(n)$-structure determined by the canonical almost complex
structure of a Riemannian $3$-symmetric space $(M = G/K,\sigma,
\langle \cdot , \cdot \rangle)$ coincides with the canonical
connection with respect to the reductive decomposition ${\mathfrak
g} = {\mathfrak k} \oplus {\mathfrak m}.$ Moreover,
$\nabla^{\Lie{U}(n)}$ is also the canonical connection of the
regular $s$-structure.
\end{proposition}
\begin{proof}
For each $X\in {\mathfrak m},$ denote by $X^{*}$ the $G$-invariant
vector field defined on a small neighborhood of the origin $o$ of
$M = G/K$ such that $X^{*}_{o} = X,$ under the identification
${\mathfrak m} \cong T_{o}M.$ Then $[X^{*},Y^{*}]_{o} =
[X,Y]_{\mathfrak m},$ where $[X,Y]_{\mathfrak m}$ denotes the
${\mathfrak m}$-component of $[X,Y]$ for all $X,Y\in {\mathfrak
m}$ and, at the origin, we have
\[
2 \langle \nabla_{X}Y^{*},Z\rangle = -\langle X, [Y,Z]_{\mathfrak
m}\rangle - \langle Y, [X,Z]_{\mathfrak m}\rangle + \langle Z,
[X,Y]_{\mathfrak m}\rangle,
\]
for $X,Y, Z\in {\mathfrak m}.$ Hence, using (\ref{JJ}), $\xi_{X}Y
= -\nabla_{X}Y^{*}$ and so, $\nabla^{\Lie{U}(n)}_{X}Y^{*} =0.$ It
implies that $\nabla^{\Lie{U}(n)}$ is the canonical connection of
$M$ with respect to the reductive decomposition ${\mathfrak g} =
{\mathfrak k} \oplus {\mathfrak m}.$ Since any $G$-invariant
tensor field on $M$ is then parallel with respect to
$\nabla^{\Lie{U}(n)}$ \cite[Proposition I.11]{Kow}, we have
$\nabla^{\Lie{U}(n)}\Theta = 0.$ Hence, taking into account that
$\nabla^{\Lie{U}(n)}$ is invariant with respect to all
$\theta_{p},$ it follows from \cite[Theorem II.17]{Kow} that
$\nabla^{\Lie{U}(n)}$  is also the canonical connection of the
$s$-structure.
\end{proof}
\begin{remark}\label{rcanonical}
{\rm Then we have $\nabla^{\Lie{U}(n)} T^{\Lie{U}(n)} =
\nabla^{\Lie{U}(n)}R^{\Lie{U}(n)} = 0,$ where $T^{\Lie{U}(n)}$ and
$R^{\Lie{U}(n)}$ denote the torsion and the curvature tensor of
the minimal connection $\nabla^{\Lie{U}(n)}$ respectively, and its intrinsic torsion
$\xi$ is a homogeneous structure \cite{TV}.}
\end{remark}
\noindent The tensor fields $T^{\Lie{U}(n)}$ and $R^{\Lie{U}(n)}$
are $G$-invariant and they are given by
\begin{equation}\label{TR}
T^{\Lie{U}(n)}(X,Y) = -[X,Y]_{\mathfrak m},\;\;\;\;\;
R^{\Lie{U}(n)}_{X \,Y} = {\rm ad}_{[X,Y]_{\mathfrak k}}.
\end{equation}
The canonical almost complex structure of a Riemannian
$3$-symmetric space is {\em quasi-K\"ahlerian} ($\xi$ satisfies $\xi_{X}Y + \xi_{JX}JY =0,$ for all vector
fields $X,Y)$ and it is nearly K\"ahlerian if and only if
$(G/K,g)$ is a naturally reductive homogeneous space with adapted
reductive decomposition ${\mathfrak g} = {\mathfrak m} \oplus
{\mathfrak k}$ \cite{G}. In general, a homogeneous Riemannian
manifold $(M=G/K,g)$ is said to be {\em naturally reductive}, or
more precisely, $G$-{\em naturally reductive}, if there exists a
reductive decomposition ${\mathfrak g} = {\mathfrak m}\oplus
{\mathfrak k}$ satisfying
\[
\langle [X,Y]_{\mathfrak m},Z\rangle + \langle [X,Z]_{\mathfrak
m},Y\rangle = 0,
\]
for all $X,Y,Z\in {\mathfrak m},$  where $\langle \cdot , \cdot
\rangle$ is the inner product induced by $g$ on ${\mathfrak m}.$
 From Proposition \ref{lcanonical} and (\ref{TR}), the intrinsic
torsion of a naturally reductive $3$-symmetric space is determined
by
\begin{equation}\label{xi}
\xi_{X}Y = -\frac{1}{2}[X,Y]_{\mathfrak m},\;\;\; X,Y\in {\mathfrak m}.
\end{equation}
\begin{remark}
{\rm The converse of Proposition \ref{lcanonical} also holds (see
\cite{Bu1} for a proof):} A quasi-K\"ahlerian homogeneous manifold
$(M =G/K,g,J)$ such that the minimal connection coincides with the
canonical connection with respect to a given reductive
decomposition ${\mathfrak g} = {\mathfrak m} \oplus {\mathfrak k}$
is a $3$-symmetric space.
\end{remark}

Next we display a characterization result for naturally reductive
$3$-symmetric spaces which is well-known.
\begin{proposition}\label{p3-sym}
If  $(M,g,J)$ is a complete, connected and simply connected nearly
K\"ahler manifold, then the following conditions are
equivalent{\rm :}
\begin{enumerate}
\item[{\rm (i)}] $M$ is a $3$-symmetric space and $J$ is the
canonical almost complex structure{\rm ;} \item[{\rm (ii)}]
$(\nabla_{X}R)_{XJXXJX} = 0,$ for all $X\in {\mathfrak X}(M);$
\item[{\rm (iii)}] the intrinsic torsion $\xi$ is a homogeneous
structure{\rm ;} \item[{\rm (iv)}] $\nabla^{\Lie{U}(n)}R =0;$
\item[{\rm (v)}] $\nabla^{\Lie{U}(n)}R^{\Lie{U}(n)} =0;$
\item[{\rm (vi)}] $(\nabla_{X}R)_{YZ} = R_{\xi_{X}YZ} +
R_{Y\xi_{X}Z} - [\xi_{X},R_{YZ}].$
\end{enumerate}
\end{proposition}
\begin{proof}
$({\rm i})\Leftrightarrow ({\rm ii})$ is a consequence from
\cite[Theorems 4.6 and 4.7 and Corollary 3.11]{G}. $({\rm
i})\Rightarrow ({\rm iii}) $ is satisfied for any Riemannian
$3$-symmetric space (see Remark \ref{rcanonical}). Because
$\nabla^{\Lie{U}(n)} g = \nabla^{\Lie{U}(n)} \xi =0,$ we get
$({\rm iii})\Leftrightarrow ({\rm iv})\Leftrightarrow ({\rm
v})\Leftrightarrow ({\rm vi}).$ Finally, since $\xi_{X}X =0$ and
$\xi_{X}J = -J\xi_{X},$ $({\rm vi})\Rightarrow ({\rm ii}).$
\end{proof}

The group ${\mathcal H}(M)$ of holomorphic isometries of a
Riemannian $3$-symmetric space $(M,g)$ with respect to its
canonical almost complex structure acts transitively on $M$
\cite[Theorem 4.8]{G}. Therefore, $M$ can be written as the coset
space  $G^{*}/K^{*},$ where $G^{*}$ is the identity component
${\mathcal H}_{o}(M)$ of ${\mathcal H}(M)$ and $K^{*}$ is the
isotropy subgroup of a point $o$ of $M.$ The triple $(M =
G^{*}/K^{*},\sigma^{*},\langle\cdot ,\cdot \rangle),$ where
$\sigma^{*}$ is the map sending $g$ to $\theta_{o}\comp g\comp
\theta^{-1}_{o},$ for $g\in G^{*},$ and $\langle\cdot , \cdot
\rangle$ is the inner product determined by $g$ on ${\mathfrak m}
= \mathrm{T}_{o}M,$ satisfies (1), (2) and (3). Thus $G^{*}/K^{*}$
becomes into a $3$-symmetric coset space \cite[Theorem 5.4]{G}. We
may also consider the closure ${\rm Cl}(\{\theta_{p}\})$ of the
group generated by the set $\{\theta_{p}\mid p\in M\}$ in
${\mathcal H}(M).$ From Proposition \ref{lcanonical} and
\cite[Theorem II.32]{Kow}, ${\rm Cl}(\{\theta_{p}\})$ coincides
with the transvection group of the affine reductive space
$(M,\nabla^{\Lie{U}(n)}).$ It is a normal subgroup of ${\mathcal
H}(M)$ and acts transitively on $M.$ Hence, if we denote by $G$
the identity component of ${\rm Cl}(\{\theta_{p}\})$ and $K$ its
isotropy subgroup at o, we have another $3$-symmetric coset
expression $G/K$ for $M$ with corresponding automorphism $\sigma$
of order $3$ given by $\sigma = \sigma^{*}_{\mid G}.$
\begin{theorem}\label{mainh}
Let $(M = G/K,\sigma,\langle\cdot, \cdot\rangle)$ be a simply
connected  Riemannian $3$-symmetric space and $J$ its canonical
almost complex structure. Suppose that $G$ is connected semisimple
and acts effectively on the coset space $M = G/K.$ Then we have:
 \begin{enumerate}
 \item[{\rm (i)}] The holonomy group ${\rm Hol}(\nabla^{\Lie{U}(n)})$ of the minimal connection $\nabla^{\Lie{U}(n)}$ coincides with the isotropy subgroup $K,$ as acting by the adjoint representation.
\item[{\rm (ii)}] $G$ is the identity component of the
transvection group ${\rm Cl}(\{\theta_{p}\})$ of
$(M,\nabla^{\Lie{U}(n)}).$ \item[{\rm (iii)}] If the identity
component ${\mathcal H}_{o}(M)$ of the group of holomorphic
isometries of $(M,g,J)$ is semisimple,  then $G = {\mathcal
H}_{o}(M).$
\end{enumerate}
\end{theorem}
\begin{proof}
From Proposition \ref{lcanonical}, the Lie algebra ${\mathfrak
h}{\mathfrak o}{\mathfrak l}(\nabla^{\Lie{U}(n)})$ of the holonomy
group ${\rm Hol}(\nabla^{\Lie{U}(n)})$ (with reference point at
the origin $o)$ is spanned by all endomorphisms of the form
$R^{\Lie{U}(n)}_{X\,Y},$ for all $X,Y\in \mathrm{T}_{o}M.$
Moreover, under the canonical identification of ${\mathfrak m}$
with $\mathrm{T}_{o}M,$ $R^{\Lie{U}(n)}_{X\,Y}$ can be expressed
as $\mathrm{ad}_{[X,Y]_{\mathfrak k}},$ for all $X,Y\in {\mathfrak
m}$ (see Remark \ref{rcanonical}). Since the linear isotropy
representation of $K$ in the tangent space $\mathrm{T}_{o}M$ is
faithful, ${\mathfrak h}{\mathfrak o}{\mathfrak
l}(\nabla^{\Lie{U}(n)})$ coincides precisely with the subalgebra
of ${\mathfrak k}$ generated by all projections $[X,Y]_{\mathfrak
k},$ $X,Y\in {\mathfrak m}.$

The subspace ${\mathfrak n} = {\mathfrak m} + [{\mathfrak
m},{\mathfrak m}]$ is an ideal of ${\mathfrak g}.$ Since
${\mathfrak g}$ is semisimple, we can consider the orthogonal
complement ${\mathfrak a}$ of ${\mathfrak n}$ in ${\mathfrak g}$
with respect to the Killing form. Then ${\mathfrak a}$ is also an
ideal and is contained in ${\mathfrak k}.$ Since $G$ is effective
on $M,$ we conclude that ${\mathfrak a}$ reduces to $0$ and
${\mathfrak g} = {\mathfrak n}.$ Now, because $M$ is simply
connected, and $K$ and ${\rm Hol}(\nabla^{\Lie{U}(n)})$ are
connected, we get (i). From here one obtains that $G$ is the
transvection group of the affine reductive space
$(M,\nabla^{\Lie{U}(n)})$ (see \cite[Theorem I.25]{Kow}). This
proves  (ii).

Finally we shall show (iii). Denote by $G^{*}$ the identity
connected component ${\mathcal H}_{o}(M)$ of ${\mathcal H}(M)$ and
$K^{*}$ the isotropy subgroup of $G^{*}$ at $o=\{K\}.$ Since $G$
acts effectively on $M,$ $G$ is a Lie subgroup of $G^{*}$ and
$K\subset K^{*}.$ Let ${\mathfrak g},$ ${\mathfrak k},$
${\mathfrak g}^{*}$ and ${\mathfrak k}^{*}$ be the Lie algebras of
$G,$ $K,$ $G^{*}$ and $K^{*},$ respectively.

The symmetry $\theta_{o}$ at $o$ is determined by $\theta_{o}\comp
\pi = \pi\comp \sigma,$ where $\pi$ is the canonical projection
$G\to M.$ Then we have $\sigma(g) = \theta_{o}\comp g\comp
\theta_{o}^{-1},$ for all $g\in G.$ The extension $\sigma^{*}$ of
$\sigma$ to $G^{*}$ becomes $G^{*}/K^{*}$ in another $3$-symmetric
coset space of $M$ (see \cite[Proposition 5.1]{G}). If $G^{*}$ is
semisimple, then from (i) $K = K^{*}.$ In a similar way, we obtain
${\mathfrak g}^{*} = {\mathfrak m}^{*} + [{\mathfrak
m}^{*},{\mathfrak m}^{*}].$ Taking into account that the
corresponding $\phi^{*}$ for $\sigma^{*},$ $\phi^{*} = 1 +
\sigma^{*} + (\sigma^{*})^{2},$ satisfies $\phi^{*}_{\mid
{\mathfrak g}} = \phi,$ one gets ${\mathfrak m}\subset {\mathfrak
m}^{*}$ and so ${\mathfrak m} = {\mathfrak m}^{*}.$ Hence we have
${\mathfrak g} = {\mathfrak g}^{*}$.
\end{proof}

\begin{remark}\label{r4.6}
{\rm In the previous Theorem, a tensor field on $M$ is
$\nabla^{\Lie{U}(n)}$-parallel if and only if it is $G$-invariant
\cite[Proposition I.38]{Kow}.}
\end{remark}

\begin{corollary}\label{luni}
Let $(M=G/K,\sigma,\langle \cdot , \cdot \rangle)$ be a Riemannian
$3$-symmetric space and $J$ its canonical almost complex
structure. Suppose that $G$ is a connected, compact Lie group
acting effectively on $M$ and ${\rm rank}\;K = {\rm rank}\; G.$
Then we have:
\begin{enumerate}
 \item[{\rm (i)}] The holonomy group ${\rm Hol}(\nabla^{\Lie{U}(n)})$ of the minimal
 connection $\nabla^{\Lie{U}(n)}$ coincides with the isotropy subgroup $K,$ as acting by the adjoint representation;
\item[{\rm (ii)}] $G$ coincides with the identity component of the group of holomorphic isometries of $(M,g,J).$
\end{enumerate}
\end{corollary}
\begin{proof}
 $K$ is connected and $M$ is simply connected
\cite[Proposition 4.1]{WG}. Moreover, because $K$ contains the
center of $G$ and the action is effective, it follows that $G$ is
semisimple. Then (i) is proved by using Theorem \ref{mainh} (i).

Next we can consider the decomposition $M = M_{1}\times \dots
\times M_{r},$ where $M_{i} = G_{i}/K_{i},$ $G = G_{1}\times \dots
\times G_{r},$ $K=K_{1}\times \dots \times K_{r},$ being $G_{i},$
$i=1,\dots ,r,$ the simple normal subgroups of $G$ and $K_{i} =
K\cap G_{i}.$ Because ${\rm rank}\;G_{i} = {\rm rank}\; K_{i},$
$i=1,\dots, r,$ it follows that each $M_{i}$ is an irreducible
$3$-symmetric space with associated inner automorphism
$\sigma_{i},$ where $\sigma = \sigma_{1}\times \dots \times
\sigma_{r},$ and hence, $\sigma$ is also an inner automorphism.
Denote as before by $G^{*}$ the identity connected component
${\mathcal H}_{o}(M)$ of ${\mathcal H}(M),$ $K^{*}$ the isotropy
subgroup of $G^{*}$ at the origin and ${\mathfrak g}^{*}$ and
${\mathfrak k}^{*}$ the corresponding Lie algebras. Then the
extension $\sigma^{*}$ of $\sigma$ to $G^{*}$ is also  an inner
automorphism of order $3$ and so, ${\rm rank}\;G^{*} = {\rm
rank}\;K^{*}.$ Hence $G^{*}$ must be semisimple and (ii) follows
from Theorem \ref{mainh} (iii).
\end{proof}

\section{Compact irreducible 3-symmetric spaces} \indent
We shall also need some general results of complex simple Lie
algebras. See \cite{He} for more details. Let ${\mathfrak
g}_{\mathbb C}$ be a simple Lie algebra over ${\mathbb C}$ and
${\mathfrak h}_{\mathbb C}$ a Cartan subalgebra of ${\mathfrak
g}_{\mathbb C}.$ Let $\Delta$ denote the set of non-zero roots of
${\mathfrak g}_{\mathbb C}$ with respect to ${\mathfrak
h}_{\mathbb C}$ and $\Pi = \{\alpha_{1},\dots ,\alpha_{l}\}$ a
system of simple roots or a basis of $\Delta.$ Because the
restriction of the Cartan-Killing form $B$ of ${\mathfrak
g}_{\mathbb C}$ to ${\mathfrak h}_{\mathbb C} \times {\mathfrak
h}_{\mathbb C}$ is non-degenerate, there exists a unique element
$H_{\alpha}\in {\mathfrak h}_{\mathbb C}$ such that
$B(H,H_{\alpha}) = \alpha(H),$ for all $H\in {\mathfrak
h}_{\mathbb C}.$ Moreover, we have ${\mathfrak h}_{\mathbb C} =
\sum_{\alpha\in \Delta}{\mathbb C}H_{\alpha}$ and $B$ is strictly
positive definite on ${\mathfrak h}_{\mathbb{R}} = \sum_{\alpha\in
\Delta}\mathbb{R} H_{\alpha}.$ Put $\langle \alpha ,\beta \rangle
= B(H_{\alpha},H_{\beta}).$ We choose root vectors
$\{E_{\alpha}\}_{\alpha \in \Delta},$ such that for all
$\alpha,\beta\in \Delta$ we have
\begin{equation}\label{v}
\left\{
\begin{array}{lcl}
[E_{\alpha},E_{-\alpha}]= H_{\alpha},& & [H,E_{\alpha}] =
\alpha(H)E_{\alpha},\;\;\;\mbox{for}\;H\in
{\mathfrak h}_{\mathbb
C};\\[0.6pc]
[E_{\alpha},E_{\beta}] = 0 , &
&\mbox{if}\; \alpha+\beta\neq
0\;\mbox{and}\;\alpha + \beta \not\in \Delta;\\[0.6pc]
[E_{\alpha},E_{\beta}] = N_{\alpha,\beta}E_{\alpha + \beta}
, & &\mbox{if}\;\alpha + \beta
\in \Delta,
\end{array}
\right.
\end{equation}
where the constants $N_{\alpha,\beta}$ satisfy $N_{\alpha,\beta} = -N_{-\alpha,-\beta},$ $N_{\alpha,\beta} =
-N_{\beta,\alpha}.$ If $\alpha,\beta,\gamma\in \Delta$ and $\alpha + \beta +
\gamma = 0,$ then
\begin{equation}\label{**}
N_{\alpha,\beta} = N_{\beta,\gamma} = N_{\gamma,\alpha}.
\end{equation}
Moreover, given an $\alpha$-series $\beta + n\alpha$ $(p\leq n\leq
q)$ containing $\beta,$ then
\begin{equation}\label{***}
(N_{\alpha,\beta})^{2} = \frac{\textstyle q(1-p)}{\textstyle 2}
  \langle \alpha,\alpha \rangle.
\end{equation}
For this choice, if $\alpha + \beta
\neq 0$, then  $ E_{\alpha}$ and $E_{\beta}$ are orthogonal under
$B$, $B(E_{\alpha},E_{-\alpha}) = 1$ and we have the orthogonal
direct sum
\[ \textstyle
{\mathfrak g}_{\mathbb C} = {\mathfrak h}_{\mathbb C} +
\sum_{\alpha\in \Delta} {\mathbb C}E_{\alpha}.
\]
Denote by $\Delta^{+}$ the set of positive roots of $\Delta$ with
respect to some lexicographic order in $\Pi.$ Then each $\alpha\in \Delta^{+}$ may be written as
\[ \textstyle
\alpha= \sum_{i=1}^{l}n_{k}(\alpha)\alpha_{i},
\]
where $n_{i}(\alpha)\in {\mathbb Z},$ $n_{i}(\alpha)\geq 0,$ for
all $i=1,\dots ,l$. The $\mathbb{R}$-linear subspace ${\mathfrak
g}$ of ${\mathfrak g}_{\mathbb C}$ is given by
\[
{\mathfrak g} = {\mathfrak h} + \textstyle \sum_{\alpha \in
\Delta^{+}} (\mathbb{R}\; U^{0}_{\alpha} + \mathbb{R}\;
U^{1}_{\alpha})
\]
is a compact real form of ${\mathfrak g}_{\mathbb C},$ where
${\mathfrak h} = \sum_{\alpha\in \Delta}\mathbb{R}
\sqrt{-1}H_{\alpha}$ and $U^{0}_{\alpha} = E_{\alpha}-E_{-\alpha} = - U^{0}_{-\alpha}$
and $U^{1}_{\alpha} = \sqrt{-1}(E_{\alpha} + E_{-\alpha})= U^{1}_{-\alpha}.$ Next,
we put $N_{\alpha,\beta} = 0$ if $\alpha + \beta \not\in \Delta.$
Then $N_{\alpha,\beta}\neq 0$ if and only if $\alpha + \beta \in \Delta$ and, using (\ref{v}), one gets
\begin{lemma}\label{bracket} For all $\alpha ,\beta \in
\Delta^{+}$ and $a = 0,1,$ the following equalities hold{\rm :}
\begin{enumerate}
\item[{\rm (i)}]
 $[U^{a}_{\alpha},\sqrt{-1} H_{\beta}] = (-1)^{a+1}
 \langle \alpha,\beta \rangle
  U^{a+1}_{\alpha};$ \item[{\rm (ii)}]
$[U^{0}_{\alpha},U^{1}_{\alpha}] = 2\sqrt{-1}H_{\alpha};$
\item[{\rm (iii)}] $[U^{a}_{\alpha},U^{b}_{\beta}] =
(-1)^{ab}N_{\alpha,\beta}U^{a+b}_{\alpha + \beta} +
(-1)^{a+b}N_{-\alpha,\beta}U^{a+b}_{\alpha -\beta},$ where $\alpha
\neq \beta$ and $a\leq b.$
\end{enumerate}
\end{lemma}

For each $\sqrt{-1}H\in {\mathfrak h},$ it implies that
\begin{equation}\label{adjoint}
\begin{array}{lcl}
\mathrm{Ad}_{\exp\sqrt{-1}H}U^{0}_{\alpha}
& = & \cos\alpha(H) U^{0}_{\alpha} +
\sin\alpha(H) U^{1}_{\alpha},\\[0.5pc]
\mathrm{Ad}_{\exp\sqrt{-1}H}U^{1}_{\alpha}
& = & \cos\alpha(H) U^{1}_{\alpha} - \sin\alpha(H) U^{0}_{\alpha}.
\end{array}
\end{equation}

A compact irreducible
Riemannian $3$-symmetric space $(M =
G/K,\sigma, \langle \cdot ,
\cdot\rangle)$ has one of the following forms \cite{WG}:

\noindent ${\sf Type \;A_{3}:}$ $G$ is a compact connected simple
Lie group acting effectively and $\sigma$ is an inner automorphism
on the Lie algebra ${\mathfrak g}$ of $G.$

Let $\mu = \sum_{i=1}^{l}m_{i}\alpha_{i}$ be the {\em highest
root} of $\Delta$ and consider $H_{i}\in {\mathfrak h}_{\mathbb
C},$ $i=1,\dots ,l,$ defined by
\[ \alpha_{j}(H_{i}) = \frac{\textstyle 1}{\textstyle m_{i}}\delta_{ij},\;\;\; i,j = 1,\dots ,l.
\]
Note that
\begin{equation}\label{Jsym}
[\sqrt{-1}H_{i},U^{a}_{\alpha}] = (-1)^{a}\frac{n_{i}(\alpha)}{m_{i}}U^{a+1}_{\alpha}.
\end{equation}
Following \cite[Theorem 3.3]{WG}, each inner automorphism of order
$3$ on ${\mathfrak g}_{\mathbb C}$ is conjugate in the inner
automorphism group of ${\mathfrak g}_{\mathbb C}$ to some
$\sigma = \mathrm{Ad}_{\exp 2\pi\sqrt{-1}H},$
 where $H = \frac{1}{3}m_{i}H_{i}$ with $1\leq
m_{i}\leq 3$ or $H = \frac{1}{3}(H_{i} + H_{j})$ with $m_{i} = m_{j}
= 1.$ Then there are four classes of $\sigma = \mathrm{Ad}_{\exp
2\pi\sqrt{-1}H}$ with corresponding simple root systems $\Pi(H)$ for
${\mathfrak g}^{\sigma}_{\mathbb C},$ Types $A_{3}I$-$A_{3}IV$ given
in Table \ref{tab:I}. Denote by $\Delta^{+}(H)$ the positive root
system generated by $\Pi(H)$. Then we have ${\mathfrak h}\subset
{\mathfrak k} = {\mathfrak g}^{\sigma}$ and
\begin{table}[tp]
\centering
\begin{tabular}{clcc}
\toprule
$\mbox{\rm Type}$ & $\qquad \quad  \sigma$ & $m_{i}$ & $\Pi(H)$\\[1mm]
 \midrule
 $A_{3}I$ & $\mathrm{Ad}_{\exp\frac{2\pi\sqrt{-1}}{3}H_{i}}$& $1$
& $\{\alpha_{k}\in \Pi\mid k\neq i\}$\\[2mm]
$A_{3} II$ & $\mathrm{Ad}_{\exp 2\pi\sqrt{-1}\frac{(H_{i}+
H_{j})}{3}}$ & $m_{i} = m_{j} =1$
 & $\{\alpha_{k}\in \Pi\mid k\neq i,\;k\neq j\}$\\[2mm]
$A_{3} III$ & $\mathrm{Ad}_{\exp\frac{4\pi\sqrt{-1}}{3}H_{i}}$ &
$2$ & $\{\alpha_{k}\in \Pi\mid k\neq i\}$\\[2mm]
$A_{3} IV$ & $\mathrm{Ad}_{\exp 2\pi\sqrt{-1}H_{i}}$ & $3$ &
$\{\alpha_{k}\in \Pi\mid k\neq i\} \cup \{-\mu\}$\\[1mm]
\bottomrule
\end{tabular}
 \vspace{2mm}

  \caption{} 
  \label{tab:I}
\end{table}
\[ \textstyle
{\mathfrak k} = {\mathfrak h} + 
\sum_{\alpha \in
\Delta^{+}(H)}(\mathbb{R}\; U^{0}_{\alpha}+ \mathbb{R}\;
U^{1}_{\alpha}).
\]
Because $B(U^{a}_{\alpha},U^{b}_{\beta}) =
-2\delta_{\alpha\beta}\delta_{ab},$ it follows that
$\{U^{a}_{\alpha}\mid a=0,1,\; \alpha\in
\Delta^{+}\smallsetminus\Delta^{+}(H)\}$ becomes into an
orthonormal basis for $({\mathfrak m},\langle \cdot , \cdot
\rangle= -\frac{1}{2}B_{\mid {\mathfrak m}}).$ \vspace{0.1cm}

\noindent ${\sf Type\; B_{3}:}$ $G$ is a compact simple Lie group
and the complexification $\mathfrak{g}_\mathbb{C}$ of $\mathfrak{g}$
is of Dynkin type
$\mathfrak{d}_{4}$ and $\sigma$ is an outer automorphism on
$\mathfrak{g}.$

\noindent ${\sf Type\; C_{3}:}$ $G = L\times L\times L,$ where $L$
is a compact simple Lie group and $\sigma$ on ${\mathfrak g} =
{\mathfrak l}\oplus {\mathfrak l}\oplus {\mathfrak l}$ is given by
$\sigma(X,Y,Z) = (Z,X,Y).$ Here ${\mathfrak k} = {\mathfrak
g}^{\sigma}=\{(X,X,X)\in {\mathfrak g}\mid X\in {\mathfrak l}\}.$

\begin{lemma}\label{linner} If $\sigma$ is of Type $A_{3},$ then $J$ defined as in
{\rm (\ref{J})} satisfies
\[
JU^{0}_{\alpha} = \varepsilon(\alpha) U^{1}_{\alpha} ,\;\;\;\; JU^{1}_{\alpha}
= -\varepsilon(\alpha) U^{0}_{\alpha},
\]
for all $\alpha\in\Delta^{+}\smallsetminus\Delta^{+}(H),$ where $\varepsilon(\alpha) = \pm 1.$
\end{lemma}
\begin{proof}
Let $\sqrt{-1}H\in {\mathfrak h}$ such that $\sigma =
\mathrm{Ad}_{\exp 2\pi\sqrt{-1}H}$ as before. For each $\alpha\in\Delta^{+}\smallsetminus\Delta^{+}(H)$
we have
$\alpha(H) = \frac{1}{3} n_i(\alpha),$ if $H=\frac{1}{3}m_{i}H_{i},$
$1\leq m_{i}\leq 3,$ and $\alpha(H)  = \frac{1}{3}(n_{i}(\alpha) +
n_{j}(\alpha)),$ if $H =\frac{1}{3}(H_{i} + H_{j}),$ $m_{i} = m_{j} =1.$
Hence it follows that the possible values for $\alpha(H)$ are
$\frac{1}{3},$ $\frac{2}{3}$ and $1.$ But $\alpha(H)\neq 1,$
because if $\alpha(H) =1$ one obtains from (\ref{adjoint}) that
$\sigma(U^{0}_{\alpha}) = U^{0}_{\alpha}.$ Then the result follows
directly using (\ref{J}).
 \end{proof}

Next we give the classification of all compact irreducible simply
connected $3$-symmetric spaces.
\begin{theorem}\label{tquotient}
If $(M,g)$ is a compact irreducible simply connected $3$-symmetric
space,  then there exists a unique $3$-symmetric quotient
expression $G/K$ for $M$ where $G$ is a compact connected Lie
group acting effectively on $M.$ In particular, $G$ is the
identity component of the group of holomorphic isometries of
$(M,g)$ with respect to the canonical complex structure and $K$ is
the holonomy group of the minimal connection. The pairs $(G,K)$
are given in Tables {\rm  \ref{tab:AI}-\ref{tab:AIII}}.
\end{theorem}
\begin{proof}
Let $(M = G/K,\sigma,\langle\cdot,\cdot\rangle)$ be a compact
irreducible simply connected $3$-symmetric space. If $\sigma$ is
of Type $A_{3},$ the result follows from Corollary \ref{luni}. If
$\sigma$ is of Type $B_{3}$ or $C_{3},$  Tojo proves in
\cite{Tojo} that the identity component of the isometry group
coincides with $G$ and so, $G = {\mathcal H}_{o}(M)$. For the lists of the pairs $(G,K),$ see \cite[Section 6]{WG}.
\end{proof}

\begin{remark}
{\rm If $(M,g)$ is not isometric to a symmetric space,
$G$ must be the identity component of the isometry group
\cite{Tojo}.

In Tables, $G/{\mathbb Z}_{n}$ denotes the quotient of $G$ by a
central cyclic subgroup of order $n.$ Irreducible $3$-symmetric
spaces of Type $A_{3}I$ are Hermitian symmetric spaces (Table
\ref{tab:AI}). Another important class of homogeneous spaces is
that of {\em flag manifolds}. These are manifolds of the form
$G/K,$ where $G$ is a compact Lie group and $K$ is the centralizer
of a torus of $G.$ Then irreducible $3$-symmetric spaces of Type
$A_{3}I,$ together with those of Type $A_{3}II$ (Table
\ref{tab:AII}) and of Type $A_{3}III$ (Table \ref{tab:AIII}) are
examples of flag manifolds and moreover, because $G$ is
semi-simple -in fact, simple- they are all {\em compact
K\"ahlerian $G$-spaces}. Although ${\rm rank}\;G = {\rm rank}\;K$
also for irreducible $3$-symmetric spaces of Type $A_{3}IV$ (Table
\ref{tab:AIV}), the isotropy subgroup $K$ is not the centralizer
of a torus. Here $K$ is semi-simple and its center has order $3$
\cite{WG}. }
\end{remark}

\begin{table}[tp]
\centering
\begin{tabular}{cc}
\toprule
 $G$ & $K$\\[1mm]
 \midrule
$\Lie{SU}(n)/{\mathbb Z}_{n},\; n\geq 2$ & $\Lie{S}(\Lie{U}(r)\times \Lie{U}(n-r))/{\mathbb Z}_{n}$\\[2mm]
$\Lie{SO}(2n+1),\; n\geq 2$ & $\Lie{SO}(2n-1)\times \Lie{SO}(2)$\\[2mm]
$\Lie{Sp}(n)/{\mathbb Z}_{2},\; n\geq 3$ & $\Lie{U}(n)/{\mathbb Z}_{2}$\\[2mm]
$\Lie{SO}(2n)/{\mathbb Z}_{2}, \; n\geq 4$ & $\{\Lie{SO}(2(n-1))\times \Lie{SO}(2)\}/{\mathbb Z}_{2}$\\[2mm]
$\Lie{SO}(2n)/{\mathbb Z}_{2},\; n\geq 4$ & $\Lie{U}(n)/{\mathbb Z}_{2}$\\[2mm]
$\Lie{E}_{6}/{\mathbb Z}_{3}$ & $\{\Lie{SO}(10)\times \Lie{SO}(2)\}/{\mathbb Z}_{2}$\\[2mm]
$\Lie{E}_{7}/{\mathbb Z}_{2}$ & $\{\Lie{E}_{6}\times \Lie{T}^{1}\}/{\mathbb Z}_{3}$\\[1mm]
\bottomrule
\end{tabular}
 \vspace{2mm}

  \caption{Type $A_{3}I$. Hermitian symmetric
  spaces}
  \label{tab:AI}
\end{table}

\begin{table}[tp]
\centering
\begin{tabular}{cc}
\toprule
$G$ & $K$ \\[1mm]\midrule
$\Lie{G}_{2}$ & $\Lie{SU}(3)$\\[2mm]
$\Lie{F}_{4}$ & $\{\Lie{SU}(3)\times \Lie{SU}(3)\}/{\mathbb Z}_{3}$\\[2mm]
$\Lie{E}_{6}/{\mathbb Z}_{3}$ & $\{\Lie{SU}(3)\times
\Lie{SU}(3)\times
 \Lie{SU}(3)\}/\{{\mathbb Z}_{3}\times {\mathbb Z}_{3}\}$\\[2mm]
$\Lie{E}_{7}/{\mathbb Z}_{2}$ & $\{\Lie{SU}(3)\times [\Lie{SU}(6)/{\mathbb Z}_{2}]\}/{\mathbb Z}_{3}$\\[2mm]
$\Lie{E}_{8}$ & $\{\Lie{SU}(3)\times \Lie{E}_{6}\}/{\mathbb Z}_{3}$\\[2mm]
$\Lie{E}_{8}$ & $\Lie{SU}(9)/{\mathbb Z}_{3}$\\[1mm]\bottomrule
\end{tabular}
 \vspace{2mm}

  \caption{Type $A_{3}IV$ }
  \label{tab:AIV}
\end{table}

\begin{table}[tp]
\centering
\begin{tabular}{cc}
\toprule
$G$ & $K$ \\[1mm]
\midrule
$\Lie{Spin}(8)$ & $\Lie{SU}(3)/{\mathbb Z}_{3}$\\[2mm]
$\Lie{Spin}(8)$ & $\Lie{G}_{2}$\\
\bottomrule \midrule
$\{L\times L\times L\}/Z(L)$ & $L/Z(L)$\\[1mm]
$\mbox{ $Z(L)$ is the center of $L$ embedded diagonally}$& \\[1mm]
\bottomrule
\end{tabular}
\vspace{2mm}

  \caption{Types $B_{3}$ and $C_{3}$  }
  \label{tab:VI}
\end{table}

\section{Proof of Theorem \ref{tmaincla}}\indent
\label{seven} We shall need some previous results. First, recall
that $(M,g,J)$  in Theorem \ref{tmaincla} is a compact naturally
reductive $3$-symmetric space equipped with its canonical complex
structure. So, in what follows, we always focus on simply
connected irreducible $3$-symmetric spaces $(M =
G/K,\sigma,\langle\cdot,\cdot\rangle)$ as in Theorem
\ref{tquotient} which are not of Type $A_{3}I.$ Moreover, the
homogeneous Riemannian manifold $(M = G/K,g)$ is supposed to be
naturally reductive with adapted reductive decomposition
${\mathfrak g} = {\mathfrak m}\oplus {\mathfrak k},$ or
equivalently, the canonical Hermitian structure $(J,g)$ is
(strict) nearly K\"ahler. According with \cite{GM}, the inner
product $\langle\cdot ,\cdot \rangle$ is then the restriction to
${\mathfrak m}$ of a bi-invariant product on ${\mathfrak g}.$
Because ${\mathfrak g}$ is semisimple, we take $\langle\cdot
,\cdot \rangle = -\frac{1}{2}B_{\mid{\mathfrak m}},$ where $B$
denotes the Killing form of ${\mathfrak g}.$

The differential of the {\em isotropy action}, i.e. the action of
$K$ as a subgroup of $G,$ determines a linear representation of
$K$ on the tangent space $\mathrm{T}_{o}M$ at the origin $o,$
called the {\em isotropy representation} which, under the
identification $\mathrm{T}_{o}M\cong {\mathfrak m},$ corresponds
with the adjoint representation $\mathrm{Ad}(K)$ of $K$ on
${\mathfrak m}.$ Using Theorem \ref{tquotient}, the isotropy
subgroup $K$ is precisely the holonomy group of the minimal
connection $\nabla^{\Lie{U}(n)}$. Hence the holonomy
representation of $\nabla^{\Lie{U}(n)}$ coincides with the
isotropy representation.

\begin{proposition}\label{I,II} A compact irreducible Riemannian
$3$-symmetric space $(M= G/K,\sigma, \langle\cdot ,\cdot \rangle)$
of Type $A_{3}IV,$ $B_{3}$ or $C_{3}$ is complex
isotropy-irreducible. Moreover,
the isotropy representation is real irreducible, i.e. $(M,g,J)$ is homogeneous NK Type I if and only if
$(M= G/K,\sigma, \langle\cdot ,\cdot \rangle)$ is of Type
$A_{3}IV$ or the universal covering of $M$ is holomorphically
isometric to $\Lie{Spin}(8)/[\Lie{SU}(3)/{\mathbb Z}_{3}].$
\end{proposition}
\begin{proof}
On irreducible $3$-symmetric spaces of Type $A_{3}IV,$ the Lie
subalgebra ${\mathfrak k}$ of ${\mathfrak g}$ is maximal. Then
they are isotropy-irreducible spaces (see \cite[Ch. 7]{Be}). Also
$\Lie{Spin}(8)/[\Lie{SU}(3)/{\mathbb Z}_{3}]$ is isotropy
irreducible by using a Wolf's result (see \cite[Proposition
7.49]{Be}).

Next we consider the $3$-symmetric space
$\Lie{Spin}(8)/\Lie{G}_{2}$ of Type $B_{3}.$ Let ${\mathfrak
g}_{\mathbb C}$ be the complex Lie algebra of type
$$
\xymatrix@R=-0.3cm@C=1.cm{
& & \stackrel{1}{\stackrel{\circ}{\alpha_{3}}} \ar@{-}[dl]\\
{\mathfrak d}_{4}: \;\;\;\;
\stackrel{1}{\stackrel{\circ}{\alpha_{1}}} \ar@{-}[r] &
\stackrel{2}{\stackrel{\circ}{\alpha_{2}}} \ar@{-}[dr]\\
 & & \stackrel{1}{\stackrel{\circ}{\alpha_{4}}}}
$$
A set $\Delta^{+}$ of the positive roots is given by
$$
\begin{array}{lcl}
\Delta^{+}  &= &  \{\alpha_{i}\;(1\leq i\leq 4), \alpha_{1} +
\alpha_{2}, \alpha_{2} +\alpha_{3}, \alpha_{2} + \alpha_{4},
\alpha_{1} + \alpha_{2} + \alpha_{3},\\[0.4pc]
 & & \hspace{0.2cm} \alpha_{1}+\alpha_{2}+\alpha_{4},  \alpha_{2} + \alpha_{3} + \alpha_{4},
 \alpha_{1}+ \alpha_{2} + \alpha_{3} + \alpha_{4},\mu =
\alpha_{1} + 2\alpha_{2} + \alpha_{3} + \alpha_{4}\}.
\end{array}
$$
Let $s:\Pi\to \Pi$ be the symmetry on
$\Pi=\{\alpha_{1},\dots,\alpha_{4}\}$ given by $ \alpha_{2}\mapsto
\alpha_{2},$ $\alpha_{1}\mapsto \alpha_{3} \mapsto
\alpha_{4}\mapsto \alpha_{1}$ and $\sigma$ the linear
transformation of ${\mathfrak d}_{4}$ defined by
$\sigma(H_{\alpha}) = H_{s(\alpha)},$ $\sigma(E_{\lambda}) =
E_{s(\lambda)},$ where $\{H_{\alpha}, E_{\lambda}: \alpha\in
\Pi,\; \lambda\in  \Delta\}$ is a Weyl basis of ${\mathfrak
d}_{4}$ and $s$ is extended to $\Delta$ by linearity. Then the set
of fixed points of $\sigma$ on ${\mathfrak d}_{4}$ is a complex
Lie algebra of type ${\mathfrak g}_{2},$ where a Weyl basis is
given by
$$
\begin{array}{l}
\{H_{\alpha_{2}}, H_{\alpha_{1}} +H_{\alpha_{3}} + H_{\alpha_{4}};
E_{\pm \alpha_{2}}, E_{\pm (\alpha_{1} + \alpha_{2} +\alpha_{3}
 + \alpha_{4})}, E_{\pm (\alpha_{1} + 2\alpha_{2} +\alpha_{3}
 + \alpha_{4})},\\[0.4pc]
 \hspace{0.2cm} E_{\pm \alpha_{1}} + E_{\pm \alpha_{3}} + E_{\pm
 \alpha_{4}}, E_{\pm (\alpha_{1}+\alpha_{2})} + E_{\pm
 (\alpha_{2}+\alpha_{3})}+ E_{\pm (\alpha_{2}+\alpha_{4})},\\[0.4pc]
 \hspace{0.2cm} E_{\pm (\alpha_{1} + \alpha_{2}
 +\alpha_{3})} + E_{\pm (\alpha_{2} +\alpha_{3} + \alpha_{4})} + E_{\pm (\alpha_{1} + \alpha_{2} +
 \alpha_{4})}\}.
 \end{array}
 $$
Hence the corresponding real form ${\mathfrak k} = {\mathfrak
g}^{\sigma}$ is generated by
$$
\begin{array}{l}
\{\sqrt{-1}H_{\alpha_{2}}, \sqrt{-1}(H_{\alpha_{1}}
+H_{\alpha_{3}} + H_{\alpha_{4}}); U_{\alpha_{2}}^{a}, U_{\alpha_{1} + \alpha_{2} +\alpha_{3} +
\alpha_{4}}^{a}, U_{\alpha_{1} + 2\alpha_{2} +\alpha_{3}
 + \alpha_{4}}^{a},\\[0.4pc]
\hspace{0.2cm}  U_{\alpha_{1}}^{a} + U_{\alpha_{3}}^{a} + U_{\alpha_{4}}^{a}, U_{\alpha_{1}+\alpha_{2}}^{a} +
U_{\alpha_{2}+\alpha_{3}}^{a} + U_{\alpha_{2}+\alpha_{4}}^{a},
U_{\alpha_{1} + \alpha_{2}
 +\alpha_{3}}^{a} + U_{\alpha_{2} +\alpha_{3} + \alpha_{4}}^{a} + U_{\alpha_{1} + \alpha_{2} +
 \alpha_{4}}^{a}\}_{a=0,1}
 \end{array}
 $$
and on ${\mathfrak m} = ({\mathfrak g}^{\sigma})^{\bot},$ an
orthonormal basis with respect to $\langle \cdot , \cdot \rangle=
-\frac{1}{2}B_{\mid {\mathfrak m}}$ is given by
\[
\{h,e^{a}_{i};\; Jh,Je_{i}^{a}\mid i=1,2,3; \; a = 0,1\},
\]
where
$$
\begin{array}{lclclcl}
h &\hspace{-2mm} = \hspace{-2mm}& \sqrt{-6}(H_{\alpha_{1}} - H_{\alpha_{3}}), & & Jh &\hspace{-2mm} = \hspace{-2mm}& \sqrt{-2}(H_{\alpha_{1}} +H_{\alpha_{3}} - 2H_{\alpha_{4}}),\\[0.4pc]
e_{1}^{a} &\hspace{-2mm} =\hspace{-2mm} & \frac{1}{\sqrt{2}}(U_{\alpha_{1}}^{a} - U_{\alpha_{3}}^{a}), & & Je_{1}^{a} & \hspace{-2mm}= \hspace{-2mm}& \frac{1}{\sqrt{6}}(U_{\alpha_{1}}^{a} + U_{\alpha_{3}}^{a} -2 U_{\alpha_{4}}^{a}),\\[0.4pc]
e_{2}^{a} &\hspace{-2mm} =\hspace{-2mm} & \frac{1}{\sqrt{2}}(U_{\alpha_{1}+\alpha_{2}}^{a} -
U_{\alpha_{2}+\alpha_{3}}^{a}), & & Je_{2}^{a} &\hspace{-2mm} =\hspace{-2mm} & \frac{1}{\sqrt{6}}(U_{\alpha_{1} + \alpha_{2}}^{a} + U_{\alpha_{2}
+\alpha_{3}}^{a} -2 U_{\alpha_{2} + \alpha_{4}}^{a}),\\[0.4pc]
e^{a}_{3} & \hspace{-2mm}=\hspace{-2mm} & \frac{1}{\sqrt{2}}(U_{\alpha_{2} + \alpha_{3}
 +\alpha_{4}}^{a} - U_{\alpha_{1} +\alpha_{2} + \alpha_{4}}^{a}), & & Je_{3}^{a} &\hspace{-2mm} =\hspace{-2mm} & \frac{1}{\sqrt{6}}(U_{\alpha_{1} + \alpha_{2} +  \alpha_{4}}^{a} + U_{\alpha_{2} + \alpha_{3}+\alpha_{4}}^{a} - 2U_{\alpha_{1} +
\alpha_{2} + \alpha_{3}}^{a}).
\end{array}
$$
Put ${\mathcal E} = \mathbb{R}\{h; e^{a}_{i}\mid i=1,2,3;\; a =
0,1\}.$ Using Lemma \ref{bracket}, one can check that ${\mathcal E}$
is $\mathrm{Ad}(\Lie{G}_{2})$-invariant and ${\mathfrak m} =
{\mathcal E} \oplus J{\mathcal E}$  is an orthonormal decomposition
into invariant and irreducible subspaces under the isotropy
representation of $G_{2}.$ From here, $\Lie{Spin}(8)/\Lie{G}_{2}$ is
complex isotropy-irreducible but not real isotropy-irreducible.

Finally, let $(G/K,\sigma,\langle\cdot,\cdot\rangle)$ be a compact
$3$-symmetric space of Type $C_{3},$ i.e. $G = \{L\times L\times
L\}/Z(L),$ where $L$ is a compact simple Lie group and simply
connected, $Z(L)$ is its center embedded diagonally and $\sigma$
on ${\mathfrak g} = {\mathfrak l}\oplus {\mathfrak l}\oplus
{\mathfrak l}$ is given by $\sigma(X,Y,Z) = (Z,X,Y),$ where $X,$
$Y,$  $Z$ belongs to the Lie algebra ${\mathfrak l}$ of $L.$ Here
${\mathfrak k} = {\mathfrak g}^{\sigma}$ coincides with
${\mathfrak l}$ embedded diagonally. Let $\Delta^{+}$ be a system
of positive roots of ${\mathfrak l}_{\mathbb C}.$ Then an orthonormal
basis on ${\mathfrak m} = ({\mathfrak g}^{\sigma})^{\bot}$ with respect to $\langle \cdot , \cdot \rangle=
-\frac{1}{2}B_{\mid {\mathfrak m}}$ is given by
\[
\{h_{\alpha},e^{a}_{\alpha};\; Jh_{\alpha},Je_{\alpha}^{a}\mid \alpha\in \Delta^{+}; \; a = 0,1\},
\]
where
$$
\begin{array}{lclclcl}
h_{\alpha} & = & \frac{\sqrt{-1}}{\sqrt{\langle\alpha,\alpha\rangle}}
(H_{\alpha},-H_{\alpha},0), & & Jh_{\alpha} & = &
\frac{\sqrt{-1}}{\sqrt{3\langle\alpha,\alpha\rangle}}(H_{\alpha},H_{\alpha},-2H_{\alpha}),\\[0.4pc]
e^{a}_{\alpha} & = & \frac{1}{\sqrt{2}}(U^{a}_{\alpha}, - U^{a}_{\alpha},0),
& & Je^{a}_{\alpha} & = & \frac{1}{\sqrt{6}}(U^{a}_{\alpha},U^{a}_{\alpha}, -2U^{a}_{\alpha}).
\end{array}
$$
Hence we have the orthogonal decomposition ${\mathfrak m} =
{\mathcal E} \oplus J {\mathcal E},$ where ${\mathcal E} =
\mathbb{R}\{h_{\alpha}; e^{a}_{\alpha}\}$. ${\mathcal E}$ and
$J{\mathcal E}$ are $\mathrm{Ad}(\Delta G)$-invariant subspaces
naturally isomorphic as vector spaces with the Lie algebra
${\mathfrak l}.$ Because ${\mathfrak l}$ is simple, it follows that
these subspaces are irreducible under the isotropy representation.
\end{proof}

\begin{proposition}\label{III,IV} If $(M = G/K,\sigma, \langle\cdot,\cdot\rangle)$
is an irreducible $3$-symmetric space of Type $A_{3}II$ or Type
$A_{3}III,$ then $(M,g,J)$ has special algebraic torsion.
Moreover, if it is of Type $A_{3}II,$ then $(M,g,J)$ is
homogeneous NK Type III.
\end{proposition}
\begin{proof}
If $\sigma$ is of Type $A_{3}II$ then $m_{i} = m_{j} = 1$ $(H =
\frac{1}{3}(H_{i} + H_{j})),$ for some $i,j\in \{1,\dots ,l\},$
$i<j.$ Put
\[
\Delta_{p,q} = \{\alpha\in \Delta\mid n_{i}(\alpha) = p,\;n_{j}(\alpha) = q\},\;\;\;\; 0\leq p,q\leq 1.
\]
Then $\Delta(H) = \Delta_{0,0}.$  Consider the subspaces of
${\mathfrak m}:$
\begin{equation}\label{VV1}
\begin{array}{lcl}
{\mathcal V}_{1} & = & \mathbb{R}\{U^{a}_{\alpha}\mid \;\alpha\in
\Delta^{+}_{1,1};\;a=0,1\},\\[0.4pc]
{\mathcal V}_{2} & = & \mathbb{R}\{U^{a}_{\alpha}\mid \;\alpha\in \Delta^{+}_{1,0};\;a = 0,1\},\\[0.4pc]
{\mathcal V}_{3} & = & \mathbb{R}\{U^{a}_{\alpha}\mid \;\alpha\in \Delta^{+}_{0,1};\; a=0,1\}.
\end{array}
\end{equation}
Then ${\mathfrak m} = {\mathcal V}_{1}\oplus {\mathcal
V}_{2}\oplus {\mathcal V}_{3}$  is an orthogonal decomposition
which is stable by $J.$ Concretely, from Lemma \ref{linner} one
gets $JU^{a}_{\alpha} = (-1)^{a+1}U^{a+1}_{\alpha},$ for all
$\alpha\in \Delta^{+}_{1,1}$, and $JU^{a}_{\alpha} =
(-1)^{a}U^{a+1}_{\alpha},$ for all $\alpha\in \Delta^{+}_{1,0}\cup
\Delta^{+}_{0,1}.$  We shall show that these subspaces are
$\mathrm{Ad}(K)$-invariant. Since $K$ is connected, it is
equivalent to show that
\begin{equation}\label{corch}
[{\mathfrak k}, {\mathcal V}_{k}]\subset {\mathcal V}_{k},\;\;\;\; k=1,2,3.
\end{equation}
Let $\alpha\in \Delta^{+}(H)$ and $\beta\in \Delta^{+}\setminus
\Delta^{+}(H).$  If $\alpha\pm \beta\in \Delta,$ then
$n_{i}(\alpha\pm \beta) = \pm n_{i}(\beta)$ and $n_{j}(\alpha\pm
\beta) = \pm n_{j}(\beta).$ Hence, using Lemma \ref{bracket}, we
obtain (\ref{corch}).

If $\alpha,\beta\in \Delta^{+}_{p,q},$ for $0\leq p,q\leq 1$ with
$(p,q)\neq (0,0),$ then $\alpha + \beta\notin \Delta$, and if
$\alpha -\beta\in \Delta$  then $\alpha -\beta \in \Delta(H).$
Hence $[{\mathcal V}_{k},{\mathcal V}_{k}]\subset {\mathfrak k},$
$k=1,2,3.$ Using again Lemma \ref{bracket}, we get $[{\mathcal
V}_{k},{\mathcal V}_{r}]\subset {\mathcal V}_{s},$ where $(k,r,s)$
is a permutation of $(1,2,3).$ Now using (\ref{xi}) we can
conclude
\[
\xi_{{\mathcal V}_{k}}{{\mathcal V}_{k}} = 0,\;\;\;\; \xi_{{\mathcal V}_{k}}{\mathcal V}_{r} \subset {\mathcal V}_{s}.
\]
\noindent It means that $(M = G/K,g,J)$ is homogeneous NK Type III
(see Proposition \ref{trespuntouno11probis} (i)).

Next suppose $\sigma$ is of Type $A_{3}III.$ Then $m_{i} = 2$ $(H
= \frac{2}{3}H_{i})$, for some $i = 1,\dots ,l.$ Put
\[
\Delta_{p}^{+} = \{\alpha\in \Delta^{+}\mid \; n_{i}(\alpha) = p\},\;\;\;\;p = 0,1,2.
\]
Then $\Delta(H) = \Delta_{0}.$ Consider the subspaces of ${\mathfrak m}:$
\begin{equation}\label{VV2}
{\mathcal H}  =  \mathbb{R}\{U^{a}_{\alpha}\mid \; \alpha\in \Delta^{+}_{1},\;;a =0,1\},\;\;\;\; {\mathcal V}  =  \mathbb{R}\{U^{a}_{\alpha}\mid \; \alpha\in
\Delta^{+}_{2};\; a =0,1\}.
\end{equation}
Then ${\mathfrak m} = {\mathcal V}\oplus {\mathcal H}$ is an
orthogonal decomposition stable for $J:$ One gets $JU^{a}_{\alpha}
= (-1)^{a}U^{a+1}_{\alpha}$,  for $\alpha\in \Delta^{+}_{1}$, and
$JU^{a}_{\alpha} = (-1)^{a+1}U^{a+1}_{\alpha}$,  for $\alpha\in
\Delta^{+}_{2}.$ A similar proof as before shows that such a
decomposition is $\mathrm{Ad}(K)$-invariant. Moreover, if
$\alpha,\beta\in \Delta^{+}_{1}$ and $\alpha + \beta\in \Delta,$
then $\alpha + \beta\in \Delta_{2}$,  and if $\alpha - \beta\in
\Delta,$ then $\alpha - \beta\in \Delta(H).$ Hence  Lemma
\ref{bracket} implies $[{\mathcal H}, {\mathcal H}]_{\mathfrak m}
\subset {\mathcal V}.$ Finally, if $\alpha,\beta\in
\Delta^{+}_{2}$ one gets that $\alpha + \beta\notin \Delta$ and if
$\alpha - \beta\in \Delta,$ then $\alpha - \beta\in \Delta(H).$ It
gives $[{\mathcal V},{\mathcal V}] \subset {\mathfrak k}.$ From
here $\xi_{\mathcal V}{\mathcal V}=0$ and $\xi_{\mathcal
H}{\mathcal H}\subset {\mathcal V}.$ It proves the result.
 \end{proof}

\begin{remark}{\rm Systems of roots of each complex simple Lie algebras are explicitly given in the last section.
See also \cite{He}}.
\end{remark}

Now we need to recall the following basic results. As in Section
$5,$ we consider  a simple Lie algebra ${\mathfrak g}_{\mathbb C}$
over ${\mathbb C}$ and $\pi = \{\alpha_{1},\dots ,\alpha_{l}\}$ a
system of simple roots with respect to a Cartan subalgebra. Let
$\mu = \sum_{i=1}^{l}m_{i}\alpha_{i}$ be the highest root in the
set of non-zero roots $\Delta$ generated by $\pi.$

\begin{lemma} \label{scalarproduct} If $\alpha_p,\alpha_q\in \pi$ are adjacent in the Dynkin diagram of ${\mathfrak g}_{\mathbb C},$ connected  by $n_{p,q}$ arcs and  $\alpha_p$ is
of higher or equal length to  $\alpha_q,$ then
$$
\langle \alpha_p , \alpha_q \rangle = - \frac{1}2 \langle \alpha_p ,
\alpha_p \rangle =  - \frac{n_{p,q}}2 \langle \alpha_q , \alpha_q
\rangle.
$$
\end{lemma}
\begin{proof}
It follows taking into account that $
n_{p,q} =  \frac{4\langle \alpha_{p} , \alpha_{q} \rangle^2}{\langle
\alpha_{p} , \alpha_{p} \rangle \langle \alpha_{q} , \alpha_{q}
\rangle},$ $\langle \alpha_{p} ,
\alpha_{p} \rangle = n_{p,q} \langle\alpha_{q} , \alpha_{q}
\rangle$ and the scalar product of
two adjacent simple roots is negative.
\end{proof}

\begin{lemma} \label{scalproductbeta} Suppose that on ${\mathfrak g}_{\mathbb C},$ $m_{p}=2,$ for some $p\in \{1,\dots ,l\},$ and there exists $\beta = \sum_{i=1}^l n_i(\beta) \alpha_i\in \Delta^{+}$ such that $n_p(\beta)= 1.$ We have:
\begin{enumerate}
 \item[{\rm (i)}] if  $\langle \beta , \alpha_p \rangle = {\sf k}$, then $\beta =
 \alpha_p$;
 \item[{\rm (ii)}] if  $\beta \neq \alpha_p$, then $\langle \beta , \alpha_p \rangle = 0$,
 or $\langle \beta , \alpha_p \rangle = \frac{\sf k}2$, or $\langle \beta , \alpha_p \rangle = - \frac{\sf k}{2},$
 \end{enumerate}
 where ${\sf k} = \langle \alpha_{p},\alpha_{p}\rangle.$
\end{lemma}
\begin{proof}
As the weight of  $\alpha_p$ in the highest root is $2$, then using
the terminology in \cite[Lemma 3.3]{He} we have
$$
- a_{\beta , \alpha_p} = -  \frac{2 \langle \beta, \alpha_p
\rangle}{\langle \alpha_p , \alpha_p \rangle} \leq 1.
$$
 If $\alpha_p$ is of length  higher or equal  to its adjacent roots,
 then $\alpha_p$ has at the  most three adjacent simple
roots $\alpha_i$, $\alpha_j$ and $\alpha_k$. Therefore, taking
Lemma \ref{scalarproduct} into account, we obtain
$$
n_i(\beta) + n_j(\beta) + n_k(\beta) - 2 \leq 1.
$$
Thus, if $\beta \neq \alpha_p$, then we have the following
alternatives:
\begin{enumerate}
 \item[-$\;$ ] $n_i (\beta) =1$ and $n_j(\beta)=n_k(\beta)=0$, then $\langle \beta, \alpha_p
\rangle = \frac{\sf k}{2}$;
 \item[-$\;$ ] $n_i (\beta) =2$ and $n_j(\beta)=n_k(\beta)=0$, then $\langle \beta, \alpha_p \rangle = 0$;
  \item[-$\;$ ] $n_i(\beta) =2$, $n_j (\beta) = 1$ and $n_k(\beta)=0$, then $\langle \beta, \alpha_p
\rangle = -\frac{\sf k}{2}$;
 \item[-$\;$ ] $n_i(\beta) =3$ and $n_k(\beta)=0$, then $\langle \beta, \alpha_p
\rangle = -\frac{\sf k}{2}$.
\end{enumerate}
Hence if $\langle \beta , \alpha_p \rangle = {\sf k}$, then
necessarily $\beta = \alpha_p$. \vspace{1mm}

 If $\alpha_p$ is of length shorter than   some adjacent root, then there are two possibilities:
${\mathfrak g}_{\mathbb C} = \mathfrak{b}_l$ and $\alpha_p
=\alpha_l$, or ${\mathfrak g}_{\mathbb C} = \mathfrak{c}_l$ and
$\alpha_p =\alpha_{l-1}$.

If ${\mathfrak g}_{\mathbb C} = \mathfrak{b}_l$ and $\alpha_p
=\alpha_l$, then $\alpha_l$ has only one adjacent simple root
$\alpha_{l-1}$. Therefore, we obtain $ 2n_{l-1} (\beta) - 2 \leq
1.$ Thus, if $\beta \neq \alpha_{l}$, then $n_{l-1}(\beta) = 1$
and $\langle \beta , \alpha_l \rangle  = 0$. Hence, if we have a positive root $\beta = \dots +
\alpha_l$, such that $\langle \beta , \alpha_l \rangle = {\sf k}$,
then necessarily $\beta = \alpha_l$.
 \vspace{1mm}

If ${\mathfrak g}_{\mathbb C} = \mathfrak{c}_l$ and $\alpha_p
=\alpha_{l-1}$, then $\alpha_{l-1}$ has two adjacent simple roots
$\alpha_{l-2}$ and $\alpha_l$. Hence we obtain $n_{l-2}(\beta) + 2
n_{l}(\beta) - 2 \leq 1.$ Thus, if $\beta \neq \alpha_{l-1}$, we
have the following alternatives:
\begin{enumerate}
 \item[-$\;$ ] $n_{l-2} (\beta) =1$ and $n_{l}(\beta)=0$, then $\langle \beta, \alpha_{l-1}
\rangle = \frac{\sf k}{2}$;
 \item[-$\;$ ] $n_{l-2}(\beta) =2$ and $n_{l}(\beta)=0$, then $\langle \beta, \alpha_{l-1} \rangle = 0$;
  \item[-$\;$ ] $n_{l-2}(\beta)=1$ and $n_l (\beta) = 1$, then $\langle \beta, \alpha_{l-1}
\rangle = -\frac{\sf k}{2}$.
\end{enumerate}
Therefore, in this case, if we have a positive root $\beta = \dots
+ \alpha_{l-1} + \dots$, such that $\langle \beta , \alpha_{l-1}
\rangle = {\sf k}$, then necessarily $\beta = \alpha_{l-1}$.
\end{proof}

\begin{proposition}\label{pIII}
If  $(M=G/K,\sigma,\langle\cdot,\cdot\rangle)$ is an irreducible
$3$-symmetric space Type $A_{3}III$, then $(M,g,J)$ is homogeneous
NK Type IV.
\end{proposition}
\begin{proof} Let $\alpha_i$ be the simple positive root which defines
 the inner automorphism $\sigma$. We consider an
$\mathrm{Ad}(K)$-invariant subspace $\mathcal{H}_1$ of
$\mathcal{H}$.  For $\beta \in \Delta^+_{1}$ and  $\langle \alpha_i
, \alpha_i \rangle = \sf{k}$,    by Lemma \ref{scalproductbeta} we
have: $\beta=\alpha_i$ or $\langle \beta , \alpha_i \rangle=0$, or
$\langle \beta , \alpha_i \rangle= \frac{\sf k}{2}$ or $\langle
\beta , \alpha_i \rangle= - \frac{\sf k}{2}$.

If $X\in \mathcal{H}_1$, then $JX\in \mathcal{H}_1$ and  $X$ is
given by
\[ \textstyle
X  = \sum_{a=0,1}\Big (x_{a;\alpha_i} U^a_{\alpha_{i}} + \sum_{\beta \in \Delta^+_{1(0)} \cup \Delta^+_{1(\frac{\sf k}2)}\cup \Delta^+_{1(-\frac{\sf
k}2)}}   x_{a;\beta} U^a_{\beta}\Big ),
\]
where $\Delta^+_{1(c)} = \{ \beta \in \Delta^+_{1} \, | \, \langle
\beta , \alpha_i \rangle = c \}$. Now,  taking Lemma \ref{bracket}
into account, it is obtained
\begin{eqnarray*}
\frac{2^2}{{\sf k}^2} \mathrm{ad}(\sqrt{-1}H_{\alpha_i})^2(X) &  =
& \textstyle \sum_{a=0,1}\Big ( - 2^2 x_{a;\alpha_i}
U^a_{\alpha_{i}} - \sum_{ \beta \in \Delta^+_{1(\frac{\sf k}2)}}
x_{a;\beta} U^a_{\beta} + \sum_{\beta \in \Delta^+_{1(-\frac{\sf
k}2)}}
x_{a;\beta} U^a_{\beta}\Big ),\\[0.4pc]
\frac{2^4}{{\sf k}^4} \mathrm{ad}(\sqrt{-1}H_{\alpha_i})^4(X) & =
& \textstyle \sum_{a=0,1}\Big (2^4 x_{a;\alpha_i} U^a_{\alpha_{i}}
+ \sum_{ \beta \in \Delta^+_{1(\frac{\sf k}2)}}  x_{a;\beta}
U^a_{\beta} - \sum_{\beta \in \Delta^+_{1(-\frac{\sf k}2)}}
x_{a;\beta} U^a_{\beta}\Big ).
\end{eqnarray*}
This implies that $Z=x_{0\alpha_i} U^0_{\alpha_{i}} + x_{1\alpha_i}
U^1_{\alpha_{i}} \in \mathcal{H}_1$ and $JZ \in \mathcal{H}_1$.

 If $Z \neq 0$, we also obtain
$U^0_{\alpha_{i}}, U^1_{\alpha_{i}} \in \mathcal{H}_1$. This
implies $\mathcal{H}_1= \mathcal{H}$ because $\mathcal{H}_1$ is
$\mathrm{Ad}(K)$-invariant.

If $Z=0$, then $U^0_{\alpha_{i}}, U^1_{\alpha_{i}} \in
\mathcal{H}_2$, where $\mathcal{H}_2$ is the orthogonal complement
of $\mathcal{H}_1$ in $\mathcal{H}$. Since $\mathcal{H}_2$ is also
$\mathrm{Ad}(K)$-invariant, then $\mathcal{H}_2= \mathcal{H}$ and
$\mathcal{H}_1=0$. Hence $\mathcal{H}$ is complex
$\mathrm{Ad}(K)$-irreducible. Therefore, using Proposition
\ref{trespuntouno11probis} (ii),  $\mathcal{H}$ is real
$\mathrm{Hol}(\nabla^{U(n)})$-irreducible  and the nearly K{\"a}hler
manifold $M$ is homogeneous NK Type IV.
\end{proof}

Now Theorem \ref{tmaincla} is a direct consequence from Theorem
\ref{tquotient} and Propositions \ref{I,II}, \ref{III,IV} and
\ref{pIII}.

\section{Homogeneous nearly K\"ahler manifolds with special algebraic torsion}\indent
In this section we obtain the eigenvalues $(l,k,m)$ of $r$ and the
dimensions of the corresponding  eigenbundles for each irreducible
homogeneous nearly K\"ahler manifold with special algebraic
torsion. From Theorem \ref{tmaincla}, using \cite[Theorem
4.1]{Nagy2}, we have
\begin{corollary} A strict homogeneous nearly K\"ahler manifold has special algebraic torsion if and only if it is a Riemannian product of compact irreducible $3$-symmetric spaces of Type $A_{3}II$ and Type $A_{3}III.$
\end{corollary}

\noindent We start studying irreducible $3$-symmetric spaces  of
Type $A_{3}II.$ First, we shall show that each ${\mathcal V}_{k},$ $k = 1,2,3,$ defined in (\ref{VV1}) is an eigenspace of $r$ in ${\mathfrak m}.$ Given a subset ${\mathcal S}$ of $\Delta^{+}$ and
$\alpha\in \Delta,$  we denote by $n_{\alpha}({\mathcal S})$ to
the number of elements in $\{\beta \in {\mathcal S}\mid \; \beta +
\alpha \in \Delta\;\mbox{and}\; \beta + 2\alpha\notin \Delta \}.$

If the inner automorphism $\sigma$ is of Type $A_{3}II,$ the
complex simple Lie algebra ${\mathfrak g}_{\mathbb C}$ is
${\mathfrak a}_{n-1}$ $(n\geq 3),$ ${\mathfrak d}_{n}$ $(n\geq 4)$
or ${\mathfrak e}_{6}.$ Here, $\langle \alpha,\alpha\rangle$ is constant, for all
$\alpha\in \Delta,$ which we shall denote by $\kappa.$

\begin{lemma}\label{IIr} On irreducible compact $3$-symmetric spaces of Type $A_{3}II$ we have:
$$
\begin{array}{lcl}
r(U^{a}_{\alpha}) & = & 2 \kappa \, n_{-\alpha}(\Delta^{+}_{1,0})
U^{a}_{\alpha} = 2 \kappa \,
n_{-\alpha}(\Delta^{+}_{0,1}))U^{a}_{\alpha},\\[0.4pc]
r(U^{a}_{\beta}) & = & 2 \kappa \, n_{-\beta}(\Delta^{+}_{1,1}) U^{a}_{\beta} = 2 \kappa \, n_{\beta}(\Delta^{+}_{0,1}))U^{a}_{\beta},\\[0.4pc]
r(U^{a}_{\gamma}) & = & 2 \kappa \, n_{-\gamma}(\Delta^{+}_{1,1})
U^{a}_{\gamma} =  2 \kappa  \, n_{\gamma}(\Delta^{+}_{1,0}))U^{a}_{\gamma},
\end{array}
$$
where $\alpha\in \Delta^{+}_{1,1},$ $\beta\in \Delta^{+}_{1,0}$
and $\gamma\in \Delta^{+}_{0,1},$ and $a=0,1.$ In particular, note
that $
n_{-\alpha}(\Delta^{+}_{1,0}) = n_{-\alpha}(\Delta^{+}_{0,1}))
=c_1,$ $n_{-\beta}(\Delta^{+}_{1,1})= n_{\beta}(\Delta^{+}_{0,1}))
=c_2$ and $ n_{-\gamma}(\Delta^{+}_{1,1}) =
n_{\gamma}(\Delta^{+}_{1,0}))=c_3,$ where $c_1$, $c_2$ and $c_3$ are integral numbers
independent of $\alpha$, $\beta$ and $\gamma$, respectively.
\end{lemma}
\begin{proof}
 Because, one gets:
\begin{enumerate}
\item[{$\bullet$}] $\alpha + \beta\notin \Delta$ and if $\alpha - \beta \in \Delta$ then $\alpha - \beta \in \Delta_{0,1},$
\item[{$\bullet$}] $\alpha + \gamma\notin \Delta$ and if $\alpha - \gamma\in \Delta$ then $\alpha - \gamma \in \Delta_{1,0},$
\item[{$\bullet$}] $\beta - \gamma\notin \Delta$ and if $\beta + \gamma\in \Delta$ then $\beta + \gamma \in \Delta_{1,1},$
\end{enumerate}
(\ref{xi}) and Lemma \ref{bracket} implies
\[
\xi_{U^{0}_{\alpha}}U^{0}_{\beta} = -\frac{N_{-\beta,\alpha}}{2}U^{0}_{\alpha -\beta},\;\;\xi_{U^{0}_{\alpha}}U^{0}_{\gamma} = -\frac{N_{-\gamma,\alpha}}{2}U^{0}_{\alpha -\gamma},\;\;\xi_{U^{0}_{\beta}}U^{0}_{\gamma} = -\frac{N_{\beta,\gamma}}{2}U^{0}_{\beta + \gamma},
\]
being zero for the rest. Hence $U^{a}_{\alpha},$ $U^{a}_{\beta}$
and $U^{a}_{\gamma}$ are eigenvectors for $r$. Moreover, we have
$$
\begin{array}{lcllcllcl}
\xi^{2}_{U^{0}_{\alpha}}U^{0}_{\beta} & = & -\frac{N_{-\beta,\alpha}^{2}}{4}U^{0}_{\beta}, & & \xi^{2}_{U^{0}_{\beta}}U^{0}_{\alpha} & = & -\frac{N_{-\beta,\alpha}^{2}}{4}U^{0}_{\alpha},\\[0.4pc]
\xi^{2}_{U^{0}_{\alpha}}U^{0}_{\gamma} & = & -\frac{N_{-\gamma,\alpha}^{2}}{4}U^{0}_{\gamma}, & & \xi^{2}_{U^{0}_{\gamma}}U^{0}_{\alpha} & = & -\frac{N_{-\gamma, \alpha}^{2}}{4}U^{0}_{\alpha},\\[0.4pc]
\xi^{2}_{U^{0}_{\beta}}U^{0}_{\gamma} & = &
-\frac{N_{\beta,\gamma}^{2}}{4}U^{0}_{\gamma}, & &
\xi^{2}_{U^{0}_{\gamma}}U^{0}_{\beta} & = &
-\frac{N_{\beta,\gamma}^{2}}{4}U^{0}_{\beta}
\end{array}
$$
and the others being zero. Hence, using Lemma
\ref{cincopuntotreslem}, we have
$$
\begin{array}{lcl}
\langle rU^{0}_{\alpha},U^{0}_{\alpha}\rangle & = & 4
\sum_{\beta\in \Delta^{+}_{1,0}}N_{-\beta,\alpha}^{2} =
4 \sum_{\gamma\in \Delta^{+}_{0,1}}N_{-\gamma, \alpha}^{2},\\[0.4pc]
\langle rU^{0}_{\beta},U^{0}_{\beta}\rangle & = & 4\sum_{\alpha\in
\Delta^{+}_{1,1}}N_{-\beta,\alpha}^{2}  =
4 \sum_{\gamma\in \Delta^{+}_{0,1}}N_{\beta,\gamma}^{2},\\[0.4pc]
\langle rU^{0}_{\gamma},U^{0}_{\gamma}\rangle & = &
4\sum_{\alpha\in \Delta^{+}_{1,1}} N_{-\gamma,\alpha}^{2}  = 4
\sum_{\beta\in \Delta^{+}_{1,0}}N_{\beta,\gamma}^{2}.
\end{array}
$$
Now Lemma follows taking into account that
$$
N_{-\beta,\alpha}^{2} = \frac{\kappa}{2}\;\;\mbox{if}\; \alpha - \beta \in \Delta,\;\;\; N_{-\gamma,\alpha}^{2}  =  \frac{\kappa}{2}\;\;\mbox{if}\; \alpha - \gamma \in \Delta,\;\;\; N_{\beta,\gamma}^{2} =  \frac{\kappa}{2}\;\;\mbox{if}\; \beta +
\gamma \in \Delta,
$$
for all $\alpha\in \Delta^{+}_{1,1},$ $\beta\in \Delta^{+}_{1,0}$
and $\gamma\in \Delta^{+}_{0,1}$, and being zero in other cases.
\end{proof}

\begin{proposition}\label{pA3II}
On a compact irreducible $3$-symmetric of Type $A_{3}II,$ the
eigenvalues $(l, k,m)$ of $r$ and the dimensions corresponding to
the eigenbundles ${\mathcal V}_{1},$ ${\mathcal V}_{2}$ and
${\mathcal V}_{3}$, respectively, are given in Table
\ref{tab:AII}.
\end{proposition}
\begin{proof}
On ${\mathfrak a}_{n-1}$ a set $\Delta^{+}$
of the positive roots is given by
\[
\Delta^{+} = \{\alpha_{p,q} = \alpha_{p}+\alpha_{p+1} + \dots
+\alpha_{q}\mid \; 1\leq p\leq q\leq n-1\}.
\]
Then, for $\sigma = {\mathrm Ad}_{\exp 2\pi\sqrt{-1}H},$ with $H = \frac{1}{3}(H_{i} + H_{j}),$ $1<i<j<n-1,$ we obtain
$\Delta^{+}(H) = \Delta^{+}_{1}(H) \cup \Delta^{+}_{2}(H)\cup
\Delta^{+}_{3}(H)$ where $\Delta^{+}_{1}(H) = \{\alpha_{p,q}\mid
\;1\leq p\leq q <i\},$ $\Delta^{+}_{2}(H) = \{\alpha_{p,q}\mid
\;i<p\leq q <j\}$ and $\Delta^+_{3}(H) = \{\alpha_{p,q}\mid \;j<
p\leq q \leq n-1\}.$ Then ${\mathfrak k}$ is of type ${\mathfrak
a}_{i-1}\oplus {\mathfrak a}_{j-i-1} \oplus {\mathfrak
a}_{n-1-j}\oplus {\mathfrak T}^{2}.$ Hence $M$ is the quotient
manifold $\Lie{SU}(n)/\Lie{S}(\Lie{U}(i)\times \Lie{U}(j-i)\times
\Lie{U}(n-j)).$

On ${\mathfrak m} = \mathbb{R}
\{U^{a}_{\alpha}\mid \; \alpha\in
\Delta^{+}\setminus\Delta^{+}(H)\}$ we consider the orthogonal
decomposition ${\mathfrak m} = {\mathcal V}_{1}\oplus {\mathcal
V}_{2} \oplus {\mathcal V}_{3},$ given in (\ref{VV1}), where
 $$
 \begin{array}{lcl}
 \Delta^{+}_{1,1} & = & \{\alpha_{p,q}\mid \; p\leq i<j\leq q\},\\[0.4pc]
 \Delta^{+}_{1,0} & = & \{\alpha_{p,q}\mid \; 1\leq p\leq i\leq q <j\},\\[0.4pc]
 \Delta^{+}_{0,1} & = & \{\alpha_{p,q}\mid \; i<p\leq j\leq q\leq n-1\}.
  \end{array}
 $$

\noindent Fixed $\alpha = \alpha_{p,q}\in \Delta^{+}_{1,1},$ a
root $\beta = \alpha_{p',q'}$ in $\Delta^{+}_{1,0}$ satisfies
$\alpha - \beta\in \Delta$ if and only if $p' = p.$ This implies
that $n_{-\alpha}(\Delta^{+}_{1,0}) = j-i.$ In a similar way we
get
$$
n_{-\alpha}(\Delta^{+}_{1,0})  =  n_{-\alpha}(\Delta^{+}_{0,1})  =  j-i,\;\;\;\;
n_{-\beta}(\Delta^{+}_{1,1})  =  n_{\beta}(\Delta^{+}_{0,1})  =  n-j,\;\;\;\;
n_{-\gamma}(\Delta^{+}_{1,1})  =  n_{\gamma}(\Delta^{+}_{1,0})  = i,
$$
for all $\alpha\in \Delta^{+}_{1,1},$ $\beta\in \Delta^{+}_{1,0}$
and $\gamma\in \Delta^{+}_{0,1}.$ Therefore, it follows from Lemma
\ref{IIr} that $l= 2(j-i)\kappa,$ $k = 2(n-j)\kappa$ and $m =
2i\kappa.$

On $\mathfrak{d}_{n}$ $(n\geq 4)$, \vspace{-8mm}

$$
\xymatrix@R=-2.0cm@C=.7cm{
 & & & & \stackrel{1}{\stackrel{\circ}{\alpha_{n-1}}} \ar@{-}[dl]\\
\;\;\;\;\stackrel{1}{\stackrel{\circ}{\alpha_{1}}} \ar@{-}[r] &
\stackrel{2}{\stackrel{\circ}{\alpha_{2}}} \ar@{-}[r] & \; \dots
\ar@{-}[r] & \stackrel{2}{\stackrel{\circ}{\alpha_{n-2}}} \ar@{-}[dr]&  \\
 & &  &  & \stackrel{1}{\stackrel{\circ}{\alpha_{n}}}},
$$
the automorphism $\sigma$ of Type $A_{3}II$ is determined by $H =
\frac{1}{3}(H_{n-1} + H_{n}).$ The other possibilities for $H,$ i.e.
$H = \frac{1}{3}(H_{n} + H_{n-1})$ or $H = \frac{1}{3}(H_{1} +
H_{n}),$ give automorphisms which are conjugated with the first one.
A set $\Delta^{+}$ of the positive roots is given by
\[
\Delta^{+}\hspace{-0,1cm} = \hspace{-0,1cm}\{\alpha_{p}\; (1\leq p\leq n),\; \alpha_{p,q}\; (1\leq
p< q< n),\; \tilde{\alpha}_{p,q} = \alpha_{p,n-2} + \alpha_{q,n}\;
(1\leq p<q\leq n,\; p\leq n-2)\}.
\]
Then
 \[
 \Delta^{+}(H) = \{\alpha_{p,q}\; (1\leq p\leq q\leq n-2)\}
 \]
 and ${\mathfrak k}$ is of type ${\mathfrak a}_{n-2} \oplus {\mathfrak T}^{2}.$
 Hence $M$ is the quotient manifold $\Lie{SO}(2n)/(\Lie{U}(n-1)\times \Lie{SO}(2)).$ Next,
 on ${\mathfrak m} = \mathbb{R} \{U^{a}_{\alpha}\mid \; \alpha\in
\Delta^{+}\setminus\Delta^{+}(H)\}$, we consider the corresponding
orthogonal decomposition ${\mathfrak m} = {\mathcal V}_{1}\oplus
{\mathcal V}_{2} \oplus {\mathcal V}_{3}$ into the vertical
distributions ${\mathcal V}_{i},$ $i=1,2,3,$ where the
corresponding subset of positive roots are given by
 $$
 \begin{array}{lcl}
\Delta^{+}_{1,1} & = & \{\tilde{\alpha}_{p,q}\mid \; 1\leq p <q\leq n-1\},\\[0.4pc]
 \Delta^{+}_{1,0} & = & \{\alpha_{p,n-1}\mid \; 1\leq p\leq n-1\},\\[0.4pc]
 \Delta^{+}_{0,1} & = & \{\alpha_{n}, \tilde{\alpha}_{p,n} \mid \; 1\leq p\leq n-2\}.

 \end{array}
 $$
Then we have
$$
\begin{array}{l}
n_{-\alpha}(\Delta^{+}_{1,0})  =  n_{-\alpha}(\Delta^{+}_{0,1})  =
2,\\[0.4pc]
n_{-\beta}(\Delta^{+}_{1,1})  =  n_{\beta}(\Delta^{+}_{0,1})  =  n_{-\gamma}(\Delta^{+}_{1,1})  =  n_{\gamma}(\Delta^{+}_{1,0})  = n-2,
\end{array}
$$
for all $\alpha\in \Delta^{+}_{1,1},$ $\beta\in \Delta^{+}_{1,0}$
and $\gamma\in \Delta^{+}_{0,1}.$ Using above lemma,
${\mathcal V}_{1}$ and ${\mathcal V}_{2}\oplus {\mathcal V}_{3}$ are eigenspaces of $r$ with
eigenvalues $l = 4 \kappa$ and $k = m = 2 (n-2) \kappa,$ respectively.

Finally, a system of positive roots $\Delta^{+}$ on $\mathfrak{e}_{6}$,
    \vspace{-4mm}

 $$
 \xymatrix@R=.5cm@C=.8cm{
& & \stackrel{2}{\stackrel{\circ}{\alpha_{2}}} \ar@{-}[d] & &\\
 \stackrel{1}{\stackrel{\circ}{\alpha_{6}}} \ar@{-}[r] &
\stackrel{2}{\stackrel{\circ}{\alpha_{5}}} \ar@{-}[r] &
\stackrel{3}{\stackrel{\circ}{\alpha_{4}}} \ar@{-}[r] &
\stackrel{2}{\stackrel{\circ}{\alpha_{3}}} \ar@{-}[r] &
\stackrel{1}{\stackrel{\circ \, ,}{\alpha_{1}}}}
$$
is given by
$$
\begin{array}{lcl}
\Delta^{+} &  = &\{\alpha_{1},\alpha_{2};\; \alpha_{p,q}\;
 (3\leq p\leq q \leq 6);\; \alpha_{1,p}, \alpha_{2,p}\; (4\leq p\leq 6);
 \; \alpha_{1} + \alpha_{3,p}\; (3\leq p\leq 6);\\[0.4pc]
& & \hspace{0.2cm} \alpha_{2} + \alpha_{4,p}\; (4\leq p\leq 6);\;  \alpha_{4} + \alpha_{p,q}\;
 (p=1,2; q = 5,6);\;  \alpha_{1,6} + \alpha_{p,q}\; (3\leq p<q\leq 5);\\[0.4pc]
 & & \hspace{0.2cm} \alpha_{2,6} + \alpha_{4,5}\, , \alpha_{1,5} + \alpha_{3,4}\, , \alpha_{1,6} + \alpha_{3,5}
  + \alpha_{4}, \mu = \alpha_{1,6} + \alpha_{2,5} + \alpha_{4}\}
\end{array}
$$
and the automorphism $\sigma$ of Type $A_{3}II$ is determined by
$H = \frac{1}{3}(H_{1} + H_{6}).$ Hence we get
\[
\Delta^{+}(H) = \{\alpha_{2},
 \alpha_{p,q}\;(3\leq p\leq q\leq 5), \alpha_{2,4}\, , \alpha_{2,5} \, ,\alpha_{2} + \alpha_{4},\alpha_{2}+\alpha_{4,5}\, ,
 \alpha_{4} + \alpha_{2,5}\}
\]
\noindent and
$$
\begin{array}{lcl}
\Delta^{+}_{1,1} & = & \{\alpha_{1,6}\,, \alpha_{1} +
\alpha_{3,6}\, ,\alpha_{4} + \alpha_{1,6}\,,\alpha_{1,6} +
\alpha_{p,q} \; (3\leq p<q\leq 5), \alpha_{1,6} + \alpha_{3,5} +
\alpha_{4}, \mu\},\\[0.4pc]
\Delta^{+}_{1,0} & = & \{\alpha_{1}, \alpha_{1,p}\; (p=4,5), \alpha_{1} + \alpha_{3,p}\; (3\leq p\leq 5),
 \alpha_{4} + \alpha_{1,5}\, , \alpha_{1,5} + \alpha_{3,4}\},\\[0.4pc]
\Delta^{+}_{0,1} & = & \{ \alpha_{p,6}\; (3\leq p\leq 6),
\alpha_{2,6}\, , \alpha_{2} + \alpha_{4,6}\, ,
 \alpha_{4} + \alpha_{2,6}\, , \alpha_{4,5} + \alpha_{2,6}\}.
\end{array}
$$
Then we have
$$
\begin{array}{l}
n_{-\alpha}(\Delta^{+}_{1,0})  =  n_{-\alpha}(\Delta^{+}_{0,1})
 =  n_{-\beta}(\Delta^{+}_{1,1})  =  n_{\beta}(\Delta^{+}_{0,1})  =  n_{-\gamma}(\Delta^{+}_{1,1})  =  n_{\gamma}(\Delta^{+}_{1,0})  = 4.
\end{array}
$$
Hence it follows from above lemma that $r$ has exactly one
eigenvalue $l = k = m = 8 \kappa$. Since $ \dim \,  \mathcal{V}_1 = 2 \,
\mathrm{card}(\Delta_{1,1}^+)$, $ \dim \,  \mathcal{V}_2 = 2 \,
\mathrm{card}(\Delta_{1,0}^+)$, and $ \dim \,  \mathcal{V}_3 = 2 \,
\mathrm{card}(\Delta_{0,1}^+)$, the dimensions of $ \mathcal{V}_1$,
$\mathcal{V}_2$ and  $\mathcal{V}_3$ given in Table \ref{tab:AII}
follow.
\end{proof}

\begin{corollary} \label{Einstein:AII}
On irreducible compact $3$-symmetric spaces of Type $A_{3}II$,
Equation \ref{ricricast} is satisfied. Therefore, an irreducible
compact $3$-symmetric spaces of Type $A_{3}II$ is Einstein if and
only if  $\dim \mathcal{V}_1 =  \dim \mathcal{V}_2 = \dim
\mathcal{V}_3$. Thus, $\Lie{SU}(3a)/\Lie{S}(\Lie{U}(a)\times
\Lie{U}(a)\times \Lie{U}(a))$ $(a\geq 1)$,
$\Lie{E}_{6}/(\Lie{SO}(8)\times \Lie{SO}(2)\times \Lie{SO}(2))$  and
$\Lie{SO}(8)/(\Lie{U}(3)\times \Lie{SO}(2))$ are the  irreducible
compact $3$-symmetric spaces of Type $A_{3}II$ which are Einstein.
\end{corollary}

\begin{table}[tp]
  \centering
  \begin{tabular}{cccccc}
   \toprule
 $G$ & $K$ &  $(l,k,m)$  & $\dim \, \mathcal{V}_1$ & $\dim \,
 \mathcal{V}_2$
& $\dim \, \mathcal{V}_3$ \\[1mm]
 \midrule
{\footnotesize $\Lie{SU}(n)/ \mathbb{Z}_n$} & {\footnotesize
$\Lie{S}\left(\Lie{U}(r_{1})\times \Lie{U}(r_{2})\times
\Lie{U}(r_{3})\right)$}
 & {\footnotesize $(2r_{2} \kappa, 2r_{3}
\kappa,  2r_{1} \kappa)$}  & {\footnotesize $2 r_1 r_3$} & {\footnotesize $2 r_1 r_2$} & {\footnotesize $2r_2 r_3$}\\
{\footnotesize $n=r_1+r_2 +r_3$} & {\footnotesize $1 \leq  r_i$} &  &  &  &  \\[2mm]
{\footnotesize $\Lie{SO}(2n)/ \mathbb{Z}_2$} & {\footnotesize
$\Lie{U}(n-1)\times \Lie{SO}(2)$} & {\footnotesize $(4 \kappa,
2(n-2) \kappa, 2(n-2) \kappa)$} & {\footnotesize $(n-1)(n-2)$} &
{\footnotesize $2(n-1)$} & {\footnotesize $2
(n-1)$} \\
{\footnotesize $\; n \geq 4$} & & & & &
\\[2mm]
{\footnotesize $\Lie{E}_{6}/ \mathbb{Z}_3$} & {\footnotesize
$\Lie{SO}(8)\times \Lie{SO}(2)\times \Lie{SO}(2)$} &
{\footnotesize $(8 \kappa,8 \kappa,8 \kappa)$} &{\footnotesize $16$}&{\footnotesize $16$}&{\footnotesize $16$} \\[2mm]
\bottomrule
\end{tabular}
 \vspace{2mm}

  \caption{Type $A_{3}II$. $\kappa = \langle \mu
, \mu \rangle$, $\mu$ is the highest positive root}
  \label{tab:AII}
\end{table}

Next, we consider compact irreducible $3$-symmetric spaces of Type $A_{3}III.$
 \begin{lemma}\label{plm}
On a compact irreducible $3$-symmetric space of Type $A_{3}III$
such that $ \mathrm{T} M =\mathcal{V} \oplus \mathcal{H}$, if
$\beta \in \Delta_2^+$ and $\alpha \in \Delta_1^+$, then we have
that
$$ \textstyle
l = n_{-\beta}(\Delta^{+}_{1})\langle\beta,\beta\rangle, \qquad k= 2
\sum_{\{\gamma \in \Delta_2^+ \mid \; \gamma  - \alpha \in \Delta\}}
\langle \gamma , \gamma \rangle
$$
 are the eigenvalues of $r$ corresponding to the
eigenbundles $\mathcal{V}$ and $\mathcal{H}$, respectively.
\end{lemma}

\begin{remark} {\rm In particular, if  $\mu$ is the highest positive root, we
have $l = n_{-\mu}(\Delta^{+}_{1})\langle\mu,\mu\rangle$. Note that
 $\langle \mu , \mu \rangle =
 \langle \alpha , \alpha \rangle$, where $\alpha$ is any simple
root of maximal length.}
\end{remark}
 \begin{proof}
For $\alpha,$ $\tilde{\alpha}\in \Delta^{+}_{1}$ and $\beta\in
\Delta^{+}_{2}$ one gets:
\begin{enumerate}
\item[{$\bullet$}] If $\alpha + \tilde{\alpha}\in
\Delta$ then $\alpha + \tilde{\alpha}\in \Delta^{+}_{2}.$
 If $\alpha - \tilde{\alpha}\in \Delta$ then $\alpha - \tilde{\alpha}\in \Delta(H).$
\item[{$\bullet$}] $\alpha + \beta\notin \Delta$
and if $\alpha - \beta\in \Delta$ then $\beta - \alpha \in \Delta^+_{1}.$
\end{enumerate}

\noindent From here, using (\ref{xi}) and Lemma \ref{bracket}, we
have
\begin{equation}\label{exi}
\xi_{U^{0}_{\alpha}}U^{0}_{\tilde{\alpha}} =
-\frac{N_{\alpha,\tilde{\alpha}}}{2}U^{0}_{\alpha
+\tilde{\alpha}},\;\;\;\xi_{U^{0}_{\alpha}}U^{0}_{\beta} =
-\frac{N_{-\alpha,\beta}}{2}U^{0}_{\alpha -\beta},
\end{equation}
being zero for the rest. Hence $U^{a}_{\alpha},$ and
$U^{a}_{\beta}$ are eigenvectors for $r$. Moreover, we obtain
\[
\xi^{2}_{U^{0}_{\alpha}}U^{0}_{\tilde{\alpha}}  = -
\frac{N^2_{\alpha,\tilde{\alpha}}}{4}U^{0}_{\tilde{\alpha}},
\;\;\;\; \xi^{2}_{U^{0}_{\alpha}}U^{0}_{\beta}  =
-\frac{N_{-\alpha,\beta}^{2}}{4}U^{0}_{\beta}, \;\;\;\;
\xi^{2}_{U^{0}_{\beta}}U^{0}_{\alpha}  =
-\frac{N_{-\beta,\alpha}^{2}}{4}U^{0}_{\alpha}
\]
and the others being zero. Finally, using Lemma
\ref{cincopuntotreslem}, it follows
$$
\begin{array}{lcl}
\langle rU^{0}_{\alpha},U^{0}_{\alpha}\rangle & = & 4
\sum_{\tilde{\alpha}\in \Delta^{+}_{1}}
N_{\alpha,\tilde{\alpha}}^{2} = 4
 \sum_{\beta\in \Delta^{+}_{2}}N^2_{-\alpha,\beta},\\[0.4pc]
\langle rU^{0}_{\beta},U^{0}_{\beta}\rangle & = & 2\sum_{\alpha\in
\Delta^{+}_{1}}N_{-\beta,\alpha}^{2}.
\end{array}
$$

\noindent From this and (\ref{***}), taking into account that the
 $-\gamma$-serie, $\gamma \in  \Delta_2^+$,
 containing $\alpha \in \Delta_1^+$ is given by $\{\alpha, \alpha - \gamma\}$,  we get
\[ \textstyle
k = \langle rU^{0}_{\alpha},U^{0}_{\alpha}\rangle =  2
\sum_{\{\gamma \in \Delta_2^+ \mid \; \gamma  - \alpha \in \Delta\}}
\langle \gamma , \gamma \rangle
\]
Also we  obtain
\[
l = \langle rU^{0}_{\beta},U^{0}_{\beta}\rangle =
n_{-\beta}(\Delta^{+}_{1})\langle \beta,\beta \rangle.
\]
\end{proof}

\begin{table}[tp]
  \centering
  \begin{tabular}{ccccc}
    \toprule
    $G$&$K$&$(l,k)$& $\dim \, \mathcal{V}$&$\dim \, \mathcal{H}$
    \\[1mm]
    \midrule
    {\footnotesize $\Lie{SO}(2n+1)$,}& {\footnotesize $\Lie{U}(i)\times
    \Lie{SO}(2(n-i)+1)$,}
    & {\footnotesize $((2(n-i) + 1) \kappa,(i-1) \kappa)$} & {\footnotesize $i(i-1)$} & {\footnotesize $2i (2(n-i)+1)$} \\
{\footnotesize $\; n>2$}  & {\footnotesize $\; i>1$} &&& \\[2mm]
{\footnotesize $ \Lie{Sp}(n)/{\mathbb Z}_{2}$,}& {\footnotesize
$\{(\Lie{U}(i)\times \Lie{Sp}(n-i))\}/{\mathbb Z}_{2}, $}&
{\footnotesize $(2(n-i) \kappa,(i+1) \kappa)$} & {\footnotesize $i
(i+1)$} &  {\footnotesize$4i (n-i)$}
\\
 {\footnotesize  $\; n\geq 2$ }  & {\footnotesize $\; 1 \leq i<n$} &&&
 \\[2mm]
{\footnotesize $\Lie{SO}(2n)/{\mathbb Z}_{2}$,} &
 {\footnotesize $\{(\Lie{U}(i)\times \Lie{SO}(2(n-i)))\}/{\mathbb Z}_{2},$} & {\footnotesize
 $(2 (n-i) \kappa, (i-1) \kappa)$} & {\footnotesize $i
 (i-1)$}
 & {\footnotesize $4i (n-i)$}
 \\
 {\footnotesize  $\; n\geq 4$}  & {\footnotesize  $\; 2\leq i\leq n-2$}  &&&
 \\[2mm]
{\footnotesize $\Lie{G}_{2}$} & {\footnotesize $\Lie{U}(2)$} &
{\footnotesize $( 4 \kappa, 2 \kappa)$} & {\footnotesize $2$} &
{\footnotesize $8$} \\[2mm]
{\footnotesize $\Lie{F}_{4}$} & {\footnotesize
$\{\Lie{Spin}(7)\times \Lie{T}^{1}\}/{\mathbb Z}_{2}$}&
{\footnotesize $(4 \kappa, 7 \kappa)$} & {\footnotesize $14$} & {\footnotesize $16$} \\[2mm]
{\footnotesize $\Lie{F}_{4}$} & {\footnotesize
$\{\Lie{Sp}(3)\times \Lie{T}^{1}\}/ {\mathbb
Z}_{2}$} & {\footnotesize $(14 \kappa, 2 \kappa)$}& {\footnotesize $2$} & {\footnotesize $28$} \\[2mm]
{\footnotesize $\Lie{E}_{6}/{\mathbb Z}_{3}$} & {\footnotesize
$\{\Lie{S}(\Lie{U}(5)\times \Lie{U}(1))\times
\Lie{SU}(2)\}/{\mathbb Z}_{2}$} & {\footnotesize $(12 \kappa, 6 \kappa)$} & {\footnotesize $10$} & {\footnotesize $40$}  \\[2mm]
{\footnotesize $\Lie{E}_{6}/{\mathbb Z}_{3}$} & {\footnotesize
$\{[\Lie{SU}(6)/{\mathbb Z}_{3}]\times \Lie{T}^{1}\}/{\mathbb
Z}_{2}$} & {\footnotesize $(20 \kappa ,2
\kappa)$} & {\footnotesize $2$}  & {\footnotesize $40$} \\[2mm]
{\footnotesize $ \Lie{E}_{7}/{\mathbb Z}_{2}$} & {\footnotesize  $
\{\Lie{SU}(2)\times [\Lie{SO}(10)\times \Lie{SO}(2)]\}/{\mathbb
Z}_{2}$} & {\footnotesize $(16 \kappa,10 \kappa)$} & {\footnotesize $20$}  & {\footnotesize $64$} \\[2mm]
{\footnotesize $ \Lie{E}_{7}/{\mathbb Z}_{2}$} & {\footnotesize
$\{\Lie{SO}(12)\times \Lie{SO}(2)\}/{\mathbb Z}_{2}$}&
{\footnotesize $(32 \kappa,2 \kappa)$} & {\footnotesize $2$} &
{\footnotesize $64$} \\[2mm]
{\footnotesize $\Lie{E}_{7}/{\mathbb Z}_{2}$} & {\footnotesize $
\Lie{S}(\Lie{U}(7)\times \Lie{U}(1))/{\mathbb Z}_{4}$} &
{\footnotesize $(20 \kappa ,8 \kappa)$} & {\footnotesize $14$} &
{\footnotesize $70$}\\[2mm]
{\footnotesize $\Lie{E}_{8}$} & {\footnotesize $\Lie{SO}(14)\times
\Lie{SO}(2)$}& {\footnotesize $(32 \kappa, 14 \kappa)$} & {\footnotesize $28$} & {\footnotesize $128$}  \\[2mm]
{\footnotesize $\Lie{E}_{8}$} & {\footnotesize
$\{\Lie{E}_{7}\times \Lie{T}^{1}\}/{\mathbb
Z}_{2}$}& {\footnotesize $(56 \kappa, 2 \kappa)$} & {\footnotesize $2$}  & {\footnotesize $112$} \\[2mm]
    \bottomrule
  \end{tabular} \vspace{2mm}

  \caption{Type $A_{3}III$. $\kappa = \langle \mu
, \mu \rangle$, $\mu$ is the highest positive root}
  \label{tab:AIII}
\end{table}

\begin{proposition}
On a compact irreducible $3$-symmetric of Type $A_{3}III,$ the
eigenvalues $(l, k)$ of $r$ and the dimensions corresponding to
the eigenbundles $\mathcal{V}$ and  $\mathcal{H}$, respectively,
are given in Table \ref{tab:AIII}.
\end{proposition}
\begin{proof} In next Section the corresponding Dynkin diagrams
are displayed. Moreover,   the sets $\Delta^+$, $\Delta^+_2$ and
$\Delta^+_1$ are given. Likewise, on each case, it is fixed the
element $H$ which define the inner automorphism $\sigma$. From
that information and using Lemma \ref{plm}, Table \ref{tab:AIII}
can be completed (we recommend the use of some adequate software
for further checking, see \textsc{Lie} in
\texttt{http://www-math.univ-poitiers.fr})
\end{proof}

\begin{corollary} \label{Einstein:AIII}
On irreducible compact $3$-symmetric spaces of Type $A_{3}III$,
Equation \ref{ricricast2} is satisfied. Therefore, a compact
irreducible $3$-symmetric spaces of Type $A_{3}III$ is Einstein if
and only if  $2 \dim \mathcal{V} = \dim \mathcal{H}$. Thus,
$\Lie{SO}(6a-1)/(\Lie{U}(2a)\times \Lie{SO}(2a-1))$, $a\geq 2,$
$\Lie{Sp}(3a-1)/(\Lie{U}(2a-1)\times \Lie{Sp}(a)),$ $a\geq 1,$ and
$\Lie{SO}(6a+2)/(\Lie{U}(2a+1)\times \Lie{SO}(a))$, $a \geq 2$, are
the  irreducible compact $3$-symmetric spaces of Type $A_{3}III$
which are Einstein.
\end{corollary}

\noindent The quotient $\frac{l}{k}$ measures how  much the nearly K\"ahler
metric of a compact irreducible $3$-symmetric space of Type
$A_{3}III$ deviates to be an Einstein metric. When the fibers of the canonical fibration is  two-dimensional,
it is known that the base manifold is quaternionic K{\"a}hler \cite{Nagy2}
of dimension $4a$, $a\geq 1$.
\begin{lemma}\label{pv2} If $\dim {\mathcal V} =2$,  $\dim \, M = 2n$ and $n=2a+1$,    then
\begin{equation}\label{lm}
l = (n-1)\langle\mu,\mu\rangle,\;\;\;\; k = 2\langle
\mu,\mu\rangle.
\end{equation}
Therefore,  $\frac{l}{k} = \frac{n-1}{2}=a$.
\end{lemma}
\begin{proof}
Our hypothesis is equivalent to $\Delta_2^{+} = \{\mu\}.$ In this
case, since $1 \leq n_{-\alpha}(\Delta_2^+) \leq \frac{\dim
\mathcal{V}}{2} =1$,  $\mu -\alpha\in \Delta^{+}_{1}$, for all
$\alpha \in \Delta_{1}^+$. Hence
\[
n_{-\mu}(\Delta^{+}_{1}) = {\rm card}\;(\Delta^{+}_{1}) =
\frac{1}{2}\dim {\mathcal H} = n-1.
\]
\end{proof}

\noindent Then these twistor spaces
are classified on terms of the quotient $\frac{l}{k}$ as follows:

 \begin{proposition} Let $(M=G/K,\sigma, \langle\cdot,\cdot\rangle)$ be a compact irreducible $3$-symmetric space of Type $A_{3}III.$ We have:
\begin{enumerate}
\item[{\rm (i)}] $\frac{l}{k}=1,$ i.e. the metric is Einstein, if
and only if $M$ is the six-dimensional $3$-symmetric space
${\mathbb C}P^{3}.$

\item[{\rm (ii)}] $\frac{l}{k}= 2$ if and only
if $M$ is one of the $10$-dimensional $3$-symmetric spaces:
${\mathbb C}P^{5},$ $\Lie{SO}(6)/(\Lie{U}(2)\times \Lie{SO}(2)),$
$\Lie{G}_{2}/\Lie{U}(2).$
\item[{\rm (iii)}] $\frac{l}{k}= m,$ for $m\geq 3,$ if and only if $M$ is ${\mathbb
C}P^{2m+1}$ or $\Lie{SO}(m+4)/(\Lie{U}(2)\times \Lie{SO}(m)),$ with the following four exceptions:
$\frac{l}{k}$ is $7$ on $\Lie{F}_{4}/(\Lie{Sp}(3)\times
\Lie{SO}(2);$ $10$ on $\Lie{E}_{6}/(\Lie{SU}(6)\times T^{1});$
$16$ on $\Lie{E}_{7}/(\Lie{SO}(12)\times \Lie{SO}(2))$ and $28$ on
$\Lie{E}_{8}/(\Lie{E}_{7}\times \Lie{SO}(2)).$
\end{enumerate}
\end{proposition}
\begin{remark}{\rm Here the complex projective space
${\mathbb C}P^{2m+1}$ is consider as the $3$-symmetric  space
$\Lie{Sp}(m+1)/(\Lie{U}(1)\times \Lie{Sp}(m))$ equipped with the
standard $\Lie{Sp}(m+1)$-metric and not with the symmetric
Fubini-Study metric.}
\end{remark}

\begin{remark}
 {\rm
 If an irreducible  nearly K{\"a}hler manifold with special
algebraic torsion is Einstein and non-homogeneous, then it must be a
twistor space over a positive quaternionic K{\"a}hler manifold
\cite{Nagy2}. In such a case, by results due to Alexandrov et al.
\cite{AlGrIv},  the manifold has to be six-dimensional. On the other
hand, we have already mentioned in Section \ref{uno} the Hitchin's
result claiming that  any positive quaternionic K\"ahler manifold of
dimension four is symmetric \cite{Hitchin}. This implies that the
corresponding six-dimensional twistor space is homogeneous.
Therefore, Corollaries \ref{Einstein:AII} and \ref{Einstein:AIII}
contain the \textit{complete list} of irreducible  nearly K{\"a}hler
manifolds with special algebraic torsion  which are Einstein. The
remaining irreducible simply connected   homogeneous strict nearly
K{\"a}hler Einstein manifolds are  the compact irreducible
$3$-symmetric spaces of Types $A_{3}IV$, $B_3$ and $C_3$ (Tables
\ref{tab:AIV} and \ref{tab:VI}).  The only possibility for
non-homogeneous irreducible strict nearly K{\"a}hler Einstein
manifolds is to be six-dimensional and
$\nabla^{\Lie{U}(n)}$-holonomy real irreducible.
 }
\end{remark}

\section{Canonical fibrations}\indent
 The notion of Lie triple system given in the theory of symmetric spaces to construct totally geodesic submanifolds can be extended to naturally reductive spaces in the following way \cite{Sagle}.
\begin{definition}{\rm Let $(M = G/K,g)$ be a naturally reductive homogeneous
manifold with adap\-ted reductive decomposition ${\frak g} =
{\frak m} \oplus {\frak k}.$ A subspace $\nu$ of ${\mathfrak m}$ such that $[\nu,\nu]_{\mathfrak m} \subset {\nu}$ and $[[\nu,\nu]_{\frak
k},\nu]\subset \nu$ is said to be
a} Lie triple system {\rm (L.t.s) of }${\mathfrak m}.$
\end{definition}
\noindent Since ${\mathfrak m}$ is $\mathrm{Ad}(K)$-invariant and
using the Jacobi identity, it follows that ${\mathfrak h} = \nu
\oplus [\nu,\nu]_{\mathfrak k}$ is a Lie subalgebra of ${\mathfrak
g}.$ Denote by $H$ the connected Lie subgroup of $G$ with Lie
algebra ${\mathfrak h}.$ Then, by a similar way than for symmetric
spaces (see \cite[Theorem 7.2, Ch. IV]{He}), one obtains the
following result, based on the fact of that geodesics of $M$
through the origin $o$ of $G/K$ are of type $(\exp tv)o,$ $v\in
{\mathfrak m}$. Note that the  manifold $M'$, there mentioned, can
be expressed as the orbit $H\cdot o$ (see also \cite{Sagle}).

\begin{theorem}\label{tg} Let $(M = G/K,g)$ be a naturally reductive homogeneous
manifold with adap\-ted reductive decomposition ${\frak g} = {\frak
m} \oplus {\frak k}$ and $\nu\subset {\mathfrak m}$ a L.t.s. Then
there exists a {\rm (}unique{\rm )} complete and connected totally
geodesic submanifold $M'$ through $o$ of $M,$ such that $T_{o}M' =
\nu.$ Moreover, $M'$ is the naturally reductive homogeneous manifold
$(M'=H/(H\cap K), \iota^{*}g),$ where $H$ is the connected Lie
subgroup of $G$ with Lie algebra ${\mathfrak h} = \nu \oplus
[\nu,\nu]_{\mathfrak k}$ and $\iota$ is the inclusion map.
\end{theorem}

Now, let $(M = G/K,\sigma, \langle\cdot,\cdot\rangle)$ be a  compact
irreducible $3$-symmetric space  of Type $A_{3}II$ or Type
$A_{3}III$ and let ${\mathcal V}$ be any of the subspaces ${\mathcal
V}_{k},$ $k = 1,2,3,$ of ${\mathfrak m}$ for Type $A_{3}II$ or
${\mathcal V}$ for Type $A_{3}III,$ defined in (\ref{VV1}) and (\ref{VV2}), respectively. In what follows, we refer to
${\mathcal V}$ as a {\em vertical distribution} of $M.$

Because ${\mathcal V}$ is a Lie triple system, each vertical
distribution ${\mathcal V}$ determines a totally geodesic
submanifold $F_{\mathcal V}$ as in Theorem \ref{tg}. Consider the
Lie subalgebras ${\mathfrak g}_{\mathcal V}$ and $\bar{\mathfrak
g}_{\mathcal V}$ of ${\mathfrak g}$ given by
\[
 {\mathfrak g}_{\mathcal V} = {\mathcal V} + [{\mathcal V},{\mathcal V}],\;\;\;\;\bar{\mathfrak g}_{\mathcal V} ={\mathcal V} + {\mathfrak k}
 \]
and their associated connected Lie subgroups $G_{\mathcal V}$ and $\bar{G}_{\mathcal V}$ of $G.$ Note that ${\mathfrak g}_{\mathcal V}$ is an ideal of $\bar{\mathfrak g}_{\mathcal V}.$

\begin{theorem}\label{canonical}
Let $(M = G/K,\sigma, \langle\cdot,\cdot\rangle)$ be a compact
irreducible $3$-symmetric space of Type $A_{3}II$ or Type
$A_{3}III$ and let ${\mathcal V}$ be a vertical distribution. Then
we have:
\begin{enumerate}
\item[{\rm (i)}] The quotient manifold $F_{\mathcal V}=G_{\mathcal
V}/(G_{\mathcal V}\cap K)$ is a  totally geodesic submanifold of $M$
through the origin tangent to ${\mathcal V}$ which is also  an
irreducible compact Hermitian symmetric space.
    \item[{\rm (ii)}] $\bar{G}_{\mathcal V}$ is a closed subgroup of $G$ and
    the projection $\pi:M = G/K\to N_{\mathcal V}=G/\bar{G}_{\mathcal V},$ $aK\mapsto a\bar{G}_{\mathcal V}$,
    is the canonical fibration of $M$ with generic fiber $F_{\mathcal
    V}.$
    \item[{\rm (iii)}] The base space $N_{\mathcal V}= G/\bar{G}_{\mathcal V}$ is an irreducible compact symmetric space with inner associated automorphism of order two on ${\mathfrak g}.$
\item[{\rm (iv)}] $\sigma$ is of Type $A_{3}II$ if and only if $N_{\mathcal V}$ is a Hermitian symmetric space.
\end{enumerate}
\end{theorem}
\begin{proof}
(i) It follows from Theorem \ref{tg} and \cite[Proposition 4.2 and
Theorem 4.1]{Nagy2}.

Next, for each vertical distribution ${\mathcal V}$ of $M,$ we shall
look for an inner  involutive automorphism $\tau_{\mathcal V}$ of
${\mathfrak g}$ such that $({\mathfrak g}, \tau_{\mathcal V})$ is an
irreducible orthogonal symmetric Lie algebra and the set of fixed
points ${\mathfrak g}^{\tau_{\mathcal V}}$ is $\bar{\mathfrak
g}_{\mathcal V}.$ Then $(G,\bar{G}_{\mathcal V})$ becomes into an
associated pair of $({\mathfrak g}, \tau_{\mathcal V}).$ Since the
center of ${\mathfrak g}$ is $\{0\},$ it follows that
$\bar{G}_{\mathcal V}$ is closed (\cite[Proposition 3.6, Ch.
IV]{He}). From here, $\bar{G}_{\mathcal V}/K$ is a regular submanifold of
$G/K$ (see \cite[Proposition 4.4, Ch. II]{He}). Moreover, the
totally geodesic submanifold $F_{\mathcal V}= G_{\mathcal
V}/(G_{\mathcal V}\cap K)$  is a submanifold of $\bar{G}_{\mathcal
V}/K,$ via the immersion $g(G_{\mathcal V}\cap K)\mapsto \iota
(g)K,$ for $g\in G_{\mathcal V}$,  where $\iota \, : \,
G_{\mathcal V} \to \bar{G}_{\mathcal V}$ is the inclusion map.
 Because $\mathrm{T}_{o}F_{\mathcal V} = {\mathcal V} =
\mathrm{T}_{o}(\bar{G}_{\mathcal V}/K),$ $F_{\mathcal V}$ is an
open submanifold of $\bar{G}_{\mathcal V}/K.$ Since it  is also
closed in $\bar{G}_{\mathcal V}/K,$ one gets that $F_{\mathcal
V}=G_{\mathcal V}/(G_{\mathcal V}\cap K) = \bar{G}_{\mathcal
V}/K.$ It proves (ii).

If $\sigma$ is of Type $A_{3}II,$  then $m_{i} = m_{j} = 1$ for
some $i,j\in \{1,\dots ,l\},$ $i<j$ and $\sigma =
\mathrm{Ad}_{\exp 2\pi\sqrt{-1}H},$ where $H = \frac{1}{3}(H_{i} +
H_{j}).$ Let ${\mathfrak m} = {\mathcal V}_{1}\oplus {\mathcal
V}_{2}\oplus {\mathcal V}_{3}$ be the orthogonal decomposition of
${\mathfrak m}$ into its vertical subspaces ${\mathcal V}_{k},$
$k=1,2,3.$ We consider the inner automorphisms of order two of ${\mathfrak g}:$
\[
\tau_{{\mathcal V}_{1}} = \mathrm{Ad}_{\exp\pi\sqrt{-1}(H_{i} +
H_{j})},\;\;\;\; \tau_{{\mathcal V}_{2}} =
\mathrm{Ad}_{\exp\pi\sqrt{-1}H_{j}},\;\;\;\; \tau_{{\mathcal
V}_{3}} = \mathrm{Ad}_{\exp\pi\sqrt{-1}H_{i}}.
\]
Using (\ref{adjoint}),  one gets ${\mathfrak g}^{\tau_{{\mathcal V}_{k}}}
=\bar{\mathfrak g}_{{\mathcal V}_{k}}$ and, because ${\mathfrak g}$ is a
compact semisimple (in fact, simple) Lie algebra the pair
$({\mathfrak g}, \tau_{{\mathcal V}_{k}})$ is an effective
orthogonal symmetric Lie algebra.

The base spaces $N_{{\mathcal V}_{k}} = G/\bar{G}_{{\mathcal
V}_{k}},$ $k = 1,2,3,$ of  the canonical fibration are irreducible
compact symmetric spaces associated with the orthogonal  symmetric
Lie algebras $({\mathfrak g},\tau_{{\mathcal V}_{k}}).$ Moreover,
they are Hermitian symmetric spaces and the complex structure
$J^{N_{{\mathcal V}_{k}}}$ on $N_{{\mathcal V}_{k}}$ is determined,
by using (\ref{Jsym}), by the $\mathrm{Ad}(\bar{G}_{{\mathcal
V}_{k}})$-invariant tensor field $J^{N_{{\mathcal V}_{k}}}$ on the
horizontal distribution ${\mathcal H}_{k} = {\mathcal V}_{k_1}\oplus
{\mathcal V}_{k_2},$ where $(k,k_1,k_2)$ is a cyclic permutation of
$(1,2,3),$ given by
$$
\begin{array}{lcll}
J^{N_{{\mathcal V}_{1}}} & = & {\mathrm{ad}_{\sqrt{-1}(H_{i} -
H_{j})}}_{\mid {\mathcal H}_{1}} & (J^{N_{{\mathcal V}_{1}}}_{\mid
{{\mathcal V}_{2}}} = J_{\mid {\mathcal V}_{2}},\;\;
J^{N_{{\mathcal V}_{1}}}_{\mid {\mathcal V}_{3}} = - J_{\mid
{\mathcal V}_{3}}),\\[0.4pc]
J^{N_{{\mathcal V}_{2}}} & = &
{\mathrm{ad}_{\sqrt{-1}H_{j}}}_{\mid {\mathcal H}_{2}} &
(J^{N_{{\mathcal V}_{2}}}_{\mid {{\mathcal V}_{3}}} = J_{\mid
{\mathcal V}_{3}},\;\;
J^{N_{{\mathcal V}_{2}}}_{\mid {\mathcal V}_{1}} = - J_{\mid {\mathcal V}_{1}}),\\[0.4pc]
J^{N_{{\mathcal V}_{3}}} & = &
{\mathrm{ad}_{\sqrt{-1}H_{i}}}_{\mid {\mathcal H}_{3}} &
 (J^{N_{{\mathcal V}_{3}}}_{\mid {{\mathcal V}_{2}}} = J_{\mid {\mathcal V}_{2}},\;\;
 J^{N_{{\mathcal V}_{3}}}_{\mid {\mathcal V}_{1}} = - J_{\mid {\mathcal V}_{1}}).
\end{array}
$$

Next, suppose $\sigma$ is of Type $A_{3}III,$ with $\sigma =
\mathrm{Ad}_{\exp 2\pi\sqrt{-1}H},$  where $H = \frac{2}{3}H_{i}$
being $m_{i} = 2$ for some $i = 1,\dots ,l.$ Let ${\mathcal V}$ be
the vertical distribution. Then the base space $N_{\mathcal V} =
G/\bar{G}_{\mathcal V}$ is associated with  the irreducible
orthogonal symmetric Lie algebra $({\mathfrak g}_{\mathcal
V},\tau_{\mathcal V} = \mathrm{Ad}_{\exp 2\pi\sqrt{-1}H_{i}})$.
 Hence $N_{\mathcal V}$ is an irreducible
compact symmetric space with inner automorphism $\tau_{\mathcal
V}$ on ${\mathfrak g}$ but it is not a Hermitian symmetric space.
In fact, the roots which determine the isotropy subgroup are
$\{\alpha_{j}\in \pi\mid j \neq i\}\cup \{-\mu\}.$ All of this
proves (iii) and (iv).
\end{proof}
Hence we can conclude
\begin{corollary} Compact irreducible $3$-symmetric spaces of Type $A_{3}II$ or Type $A_{3}III$
are twistor spaces over irreducible compact symmetric spaces,
and over Hermitian symmetric spaces for those of Type $A_{3}II.$
\end{corollary}

\begin{remark}{\rm In Theorems \ref{cano1} and \ref{cano2} we
will   explicitly determine
 all pairs  $({\mathfrak g}_{\mathcal
V},\bar{\mathfrak g}_{\mathcal V})$ for each vertical distribution
${\mathcal V}.$ In particular, we will obtain that ${\mathfrak
g}_{\mathcal V}$ is a compact simple Lie algebra. If $\sigma$ is
of Type $A_{3}II,$ it follows from (\ref{adjoint}) that the pair
$({\mathfrak g}_{{\mathcal V}_{k}},\delta_{k}),$ for each $k =
1,2,3,$ where $\delta_{k}$ is the inner automorphism given by
\[
\delta_{1} =
{\mathrm{Ad}_{\exp\frac{\pi\sqrt{-1}}{2}(H_{i} + H_{j})}}_{\mid
{\mathfrak g}_{{\mathcal V}_{1}}},\;\;\;\;
\delta_{2}  =   {\mathrm{Ad}_{\exp\pi\sqrt{-1}H_{i}}}_{\mid
{\mathfrak g}_{{\mathcal V}_{2}}},\;\;\;\;\delta_{3} =
{\mathrm{Ad}_{\exp\pi\sqrt{-1}H_{j}}}_{\mid {\mathfrak
g}_{{\mathcal V}_{3}}},
\]
is an irreducible  orthogonal symmetric Lie algebra associated to
$F_{{\mathcal V}_{k}}=G_{{\mathcal V}_{k}}/(G_{{\mathcal
V}_{k}}\cap K).$ It is a Hermitian symmetric space with almost
complex structure $J_{{\mathcal V}_{k}} = J_{\mid {\mathcal
V}_{k}}$ and determined by the $\mathrm{Ad}(K)$-invariant tensor
field on ${\mathcal V}_{k}$ given by ${\mathrm{ad}_{\frac{\sqrt{-1}}{2}(H_{i}+
H_{j})}}_{\mid {\mathcal V}_{1}},$ for $k=1,$
${\mathrm{ad}_{\sqrt{-1}H_{i}}}_{\mid {\mathcal V}_{2}},$ for
$k=2,$ and ${\mathrm{ad}_{\sqrt{-1}H_{j}}}_{\mid {\mathcal V}_{3}},$
for $k=3.$ If $\sigma$ is of
Type $A_{3}III,$ the pair $({\mathfrak g}_{\mathcal V}, \delta =
{\mathrm{Ad}_{\exp\pi\sqrt{-1}H_{i}}}_{\mid {\mathfrak
g}_{\mathcal V}})$ is an irreducible orthogonal symmetric Lie
algebra associated to $F_{{\mathcal V}}=G_{{\mathcal
V}}/(G_{{\mathcal V}}\cap K)$ and with almost complex structure
determined by the tensor on ${\mathcal V}$ given by
${\mathrm{ad}_{\sqrt{-1}H_{i}}}_{\mid {\mathcal V}}.$}
 \end{remark}

\begin{theorem}\label{cano1} The canonical fibrations for compact
irreducible $3$-symmetric spaces of Type $A_{3}II$ are:
\begin{enumerate}
\item[{\rm (i)}] ${\mathbb C}G_{r_{i},r_{j}} \to
\frac{\Lie{SU}(n)}{\Lie{S}(\Lie{U}(r_{1})\times
\Lie{U}(r_{2})\times \Lie{U}(r_{3}))} \to {\mathbb
C}G_{r_{k},n-r_{k}}, \;$ where $1 \leq r_{i},\;$ $r_{1} + r_{2} +
r_{3} = n,$ and $(i,j,k)$ is a cyclic permutation of $(1,2,3).$
 \item[{\rm (ii)}] ${\mathbb C}P^{n-1}
\to \frac{\Lie{SO}(2n)}{\Lie{U}(n-1)\times \Lie{SO}(2)} \to
\frac{\Lie{SO}(2n)}{\Lie{U}(n)},$ $n\geq 4.$
 \item[{\rm (iii)}] $\frac{\Lie{SO}(2(n-1))}{\Lie{U}(n-1)} \to
\frac{\Lie{SO}(2n)}{\Lie{U}(n-1)\times \Lie{SO}(2)} \to \mathbb{R}
G_{2(n-1),2},$ $n\geq 4.$
 \item[{\rm (iv)}] $\mathbb{R} G_{2,8} \to \frac{\Lie{E}_{6}}{\Lie{SO}(8)\times
\Lie{SO}(2)\times \Lie{SO}(2)} \to
\frac{\Lie{E}_{6}}{\Lie{SO}(10)\times \Lie{SO}(2)}.$
 \end{enumerate}
\end{theorem}
\begin{remark}{\rm Note that on above compact quotient $G/K$
the action of $G$ is almost effective but not necessarily effective.}
\end{remark}
\begin{proof}
 We take $\Delta^{+}_{1,1},$ $\Delta^{+}_{1,0}$ and $\Delta^{+}_{0,1}$ as in proof of Proposition \ref{pA3II} for the simple complex Lie algebras ${\mathfrak a}_{n-1},$ ${\mathfrak d}_{n}$ and ${\mathfrak e}_{6}.$ Then on ${\mathfrak a}_{n-1}$ $(n\geq 3),$ using Lemma \ref{bracket}, we have:
$$
\begin{array}{lcl}
{\mathfrak g}_{{\mathcal V}_{1}}  & = & \sum_{\alpha\in
\Delta^{+}_{1,1}\cup \Delta^{+}_{1}(H)\cup \Delta^{+}_{3}(H) \atop
a =0,1}\mathbb{R} U^{a}_{\alpha} + \sum_{\alpha\in
\Delta^{+}_{1,1}}\mathbb{R}\sqrt{-1}H_{\alpha},\\[0.4pc]
{\mathfrak g}_{{\mathcal V}_{2}}  & = & \sum_{\alpha\in \Delta^{+}_{1,0}\cup \Delta^{+}_{1}(H)\cup \Delta^{+}_{2}(H) \atop a =0,1}\mathbb{R} U^{a}_{\alpha} + \sum_{\alpha\in \Delta^{+}_{1,0}}\mathbb{R}\sqrt{-1}H_{\alpha},\\[0.4pc]
{\mathfrak g}_{{\mathcal V}_{3}}  & = & \sum_{\alpha\in \Delta^{+}_{0,1}\cup \Delta^{+}_{2}(H)\cup \Delta^{+}_{3}(H) \atop a =0,1}\mathbb{R} U^{a}_{\alpha} + \sum_{\alpha\in \Delta^{+}_{0,1}}\mathbb{R}\sqrt{-1}H_{\alpha}.
\end{array}
$$
Therefore, we get ${\mathfrak g}_{{\mathcal V}_{2}} \cong
{\mathfrak s}{\mathfrak u}(j)$ and ${\mathfrak g}_{{\mathcal
V}_{3}} \cong {\mathfrak s}{\mathfrak u}(n -i).$ Next we show that
${\mathfrak g}_{{\mathcal V}_{1}} \cong {\mathfrak s}{\mathfrak
u}(n+i-j).$ Denote by $\pi_{n-1+i-j}$ a system of simple roots
$\pi_{n-1 + i -j} = \{\beta_{1},\dots ,\beta_{n-1+i-j}\}$ of
${\mathfrak a}_{n-1+i-j}.$ Let $\phi$ be the bijection
$\phi:\pi_{n-1+i-j} \to \{\alpha_{1},\dots ,\alpha_{i-1},
\alpha_{i,j}\, ,\alpha_{j+1},\dots \alpha_{n-1}\},$ given by
\[
\phi(\beta_{1}) = \alpha_{1},\dots ,\phi(\beta_{i-1}) =
\alpha_{i-1}, \;  \phi(\beta_{i}) = \alpha_{i,j}\, ,\;
\phi(\beta_{i+1}) = \alpha_{j+1},\dots ,\phi(\beta_{n-1+i-j}) =
\alpha_{n-1}.
\]
Then $\phi$ can be extended  by linearity to a bijection  from
$\Delta^{+}_{{\mathfrak a}_{l + i-j}}$ to $\Delta^{+}_{1,1} \cup
\Delta^{+}_{1}(H)\cup \Delta^{+}_{3}(H)$, where
$\Delta^{+}_{{\mathfrak a}_{l + i-j}}$ is the positive root set of
${\mathfrak a}_{n-1+i-j}$ generated by $\pi_{n-1+i-j}.$ Also
denote by the same letter $\phi$ the homomorphism from ${\mathfrak
s}{\mathfrak u}(n+i-j))$ to $\bar{\mathfrak g}_{{\mathcal V}_{1}}$
defined by $\phi(\sqrt{-1}H_{\beta_{s}}) =
\sqrt{-1}H_{\phi(\beta_{s})}$ and $\phi(U^{a}_{\beta_{s}}) =
U^{a}_{\phi(\beta_{s})},$ $s = 1,\dots ,n-1 + i-j.$ Then
\[ \textstyle
{\mathfrak g}_{{\mathcal V}_{1}}  =  \sum_{\alpha\in
\phi(\Delta^{+}_{{\mathfrak a}_{n-1+i-j}})\atop a =0,1}\mathbb{R}
U^{a}_{\alpha} + \sum_{\alpha\in \phi(\pi_{n-1+i-j})}
\mathbb{R}\sqrt{-1}H_{\alpha} = \phi({\mathfrak s}{\mathfrak
u}(n+i-j)).
\]
Moreover, we get
$$
\begin{array}{l}
\bar{\mathfrak g}_{{\mathcal V}_{1}} = {\mathfrak g}_{{\mathcal
V}_{1}} + \sum_{\alpha\in \Delta^{+}_{2}(H)\atop a=0,1}\mathbb{R}
U^{a}_{\alpha} +\sum_{\alpha\in \{\alpha_{i},\dots
,\alpha_{j-1}\}}\sqrt{-1}H_{\alpha} \cong {\mathfrak s}{\mathfrak
u}(j-i)\oplus {\mathfrak s}{\mathfrak u}(n+i-j) \oplus {\mathfrak
T}^{1},\\[0.4pc]
\bar{\mathfrak g}_{{\mathcal V}_{2}}  =
 {\mathfrak g}_{{\mathcal V}_{2}} + \sum_{\alpha\in \Delta^{+}_{3}(H)\atop a=0,1}\mathbb{R} U^{a}_{\alpha}
 +\sum_{\alpha\in \{\alpha_{j},\dots ,\alpha_{n-1}\}}\sqrt{-1}H_{\alpha}
 \cong {\mathfrak s}{\mathfrak u}(j)\oplus {\mathfrak s}{\mathfrak u}(n-j) \oplus {\mathfrak T}^{1},\\[0.4pc]
\bar{\mathfrak g}_{{\mathcal V}_{3}}  =  {\mathfrak g}_{{\mathcal V}_{3}}
+ \sum_{\alpha\in \Delta^{+}_{1}(H)\atop a=0,1}\mathbb{R} U^{a}_{\alpha}
+\sum_{\alpha\in \{\alpha_{1},\dots ,\alpha_{i}\}}\sqrt{-1}H_{\alpha}\cong {\mathfrak s}{\mathfrak u}(n-i)
\oplus {\mathfrak s}{\mathfrak u}(i) \oplus {\mathfrak T}^{1}.
\end{array}
$$
\noindent From here and Theorem \ref{canonical}, taking $r_{1} = i,$ $r_{2} = j-i$ and $r_{3} = n-j,$ we have the fibrations given in (i).

\vspace{2mm}

On $\mathfrak{d}_{n}$ $(n\geq 4)$, using again Lemma \ref{bracket}, we
have:
$$
\begin{array}{lcl}
{\mathfrak g}_{{\mathcal V}_{1}}  & = & \sum_{\alpha\in
\Delta^{+}_{1,1}\cup \Delta^{+}(H)\atop a =0,1}\mathbb{R}
U^{a}_{\alpha} + \sum_{\alpha\in
\Delta^{+}_{1,1}}\mathbb{R}\sqrt{-1}H_{\alpha},\\[0.4pc]
{\mathfrak g}_{{\mathcal V}_{2}}  & = &
\sum_{\alpha\in \Delta^{+}_{1,0}\cup \Delta^{+}(H)\atop a =0,1}\mathbb{R} U^{a}_{\alpha}
+ \sum_{\alpha\in \pi\setminus \{\alpha_{n}\}}\mathbb{R}\sqrt{-1}H_{\alpha},\\[0.4pc]
{\mathfrak g}_{{\mathcal V}_{3}}  & = & \sum_{\alpha\in \Delta^{+}_{0,1}\cup
\Delta^{+}(H) \atop a =0,1}\mathbb{R} U^{a}_{\alpha} + \sum_{\alpha\in \pi \setminus\{\alpha_{n-1}\}}\mathbb{R}\sqrt{-1}H_{\alpha}.
\end{array}
$$

\noindent Therefore, we directly get  ${\mathfrak g}_{{\mathcal
V}_{2}}\cong {\mathfrak g}_{{\mathcal V}_{3}} \cong {\mathfrak
s}{\mathfrak u}(n).$ Let $\Delta^{+}_{{\mathfrak d}_{n-1}}$ be the
positive root set for ${\mathfrak d}_{n-1}$ generated by a system
of simple roots $\pi_{n-1} = \{\beta_{1},\dots ,\beta_{n-1}\}$ and
consider the bijection $\phi:\pi_{n-1} \to \{\alpha_{1},\dots
,\alpha_{n-2}, \tilde{\alpha}_{n-2,n-1}\},$ given by
$\phi(\beta_{i}) = \alpha_{i},$ $i=1,\dots ,n-2,$ and
$\phi(\beta_{n-1}) = \tilde{\alpha}_{n-2,n-1}.$ Then $\phi$
determines a homomorphism from ${\mathfrak s}{\mathfrak
o}(2(n-1))$ to the Lie algebra $\bar{\mathfrak g}_{{\mathcal
V}_{1}}$ defined as before. Then
\[ \textstyle
{\mathfrak g}_{{\mathcal V}_{1}}  =  \sum_{\alpha\in
\phi(\Delta^{+}_{{\mathfrak d}_{n-1}})\atop a =0,1}\mathbb{R}
U^{a}_{\alpha} + \sum_{\alpha\in \phi(\pi_{n-1})}
\mathbb{R}\sqrt{-1}H_{\alpha} = \phi({\mathfrak s}{\mathfrak
o}(2(n-1)).
\]
Moreover, we get
$$
\begin{array}{l}
\bar{\mathfrak g}_{{\mathcal V}_{1}} = {\mathfrak g}_{{\mathcal
V}_{1}} + \mathbb{R}\sqrt{-1}H_{\alpha_{n}} \cong {\mathfrak
s}{\mathfrak o}(2(n-1))\oplus {\mathfrak T}^{1},\\[0.4pc]
\bar{\mathfrak g}_{{\mathcal V}_{2}}  =  {\mathfrak g}_{{\mathcal V}_{2}} + \mathbb{R}\sqrt{-1}H_{\alpha_{l}} \cong {\mathfrak u}(n-1),\\[0.4pc]
\bar{\mathfrak g}_{{\mathcal V}_{3}}  =  {\mathfrak g}_{{\mathcal V}_{3}} + \mathbb{R}\sqrt{-1}H_{\alpha_{l-1}}\cong {\mathfrak u}(n -1).
\end{array}
$$
\noindent From here and Theorem \ref{canonical}, we obtain the fibrations
(ii) and (iii); (ii) corresponds with the vertical distributions
${\mathcal V}_{2}$ and ${\mathcal V}_{3}$ and (iii) with
${\mathcal V}_{1}.$ \vspace{2mm}

Finally, on ${\mathfrak e}_{6}$  we have:
$$
\begin{array}{lcl}
{\mathfrak g}_{{\mathcal V}_{1}}  & = & \sum_{\alpha\in
\Delta^{+}_{1,1}\cup \Delta^{+}(H)\atop a =0,1}\mathbb{R}
U^{a}_{\alpha} + \sum_{\alpha\in
\Delta^{+}_{1,1}}\mathbb{R}\sqrt{-1}H_{\alpha},\\[0.4pc]
{\mathfrak g}_{{\mathcal V}_{2}}  & = & \sum_{\alpha\in \Delta^{+}_{1,0}\cup \Delta^{+}(H)\atop a =0,1}\mathbb{R} U^{a}_{\alpha} + \sum_{\alpha\in \pi\setminus \{\alpha_{6}\}}\mathbb{R}\sqrt{-1}H_{\alpha},\\[0.4pc]
{\mathfrak g}_{{\mathcal V}_{3}}  & = & \sum_{\alpha\in \Delta^{+}_{0,1}\cup \Delta^{+}(H) \atop a =0,1}\mathbb{R} U^{a}_{\alpha} + \sum_{\alpha\in \pi \setminus\{\alpha_{1}\}}\mathbb{R}\sqrt{-1}H_{\alpha}.
\end{array}
$$
Using the corresponding Dynkin diagrams we get ${\mathfrak
g}_{{\mathcal V}_{2}}\cong {\mathfrak g}_{{\mathcal V}_{3}} \cong
{\mathfrak s}{\mathfrak o}(10).$ Next we also show that
${\mathfrak g}_{{\mathcal V}_{1}}\cong {\mathfrak s}{\mathfrak
o}(10).$ Let $\Delta^{+}_{{\mathfrak d}_{5}}$ be the positive root
set for ${\mathfrak d}_{5}$ generated by a system of simple roots
$\pi_{5} = \{\beta_{1},\dots ,\beta_{5}\}.$ Let $\phi$ be the
bijection $\phi:\pi_{5} \to \{\alpha_{1} + \alpha_{3,6}\,
,\alpha_{2} \, ,\dots ,\alpha_{5}\},$ given by
\[
\phi(\beta_{1})  = \alpha_{1} + \alpha_{3,6} \, ,\;\;
\phi(\beta_{2}) = \alpha_{2},\;\; \phi(\beta_{3}) =
\alpha_{4},\;\; \phi(\beta_{4}) = \alpha_{5},\;\; \phi(\beta_{5})
= \alpha_{3}.
\]
Then $\phi$ can be extended  to $\Delta_{{\mathfrak d}_{5}} \to
\Delta(H)\cup \Delta_{1,1}$ by linearity and the result follows as
in the above cases. Moreover, we get
\[
\bar{\mathfrak g}_{{\mathcal V}_{1}}   =  {\mathfrak g}_{{\mathcal V}_{1}} + \mathbb{R}\sqrt{-1}H_{\alpha_{6}},\;\;\; \bar{\mathfrak g}_{{\mathcal V}_{2}}  =  {\mathfrak g}_{{\mathcal V}_{2}} + \mathbb{R}\sqrt{-1}H_{\alpha_{6}},\;\;\; \bar{\mathfrak g}_{{\mathcal V}_{3}} = {\mathfrak g}_{{\mathcal
V}_{3}} + \mathbb{R}\sqrt{-1}H_{\alpha_{1}}.
\]
\noindent Hence,  $\bar{\mathfrak g}_{{\mathcal V}_{1}}\cong \bar{\mathfrak
g}_{{\mathcal V}_{2}}\cong \bar{\mathfrak g}_{{\mathcal
V}_{3}}\cong {\mathfrak s}{\mathfrak o}(10)\oplus {\mathfrak
T}^{1}.$
\end{proof}

\begin{remark}{\rm The coset spaces  $\frac{\Lie{SU}(n)}{\Lie{S}(\Lie{U}(1) \times \Lie{U}(1)
\times\Lie{U}(n-2))}$, $n \geq 3$, are the
 compact irreducible $3$-symmetric spaces of Type
$A_{3}II$ such the the fibers of the canonical fibration are
$2$-dimensional. Such a fibration is given by
\[ \textstyle
S^{2}\cong {\mathbb C}P^{1} = \frac{\Lie{SU}(2)}{\Lie{S}(\Lie{U}(1) \times \Lie{U}(1))}\to
\frac{\Lie{SU}(n)}{\Lie{S}(\Lie{U}(1) \times \Lie{U}(1)
\times\Lie{U}(n-2))}  \to {\mathbb C}G_{2,n-2}=
\frac{\Lie{SU}(n)}{\Lie{S}(\Lie{U}(2) \times\Lie{U}(n-2))} .
\]
Furthermore, for these manifolds $k=m$ and  $\frac{l}{k} =n-2$.
Therefore, $\frac{\Lie{SU}(3)}{\Lie{S}(\Lie{U}(1) \times \Lie{U}(1)
\times\Lie{U}(1))}$ known as the flag manifold ${\mathbb F}^{3},$ is
the unique compact irreducible $3$-symmetric space of Type $A_{3}II$
with $2$-dimensional fibers which is also  Einstein.
 }
\end{remark}
\begin{theorem}\label{cano2} The canonical fibrations for compact irreducible $3$-symmetric spaces of Type $A_{3}III$ are:
\begin{enumerate}
\item[{\rm (i)}] $\frac{\Lie{SO}(2i)}{U(i)}\to
\frac{\Lie{SO}(2n+1)}{\Lie{U}(i)\times \Lie{SO}(2(n-i)+1)} \to
\mathbb{R} G_{2i,2(n-i) +1},$ $(n>2,\; i>1).$ \item[{\rm (ii)}]
$\frac{\Lie{Sp}(i)}{\Lie{U}(i)} \to
\frac{\Lie{Sp}(n)}{\Lie{U}(i)\times \Lie{Sp}(n-i)} \to {\mathbb
H}G_{i,n-i},$ $(n\geq 2,\; i<n).$ \item[{\rm (iii)}]
$\frac{\Lie{SO}(2i)}{U(i)}\to \frac{\Lie{SO}(2n)}{\Lie{U}(i)\times
\Lie{SO}(2(n-i))} \to \mathbb{R} G_{2i,2(n-i)},$ $(n\geq 4,\;
2\leq i \leq n-2).$ \item[{\rm (iv)}] $S^{2} \to
\frac{\Lie{G}_{2}}{\Lie{U}(2)} \to
\frac{\Lie{G}_{2}}{\Lie{SU}(2)\times \Lie{SU}(2)}.$ \item[{\rm
(v)}] $S^{2}\to \frac{\Lie{F}_{4}}{\Lie{Sp}(3)\times \Lie{SO}(2)}
\to \frac{\Lie{F}_{4}}{\Lie{Sp}(3)\times \Lie{SU}(2)}.$ \item[{\rm
(vi)}] $\mathbb{R} G_{2,7}\to \frac{\Lie{F}_{4}}{\Lie{SO}(7)\times
\Lie{SO}(2)}\to \frac{\Lie{F}_{4}}{\Lie{SO}(9)}.$ \item[{\rm
(vii)}] ${\mathbb C}G_{1,5}\to
\frac{\Lie{E}_{6}}{\Lie{S}(\Lie{U}(5)\times \Lie{U}(1))\times
\Lie{SU}(2)} \to \frac{\Lie{E}_{6}}{\Lie{SU}(6)\times
\Lie{SU}(2)}.$ \item[{\rm (viii)}] $S^{2}\to
\frac{\Lie{E}_{6}}{\Lie{SU}(6)\times T^{1}}\to
\frac{\Lie{E}_{6}}{\Lie{SU}(6)\times \Lie{SU}(2)}.$ \item[{\rm
(ix)}] $S^{2}\to \frac{\Lie{E}_{7}}{\Lie{SO}(12)\times
\Lie{SO}(2)} \to \frac{\Lie{E}_{7}}{\Lie{SO}(12)\times
\Lie{SU}(2)}.$ \item[{\rm (x)}] ${\mathbb C}G_{1,7}\to
\frac{\Lie{E}_{7}}{\Lie{S}(\Lie{U}(7)\times \Lie{U}(1)}\to
\frac{\Lie{E}_{7}}{\Lie{SU}(8)}.$ \item[{\rm (xi)}] $\mathbb{R}
G_{2,10}\to \frac{\Lie{E}_{7}}{\Lie{SU}(2)\times
\Lie{SO}(10)\times \Lie{SO}(2)} \to
\frac{\Lie{E}_{7}}{\Lie{SO}(12)\times \Lie{SU}(2)}.$ \item[{\rm
(xii)}] $S^{2}\to \frac{\Lie{E}_{8}}{\Lie{E}_{7}\times
\Lie{SO}(2)} \to \frac{\Lie{E}_{8}}{\Lie{E}_{7}\times
\Lie{SU}(2)}.$ \item[{\rm (xiii)}] $\mathbb{R} G_{2,14}\to
\frac{\Lie{E}_{8}}{\Lie{SO}(14)\times \Lie{SO}(2)}\to
\frac{\Lie{E}_{8}}{\Lie{SO}(16)}.$

\end{enumerate}
\end{theorem}
\begin{proof}
If $\sigma$ is of Type $A_{3}III$  then $m_{i} =2$ $(H =
\frac{2}{3}H_{i}),$ for some $i = 1,\dots, n.$ It implies that
${\mathfrak g}_{\mathbb C}$ is one of the following: ${\mathfrak
b}_{n}$ $(n\geq 2),$ ${\mathfrak c}_{n}$ $(n\geq 2),$ ${\mathfrak
d}_{n}$ $(n\geq 4),$ ${\mathfrak g}_{2},$ ${\mathfrak f}_{4},$
${\mathfrak e}_{6},$ ${\mathfrak e}_{7}$ and ${\mathfrak e}_{8}.$

\vspace{2mm}

On ${\mathfrak b}_{n}$ we consider the automorphism $\sigma$
determined by $H = \frac{2}{3}H_{i},$ where $2\leq i\leq n.$ A set
$\Delta^{+}$ of the positive roots is given by
\[
\Delta^{+} = \{\alpha_{p,q} \; (1\leq p\leq q\leq n); \;
\tilde{\alpha}_{p,q} = \alpha_{p} + \dots + \alpha_{q-1}
+2\alpha_{q} + \dots + 2\alpha_{n}\; (1\leq p<q\leq n)\}.
\]
Then,
 \[
 \Delta^{+}(H) = \{\alpha_{p,q}\mid \;1\leq p\leq q <i\}\cup \{\alpha_{p,q}\mid \;i< p\leq q \leq n\}
 \cup \{\tilde{\alpha}_{p,q}\mid \; i<p<q\leq n\}
 \]
and ${\mathfrak k}$ is of type ${\mathfrak a}_{i-1}\oplus
{\mathfrak b}_{n-i}\oplus {\mathfrak T}^{1}.$ Hence $M$ is the
quotient manifold $\Lie{SO}(2n+1)/(\Lie{U}(i)\times
\Lie{SO}(2(n-i) +1)).$ Moreover, the subspace ${\mathcal V}$ of
${\mathfrak m}$ of the vertical distribution has dimension
$i(i-1)$ and it is given by ${\mathcal V} =
\mathbb{R}\{U^{a}_{{\tilde{\alpha}}_{p,q}}\mid \; 1\leq p <q\leq
i\}$ and $\Delta^{+}_{2} = \{\tilde{\alpha}_{p,q}\mid 1\leq
p<q\leq i\}.$ Then we have
\[ \textstyle
 {\mathfrak g}_{\mathcal V}  =
  \sum_{\alpha\in \Delta^{+}_{2}\cup \{\alpha_{p,q}\mid 1\leq p\leq q <i\}\atop a =0,1}\mathbb{R} U^{a}_{\alpha} + \sum_{\alpha\in \Delta^{+}_{2}}\mathbb{R}\sqrt{-1}H_{\alpha}.
\]
Let $\Delta^{+}_{{\mathfrak d}_{i}}$ be the positive root set for
${\mathfrak d}_{i}$ generated  by a system of simple roots
$\pi_{i} = \{\beta_{1},\dots ,\beta_{i}\}.$ Let $\phi$ be the
bijection $\phi:\pi_{i} \to \{\alpha_{1},\dots ,\alpha_{i-1},
\tilde{\alpha}_{i-1,i}\},$ given by $\phi(\beta_{j}) =
\alpha_{j},$ $j=1,\dots ,i-1,$ and $\phi(\beta_{i}) =
\tilde{\alpha}_{i-1,i}.$ Then $\phi$ can be extended  by linearity
to $\Delta_{{\mathfrak d}_{i}}^{+} \to \Delta^{+}_{2}\cup
\{\alpha_{p,q}\mid 1\leq p\leq q <i\}$. Also denote by the same
letter $\phi$ the extended homomorphism from ${\mathfrak
s}{\mathfrak o}(2i)$ to the Lie algebra $\bar{\mathfrak
g}_{\mathcal V}$ defined as before. Then
\[ \textstyle
{\mathfrak g}_{\mathcal V}  =  \sum_{\alpha\in
\phi(\Delta^{+}_{{\mathfrak d}_{i}})\atop a =0,1}\mathbb{R}
U^{a}_{\alpha} + \sum_{\alpha\in \phi(\pi_{i})}
\mathbb{R}\sqrt{-1}H_{\alpha} = \phi({\mathfrak s}{\mathfrak
o}(2i)).
\]

\noindent Moreover, we get
$$
\begin{array}{lcl}
\bar{\mathfrak g}_{\mathcal V}  & = & {\mathfrak g}_{\mathcal V}
\oplus \mathbb{R}\{\sqrt{-1}H_{\alpha_{i+1}}, \dots
\sqrt{-1}H_{\alpha_{l}}\}\oplus \sum_{\alpha\in \Xi }(\mathbb{R} U^{0}_{\alpha} + \mathbb{R} U^{1}_{\alpha})\\[0.4pc]
 & \cong & {\mathfrak s}{\mathfrak o}(2i)\oplus {\mathfrak s}{\mathfrak o}(2(n-i)+1),
 \end{array}
$$
where $\Xi = \{\alpha_{p,q}\mid i<p\leq q\leq n\} \cup
\{\tilde{\alpha}_{p,q}\mid i<p<q\leq n\}$.
 Hence  the fibration given in (i) follows. \vspace{2mm}

 For the complex Lie algebra ${\mathfrak
g}_{\mathbb C}  = \mathfrak{c}_{n}$ $(n\geq 2),$
$$
\xymatrix@R=.5cm@C=.8cm{ \stackrel{2}{\stackrel{\circ}{\alpha_{1}}}
\ar@{-}[r] & \stackrel{2}{\stackrel{\circ}{\alpha_{2}}} \ar@{-}[r] &
\; \dots \ar@{-}[r] & \stackrel{2}{\stackrel{\circ}{\alpha_{n-1}}}
\ar@2{<-}[r] & \stackrel{1}{\stackrel{\circ}{\alpha_{n}}}},
$$

\noindent a set of positive roots is given by
\[
\Delta^{+}  =  \{\alpha_{p,q} \; (1\leq p\leq q\leq n),
\widetilde{\alpha_{p,q}} = \alpha_{p} +\dots + 2\alpha_{q} + \dots +
2\alpha_{n-1} + \alpha_{n}\;\; (1\leq p\leq q\leq n-1)\}.
\]
The inner automorphism $\sigma = \mathrm{Ad}_{\exp
2\pi\sqrt{-1}H},$ where $H = \frac{2}{3}H_{i},$  $1\leq i\leq
n-1.$ Then
\[
\Delta^{+}(H) = \{\alpha_{p,q}\mid \;1\leq p\leq q <i\}\cup
\{\alpha_{p,q}\mid \;i< p\leq q \leq n\}\cup
\{\tilde{\alpha}_{p,q}\mid \; i<p\leq q\leq n-1\}.
 \]
It implies that ${\mathfrak c}_{n}^{\sigma}$ is of type
${\mathfrak a}_{i-1}\otimes {\mathfrak c}_{n-i}\otimes {\mathfrak
T}^{1}.$ Hence $M$ is the quotient manifold
$\Lie{Sp}(n)/(\Lie{U}(i)\times \Lie{Sp}(n-i)).$ Moreover, the
subspace ${\mathcal V}$ of ${\mathfrak m}$ of the vertical
distribution has dimension $i(i+1)$ and is given by ${\mathcal V}
= \{U^{a}_{{\tilde{\alpha}}_{p,q}}\mid \; 1\leq p \leq q\leq i\}$,
and $\Delta^{+}_{2} = \{\tilde{\alpha}_{p,q}\mid 1\leq p\leq q
\leq i\}.$

Now, let $\Delta^{+}_{{\mathfrak c}_{i}}$ be the positive root set
for ${\mathfrak c}_{i}$ generated by a system of simple roots
$\pi_{i} = \{\beta_{1},\dots ,\beta_{i}\}.$ We consider the
bijection $\phi:\pi_{i} \to \{\alpha_{1},\dots ,\alpha_{i-1},
\tilde{\alpha}_{i,i}\},$ given by $\phi(\beta_{j}) = \alpha_{j},$
$j=1,\dots ,i-1,$ and $\phi(\beta_{i}) = \tilde{\alpha}_{i,i}.$
Then $\phi$ can be extended by linearity  to $\Delta_{{\mathfrak
c}_{i}}^{+} \to \Delta^{+}_{2}\cup \{\alpha_{p,q}\mid 1\leq p\leq
q <i\}$ and gives a homomorphism from ${\mathfrak s}{\mathfrak
p}(i)$ to the Lie algebra $\bar{\mathfrak g}_{\mathcal V}.$ Then
$$
\begin{array}{lcl}
{\mathfrak g}_{\mathcal V}  & =  &
\sum_{\alpha\in \Delta^{+}_{2}\cup \{\alpha_{p,q}\mid 1\leq p\leq q <i\}\atop a =0,1}\mathbb{R} U^{a}_{\alpha}
+ \sum_{\alpha\in \Delta^{+}_{2}}\mathbb{R}\sqrt{-1}H_{\alpha}\\[0.4pc]
 & = & \sum_{\alpha\in \phi(\Delta^{+}_{{\mathfrak c}_{i}})\atop a =0,1}\mathbb{R} U^{a}_{\alpha}
 + \sum_{\alpha\in \phi(\pi_{i})} \mathbb{R}\sqrt{-1}H_{\alpha} = \phi({\mathfrak s}{\mathfrak p}(i)).
 \end{array}
$$
Moreover, we get
$$
\begin{array}{lcl}
\bar{\mathfrak g}_{\mathcal V}  & = & {\mathfrak g}_{\mathcal V}
\oplus \mathbb{R}\{\sqrt{-1}H_{\alpha_{i+1}},\dots
\sqrt{-1}H_{\alpha_{l}}\}\oplus \sum_{\alpha\in  \Xi }(\mathbb{R}
U^{0}_{\alpha}
+ \mathbb{R} U^{1}_{\alpha})\\[0.4pc]
 & \cong & {\mathfrak s}{\mathfrak p}(i)\oplus {\mathfrak s}{\mathfrak p}(n-i),
 \end{array}
$$
where $\Xi = \{\alpha_{p,q}\mid i<p\leq q\leq n\}\cup
\{\tilde{\alpha}_{p,q}\mid i<p\leq q\leq n-1\}$.   Hence we get
the fibration given in (ii). \vspace{2mm}

On ${\mathfrak d}_{n},$ the automorphism $\sigma$ corresponds to
$H = \frac{2}{3}H_{i},$ with $2\leq i\leq n-2.$ Then
\[
\Delta^{+}(H) = \{\alpha_{n},\; \alpha_{p,q} \; (1\leq p \leq q
<i),\;\alpha_{p,q}\; (i< p\leq q < n),\; \tilde{\alpha}_{p,q}\;
(i<p< q\leq n,\; p\leq n-2)\}.
 \]
Moreover, $\Delta^{+}_{2} = \{\tilde{\alpha}_{p,q}\; (1\leq
p<q\leq i)\}$ and $\dim {\mathcal V} = i(i-1).$ Let
$\Delta^{+}_{{\mathfrak d}_{i}}$ be the positive root set for
${\mathfrak d}_{i}$ generated by a system of simple roots $\pi_{i}
= \{\beta_{1},\dots ,\beta_{i}\}.$ Let $\phi$ be the bijection
$\phi:\pi_{i} \to \{\alpha_{1},\dots ,\alpha_{i-1},
\tilde{\alpha}_{i-1,i}\},$ given by $\phi(\beta_{i}) =
\alpha_{i},$ $i=1,\dots ,i-1,$ and $\phi(\beta_{i}) =
\tilde{\alpha}_{i-1,i}.$ Then $\phi$ can be extended by linearity
to $\Delta_{{\mathfrak d}_{i}} \to \Delta^{+}_{2}\cup
\{\alpha_{p,q}\mid 1\leq p\leq q <i\}$. Also denote by the same
letter $\phi$ the extended homomorphism from ${\mathfrak
s}{\mathfrak o}(2i)$ to the Lie algebra $\bar{\mathfrak
g}_{\mathcal V}$ defined as before. Then
$$
\begin{array}{lcl}
{\mathfrak g}_{\mathcal V}  & =  &
\sum_{\alpha\in \Delta^{+}_{2}\cup
\{\alpha_{p,q}\mid 1\leq p\leq q <i\}\atop a =0,1}\mathbb{R} U^{a}_{\alpha} + \sum_{\alpha\in \Delta^{+}_{2}}\mathbb{R}\sqrt{-1}H_{\alpha}\\[0.4pc]
 & = &
 \sum_{\alpha\in \phi(\Delta^{+}_{{\mathfrak d}_{i}})\atop a =0,1}
 \mathbb{R} U^{a}_{\alpha} + \sum_{\alpha\in \phi(\pi_{i})} \mathbb{R}\sqrt{-1}H_{\alpha} = \phi({\mathfrak s}{\mathfrak o}(2i)).
 \end{array}
$$
and we get
$$
\begin{array}{lcl}
\bar{\mathfrak g}_{\mathcal V}  & = & {\mathfrak g}_{\mathcal V}
\oplus \mathbb{R}\{\sqrt{-1}H_{\alpha_{i+1}},\dots
\sqrt{-1}H_{\alpha_{n}}\}\oplus \sum_{\alpha\in \Xi \atop a=0,1}
\mathbb{R} U^{a}_{\alpha}\\[0.4pc]
 & \cong & {\mathfrak s}{\mathfrak o}(2i)\oplus {\mathfrak s}{\mathfrak
 o}(2(n-i)),
 \end{array}
$$
where $\Xi = \{\alpha_{n},\alpha_{p,q}\; (i<p\leq q<
n),\tilde{\alpha}_{p,q}\;(i<p<q\leq n, p\leq n-2)\}$. Hence we
obtain the fibration given in (iii). \vspace{2mm}

For the exceptional Lie algebra ${\mathfrak g}_{2}$,
$$
\mathfrak {g}_{2}:\xymatrix@R=.5cm@C=.8cm{
\stackrel{3}{\stackrel{\circ}{\alpha_{1}}} \ar@3{<-}[r] &
\stackrel{2}{\stackrel{\circ}{\alpha_{2}}}},
$$
a set of positive roots is given by $\Delta^{+} =
\{\alpha_{1},\alpha_{2},\alpha_{1} + \alpha_{2}, 2\alpha_{1} +
\alpha_{2}, 3\alpha_{1} + \alpha_{2}, 3\alpha_{1} +
2\alpha_{2}\},$ being $\mu = 3\alpha_{1} + 2\alpha_{2}$ the
maximal root. In this case $H = \frac{2}{3}H_{2}$, $\Delta^{+}_{2}
= \{\mu\}$ and  $\Delta^{+}(H) = \{\alpha_{1}\}.$ Then
$$
\begin{array}{lcl}
{\mathfrak g}_{\mathcal V} & = & \mathbb{R}\sqrt{-1}H_{\mu}\oplus
 \mathbb{R} U^{a}_{\mu}\cong {\mathfrak s}{\mathfrak u}(2)\cong {\mathfrak s}{\mathfrak o}(3),\\[0.4pc]
\bar{\mathfrak g}_{\mathcal V} & = & {\mathfrak g}_{\mathcal
V}\oplus \mathbb{R}\sqrt{-1}H_{\alpha_{1}} \oplus \mathbb{R}
U^{a}_{\alpha_{1}}\cong{\mathfrak s}{\mathfrak u}(2)\oplus
{\mathfrak s}{\mathfrak u}(2).
\end{array}
$$
Hence $F$ is the $2$-sphere $S^{2}$ and the fibering (iv) is
obtained. \vspace{2mm}

On ${\mathfrak f}_{4},$
$$
 \xymatrix@R=.5cm@C=.8cm{
 \stackrel{2}{\stackrel{\circ}{\alpha_{1}}} \ar@1{-}[r] &
   \stackrel{3}{\stackrel{\circ}{\alpha_{2}}} \ar@2{->}[r]
   & \stackrel{4}{\stackrel{\circ}{\alpha_{3}}} \ar@1{-}[r] &
\stackrel{2}{\stackrel{\circ}{\alpha_{4}}}},
$$
$\mu= 2\alpha_{1,4} + \alpha_{2,3} + \alpha_{3}$ and a set of
positive roots is given by
$$
\begin{array}{lcl}
\Delta^{+} &\hspace{-0,3cm}  = \hspace{-0,3cm}&\{\alpha_{p,q}\;
 (1\leq p \leq q \leq 4),\; \alpha_{2,3} + \alpha_{3}, \alpha_{1,q} + \alpha_{p,3}\;(p=2,3;\;q=3,4),
 \alpha_{2,4} + \alpha_{3,q}\; (q=3,4), \\[0.4pc]
  &\hspace{-0,6cm} & \alpha_{14} + \alpha_{p4}\;(p=2,3), \alpha_{1,4} + \alpha_{2,q} + \alpha_{3}\;(q = 3,4),\alpha_{1,4}+\alpha_{2,4}
  +\alpha_{p,3}+\alpha_{3}\;(p=2,3), \mu \}.
\end{array}
$$

Here $H = \frac{2}{3}H_{1}$ or $H = \frac{2}{3}H_{4}.$ For $H =
\frac{2}{3}H_{1},$ one gets
\[
\Delta^{+}(H) = \{\alpha_{p,q}\; (2\leq p\leq q\leq 4), \alpha_{2}
+ 2\alpha_{3}, \alpha_{2} + 2\alpha_{3} + \alpha_{4}, \alpha_{2} +
2\alpha_{3} + 2\alpha_{4}\},
\]
which coincides with a positive root set for ${\mathfrak c}_{3}.$ Then, $\Delta^{+}_{2} =\{\mu\}$ and we have
$$
\begin{array}{lcl}
{\mathfrak g}_{\mathcal V} & = &
\mathbb{R}\sqrt{-1}H_{\mu}\oplus \mathbb{R} U^{a}_{\mu}\cong {\mathfrak s}{\mathfrak u}(2),\\[0.4pc]
\bar{\mathfrak g}_{\mathcal V} & = & {\mathfrak g}_{\mathcal
V}\oplus
\mathbb{R}\{\sqrt{-1}H_{\alpha_{2}},\sqrt{-1}H_{\alpha_{3}},\sqrt{-1}H_{\alpha_{4}}\}
\oplus \sum_{\alpha\in \Delta^{+}(H)}\mathbb{R}
U^{a}_{\alpha}\cong{\mathfrak s}{\mathfrak u}(2)\oplus {\mathfrak
s}{\mathfrak p}(3).
\end{array}
$$
>From here we get (v).

For $H = \frac{2}{3}H_{4},$ one gets
\[
\Delta^{+}(H) = \{\alpha_{p,q}\; (1\leq p\leq q\leq 3), \alpha_{2} +
2\alpha_{3}, \alpha_{1} + \alpha_{2} + 2\alpha_{3}, \alpha_{1} +
2\alpha_{2} + 2\alpha_{3} \}
\]
and
\[
\textstyle \Delta^{+}_{2} = \{\alpha_{2,4} + \alpha_{3,4},
\alpha_{1,4}+\alpha_{p,4}\;(p=2,3),
\alpha_{1,4}+\alpha_{2,4}+\alpha_{3}, \alpha_{1,4} + \alpha_{2,4}
+ \alpha_{p,3} + \alpha_{3}\; (p=2,3), \mu\}.
\]
Hence
 \[
  \textstyle {\mathfrak g}_{\mathcal V} = \bar{\mathfrak g}_{\mathcal
V} = {\mathfrak h} \oplus \sum_{\alpha\in \Delta^{+}(H)\cup
\Delta^{+}_{2}\atop a=0,1}\mathbb{R} U^{a}_{\alpha}.
\]
Put $\beta_{1} = \alpha_{2,4} + \alpha_{3,4},$ $\beta_{2} =
\alpha_{1},$  $\beta_{3} = \alpha_{2}$ and $\beta_{4} =
\alpha_{3}.$ Then $\{\beta_{1},\dots,\beta_{4}\}$ is a system of
simple roots of ${\mathfrak b}_{4}$ and the positive root set
generated by them coincides with $\Delta^{+}(H)\cup
\Delta^{+}_{2}.$ Hence ${\mathfrak g}_{\mathcal V} =
\bar{\mathfrak g}_{\mathcal V}\cong {\mathfrak s}{\mathfrak o}(9)$
and we obtain (vi). \vspace{2mm}

On ${\mathfrak e}_{6},$ $H = \frac{2}{3}H_{3}$ and $H =
\frac{2}{3}H_{5}$ determine conjugate automorphisms of Type
$A_{3}III.$ Also $H = \frac{2}{3}H_{2}$ determines an automorphism
of Type $A_{3}III.$ We start with $H = \frac{2}{3}H_{3}.$ Then
$$
\begin{array}{rcl}
\Delta^{+}(H) & = & \{\alpha_{1},\alpha_{2}, \alpha_{p,q}\;
(4\leq p \leq q \leq 6),\alpha_{2} + \alpha_{4,p}\; (4\leq p\leq 6)\},\\[0.4pc]
\Delta^{+}_{2} & = & \{\alpha_{1,5} + \alpha_{3,4} \, ,
\alpha_{1,6} + \alpha_{3,4}\, , \alpha_{1,6} + \alpha_{3,5}\, ,
\alpha_{1,6} + \alpha_{3,5} + \alpha_{4}, \mu \}.
\end{array}
$$
Let $\Delta^{+}_{{\mathfrak a}_{5}}$ be the positive root set for
${\mathfrak a}_{5}$ generated by a system of simple roots $\pi_{5}
= \{\beta_{1},\dots ,\beta_{5}\}.$ Let $\phi$ be the bijection
$\phi:\pi_{5} \to \{\alpha_{2}, \alpha_{4},\alpha_{5},\alpha_{6},
\alpha_{1,5} + \alpha_{3,4}\},$ given by
\[
\phi(\beta_{1}) = \alpha_{2},\;\;\phi(\beta_{2}) =
\alpha_{4},\;\;\phi(\beta_{3}) = \alpha_{5},\;\;\phi(\beta_{4}) =
\alpha_{6},\;\; \phi(\beta_{5}) = \alpha_{1,5} + \alpha_{3,4}.
\]
Then $\phi$ can be extended by linearity to
$\Delta^{+}_{{\mathfrak a}_{5}} \to \Delta^{+}_{2}\cup
(\Delta^{+}(H)\setminus\{\alpha_{1}\})$. Also denote by the same
letter $\phi$ the extended homomorphism from ${\mathfrak
s}{\mathfrak o}(6)$ to the Lie algebra $\bar{\mathfrak
g}_{\mathcal V}.$ Then
\[ \textstyle
{\mathfrak g}_{\mathcal V}  =  \sum_{\alpha\in \Delta^{+}_{2}\cup
(\Delta^{+}(H)\setminus\{\alpha_{1}\})\atop a =0,1}\mathbb{R}
U^{a}_{\alpha} + \sum_{\alpha\in
\Delta^{+}_{2}}\mathbb{R}\sqrt{-1}H_{\alpha}\cong {\mathfrak
s}{\mathfrak u}(6).
\]
Moreover, we get
\[
\bar{\mathfrak g}_{\mathcal V}   =  {\mathfrak g}_{\mathcal V}
\oplus \mathbb{R}\sqrt{-1}H_{\alpha_{1}}\oplus\mathbb{R}
U^{0}_{\alpha_{1}}\oplus \mathbb{R} U^{1}_{\alpha_{1}}\cong
{\mathfrak s}{\mathfrak u}(6)\oplus {\mathfrak s}{\mathfrak u}(2).
\]
Hence we get the fibration given in (vii).

Next we consider $H = \frac{2}{3}H_{2}.$ Then
\[
\Delta^{+}(H) = \{\alpha_{1}, \alpha_{p,q}\; (3\leq p \leq q \leq
6), \alpha_{1} + \alpha_{3,p}\; (3\leq p \leq 6)\}
\]
and $\Delta^{+}_{2} = \{\mu\}.$ Then
\[
{\mathfrak g}_{\mathcal V} = \mathbb{R}\sqrt{-1}H_{\mu} \oplus
\mathbb{R} U^{0}_{\mu} \oplus \mathbb{R} U^{1}_{\mu} \cong
{\mathfrak s}{\mathfrak u}(2)\cong {\mathfrak s}{\mathfrak o}(3).
\]
Since $\Delta^{+}(H)$ is a root set for ${\mathfrak a}_{5}$ with
simple system $\{\alpha_{1},\alpha_{3},\dots ,\alpha_{6}\},$ we
have
\[
\bar{\mathfrak g}_{\mathcal V}   \cong {\mathfrak s}{\mathfrak u}(6)\oplus {\mathfrak s}{\mathfrak u}(2).
\]
This gives the fibration in (viii).
 \vspace{2mm}

On $\mathfrak{e}_{7}$,
$$ \xymatrix@R=.5cm@C=.8cm{
& & & \stackrel{2}{\stackrel{\circ}{\alpha_{2}}} \ar@{-}[d] & &\\
\;\;\;\;\;\stackrel{1}{\stackrel{\circ}{\alpha_{7}}} \ar@{-}[r] &
\stackrel{2}{\stackrel{\circ}{\alpha_{6}}} \ar@{-}[r] &
\stackrel{3}{\stackrel{\circ}{\alpha_{5}}} \ar@{-}[r] &
\stackrel{4}{\stackrel{\circ}{\alpha_{4}}} \ar@{-}[r] &
\stackrel{3}{\stackrel{\circ}{\alpha_{3}}} \ar@{-}[r] &
\stackrel{2}{\stackrel{\circ \, ,}{\alpha_{1}}}}
$$
a positive system of roots is given by
$$
\begin{array}{lcl}
\Delta^{+} &\hspace{-0,3cm}  =\hspace{-0,3cm} &
\{\alpha_{1},\alpha_{2};\; \alpha_{p,q}\; (3\leq p\leq q \leq 7);\; \alpha_{1,p} \, , \alpha_{2,p}\; (4\leq p\leq 7); \; \alpha_{1} + \alpha_{3,p}\; (3\leq p\leq 7);\\[0.4pc]
 &\hspace{-0,6cm} & \hspace{0.2cm} \alpha_{2} + \alpha_{4,p}\; (4\leq p\leq 7);\;  \alpha_{4} + \alpha_{p,q}\; (p=1,2; q = 5,6,7);\;  \alpha_{1,6} + \alpha_{p,q}\; (3\leq p<q\leq 5);\\[0.4pc]
 &\hspace{-0,6cm} & \hspace{0.2cm} \alpha_{2,p} + \alpha_{4,5}\; (p = 6,7), \alpha_{4,6} + \alpha_{2,7}, \alpha_{1,7} + \alpha_{p,6}\; (p = 3,4), \alpha_{1,7} + \alpha_{p,q}\; (3\leq p<q\leq 5),\\[0.4pc]
 &\hspace{-0,6cm} & \alpha_{1,5} + \alpha_{3,4}, \alpha_{1,6} + \alpha_{p,5} + \alpha_{4}\; (p = 2,3), \alpha_{1,7} + \alpha_{p,q} + \alpha_{4}\; (p=2,3; q = 5,6),\\[0.4pc]
 &\hspace{-0,6cm} & \alpha_{1,7} + \alpha_{p,6} + \alpha_{4,5}\; (p=2,3), \alpha_{1,7} + \alpha_{2,6} + \alpha_{p,5} + \alpha_{4}\; (p=3,4),
  \mu \}.
\end{array}
$$
The possible automorphisms of Type $A_{3}III$ are determined  by
$H = \frac{2}{3}H_{1},$ $H= \frac{2}{3}H_{2}$ or $H =
\frac{2}{3}H_{6}.$

If $H = \frac{2}{3}H_{1},$ then
$$
\begin{array}{lcl}
\Delta^{+}(H) & = & \{\alpha_{2};\; \alpha_{p,q}\; (3\leq p\leq q
\leq 7);\; \alpha_{2,p}\; (4\leq p\leq 7);
 \; \alpha_{2} + \alpha_{4,p}\; (4\leq p\leq 7);\\[0.4pc]
 & & \;  \alpha_{4} + \alpha_{2,q}\; (q = 5,6,7);\; \alpha_{2,p} + \alpha_{4,5}\; (p = 6,7), \alpha_{4,6} + \alpha_{2,7}\},
\end{array}
$$
and $\Delta^{+}_{2} = \{\mu\}.$ Hence, as before, ${\mathfrak
g}_{\mathcal V} \cong {\mathfrak s}{\mathfrak u}(2).$ From Dynkin
diagrams, $\Delta^{+}(H)$ is a root set for ${\mathfrak d}_{5}$
with simple system $\{\beta_{1},\dots ,\beta_{6}\},$ where
$\beta_{k} = \alpha_{8-k}, $ $k = 1,\dots, 6.$ Then we have
\[
\bar{\mathfrak g}_{\mathcal V}   \cong {\mathfrak s}{\mathfrak o}(12)\oplus {\mathfrak s}{\mathfrak u}(2).
\]
This gives the fibration in (ix).

If $H = \frac{2}{3}H_{2},$ then
$$
\begin{array}{rcl}
\Delta^{+}(H) & = & \{\alpha_{1},\alpha_{p,q}\;
(3\leq p\leq q \leq 7);\; \alpha_{1} + \alpha_{3,p}\; (3\leq p\leq 7)\},\\[0.4pc]
\Delta^{+}_{2} & = & \left\{\alpha_{1,6}+\alpha_{2,5}+\alpha_{4},
\alpha_{1,7} + \alpha_{2,q} + \alpha_{4}\;(q=5,6), \alpha_{1,7} + \alpha_{2,6} + \alpha_{4,5}, \right.\\[0.4pc]
 & & \hspace{.5cm} \left. \alpha_{1,7} + \alpha_{2,6} + \alpha_{p,5} + \alpha_{4}\; (p=3,4),\mu\right\}.
\end{array}
$$
Hence
\[ \textstyle
{\mathfrak g}_{\mathcal V} = \bar{\mathfrak g}_{\mathcal V} =
{\mathfrak h} \oplus \sum_{\alpha\in \Delta^{+}(H)\cup
\Delta^{+}_{2}\atop a=0,1}\mathbb{R} U^{a}_{\alpha}.
\]
Put $\beta_{1} = \alpha_{1},$ $\beta_{2} = \alpha_{3},$ $\beta_{3}
= \alpha_{4},$ $\beta_{4} = \alpha_{5},$ $\beta_{5} = \alpha_{6},$
$\beta_{6} = \alpha_{7}$ and $\beta_{7} = \alpha_{1,6} +
\alpha_{2,5} + \alpha_{4}.$ Then $\{\beta_{1},\dots,\beta_{7}\}$
is a system of simple roots of ${\mathfrak a}_{7}$ and the
positive root set generated by them coincides with
$\Delta^{+}(H)\cup \Delta^{+}_{2}.$ Hence ${\mathfrak g}_{\mathcal
V} = \bar{\mathfrak g}_{\mathcal V}\cong {\mathfrak s}{\mathfrak
u}(8)$ and we obtain (x).

If $H = \frac{2}{3}H_{6},$ then
$$
\begin{array}{rcl}
\Delta^{+}(H) & = &\{\alpha_{1},\alpha_{2};\; \alpha_{p,q}\; (3\leq p\leq q \leq 5);\;
\alpha_{7}, \alpha_{1,p}, \alpha_{2,p}\; (p = 4,5); \; \alpha_{1} + \alpha_{3,p}\; (p=3,4,5);\\[0.4pc]
 & & \hspace{0.5cm} \alpha_{2} + \alpha_{4,p}\; (p= 4,5);\;  \alpha_{4} + \alpha_{p,5}\;
  (p=1,2;);\;  \alpha_{1,5} + \alpha_{3,4}\},\\[0.4pc]
\Delta^{+}_{2} & = & \{\alpha_{2,7} + \alpha_{4,6}, \alpha_{1,7} + \alpha_{p,6}\;(p=3,4), \alpha_{1,7} + \alpha_{p,6}+\alpha_{4}\;(p=2,3),\\[0.4pc]
 & & \hspace{0.5cm}\alpha_{1,7} + \alpha_{p,6} + \alpha_{4,5}\;(p=2,3),\alpha_{1,7} + \alpha_{2,6} + \alpha_{p,5} + \alpha_{4}\;(p=3,4), \mu\}.
 \end{array}
$$
Therefore,
\[ \textstyle
{\mathfrak g}_{\mathcal V}  =  \sum_{\alpha\in \Delta^{+}_{2}\cup
(\Delta^{+}(H)\setminus\{\alpha_{7}\})\atop a =0,1}\mathbb{R}
U^{a}_{\alpha} + \sum_{\alpha\in
\Delta^{+}_{2}}\mathbb{R}\sqrt{-1}H_{\alpha}.
\]

Let $\Delta^{+}_{{\mathfrak d}_{6}}$ be the  positive root set for
${\mathfrak d}_{6}$ generated by a system of simple roots $\pi_{6}
= \{\beta_{1},\dots ,\beta_{6}\}.$ Let $\phi$ be the bijection
$\phi:\pi_{6} \to \{\alpha_{2,7} + \alpha_{4,6}\,,
\alpha_{1},\dots ,\alpha_{5}\},$ given by
\[
\phi(\beta_{1}) = \alpha_{2,7} + \alpha_{4,6}\,
,\;\;\phi(\beta_{2}) = \alpha_{1},\;\;\phi(\beta_{3}) =
\alpha_{3},\;\;\phi(\beta_{4}) = \alpha_{4},\;\; \phi(\beta_{5}) =
\alpha_{2},\;\; \phi(\beta_{6}) = \alpha_{5}.
\]
Then $\phi$ can be extended by linearity  to
$\Delta^{+}_{{\mathfrak d}_{6}} \to \Delta^{+}_{2}\cup
(\Delta^{+}(H)\setminus\{\alpha_{7}\})$. This implies that
${\mathfrak g}_{\mathcal V}  \cong {\mathfrak s}{\mathfrak
o}(12).$ Moreover, we get
\[
\bar{\mathfrak g}_{\mathcal V}   =  {\mathfrak g}_{\mathcal V}
\oplus \mathbb{R}\sqrt{-1}H_{\alpha_{7}}\oplus\mathbb{R}
U^{0}_{\alpha_{7}}\oplus \mathbb{R} U^{1}_{\alpha_{7}}\cong
{\mathfrak s}{\mathfrak o}(12)\oplus {\mathfrak s}{\mathfrak
u}(2).
\]
Hence we get the fibration given in (xi). \vspace{2mm}

On ${\mathfrak e}_{8}$,
$$
\xymatrix@R=.5cm@C=.8cm{
& & & & \stackrel{3}{\stackrel{\circ}{\alpha_{2}}} \ar@{-}[d] & &\\
\;\;\;\;\; \stackrel{2}{\stackrel{\circ}{\alpha_{8}}} \ar@{-}[r] &
\stackrel{3}{\stackrel{\circ}{\alpha_{7}}} \ar@{-}[r] &
\stackrel{4}{\stackrel{\circ}{\alpha_{6}}} \ar@{-}[r] &
\stackrel{5}{\stackrel{\circ}{\alpha_{5}}} \ar@{-}[r] &
\stackrel{6}{\stackrel{\circ}{\alpha_{4}}} \ar@{-}[r] &
\stackrel{4}{\stackrel{\circ}{\alpha_{3}}} \ar@{-}[r] &
\stackrel{2}{\stackrel{\circ \,\, ,}{\alpha_{1}}}}
$$
we consider the automorphisms $\sigma$ of Type $A_{3}III$
determined by $H = \frac{2}{3}H_{1}$ and $H = \frac{2}{3}H_{8}.$

If $H = \frac{2}{3}H_{8},$ we get that $\Delta^{+}(H)$ is the set
of positive roots of ${\mathfrak e}_{7}$ given below and
\[
\Delta^{+} = \Delta^{+}(H) \cup \{\mu\} \cup \{\alpha, \mu-\alpha\},
\]
where $\mu = 2\alpha_{1,8} +\alpha_{2,7} + \alpha_{3,6} +
\alpha_{4,5} + \alpha_{4}$ and $\alpha$ is any element of the
following set:
$$
\begin{array}{l}
\{\alpha_{p,8}\; (1\leq p\leq 8), \alpha_{1,8} + \alpha_{3,q},
\alpha_{1,8} + \alpha_{4,q}, \alpha_{2,8} + \alpha_{4,q}\;(4\leq
q\leq 7),
\alpha_{3,8} + \alpha_{1}, \alpha_{4,8} + \alpha_{2},\\[0.4pc]
\hspace{1cm} \alpha_{1,8} + \alpha_{3,q} + \alpha_{4}\;(5\leq
q\leq 7), \alpha_{1,8} + \alpha_{2,q} + \alpha_{4}\; (5\leq q \leq
6), \alpha_{1,8} + \alpha_{2,6} + \alpha_{4,5}\}.
\end{array}
$$
Since  $\Delta^{+}_{2} = \{\mu\},$ ${\mathfrak g}_{\mathcal V}
\cong {\mathfrak s}{\mathfrak u}(2)$ and then $\bar{\mathfrak
g}_{\mathcal V}   \cong {\mathfrak e}_{7}\oplus {\mathfrak
s}{\mathfrak u}(2).$ This gives the fibration in (xii).

If $H = \frac{2}{3}H_{1},$ we have
$$
\begin{array}{rcl}
\Delta^{+}(H) &\hspace{-0,4mm} = \hspace{-0,4mm}& \{\alpha_{2}, \alpha_{p,q}\;(3\leq p\leq q
\leq 8), \alpha_{2,p}\;(4\leq p \leq 8),
 \alpha_{4} + \alpha_{2,p}\;(5\leq p \leq 8),\\[0.4pc]
&\hspace{-0,4mm} &\hspace{-0,4mm}  \alpha_{2} + \alpha_{4,p}\; (4\leq p \leq 8),\alpha_{2,p} + \alpha_{4,5}\; (6\leq p \leq 8),
 \alpha_{2,p} + \alpha_{4,6}\; (p = 7,8), \alpha_{2,8} + \alpha_{4,7}\},\\[0.4pc]
\Delta^{+}_{2} &\hspace{-0,4mm} =\hspace{-0,4mm} & \{\mu, \mu -\alpha\},
\end{array}
$$
where $\alpha$ is any element of
\[
\{\alpha_{2}, \alpha_{p,q}\;(3\leq p\leq q \leq 8), \alpha_{2,p}\; (4\leq p \leq 8), \alpha_{2} + \alpha_{4,8}, \alpha_{2,8} + \alpha_{4,p}\; (4\leq p \leq 8)\}.
\]
Hence
\[ \textstyle
{\mathfrak g}_{\mathcal V} = \bar{\mathfrak g}_{\mathcal V} =
{\mathfrak h} \oplus \sum_{\alpha\in \Delta^{+}(H)\cup
\Delta^{+}_{2}\atop a=0,1}\mathbb{R} U^{a}_{\alpha}.
\]
Put  $\beta_{1} = \mu - (\alpha_{2,8} + \alpha_{4,8}) =
\alpha_{1,7} + \alpha_{1,6} + \alpha_{3,5} + \alpha_{4}$ and
$\beta_{k} = \alpha_{10-k},$ for $k\in \{2,\dots, 8\}.$ Then
$\{\beta_{1},\dots,\beta_{8}\}$ is a system of simple roots of
${\mathfrak d}_{8}$ and the positive root set generated by them
coincides with $\Delta^{+}(H)\cup \Delta^{+}_{2}.$ Hence
${\mathfrak g}_{\mathcal V} = \bar{\mathfrak g}_{\mathcal V}\cong
{\mathfrak s}{\mathfrak o}(16)$ and we obtain the fibration in
(xiii).
\end{proof}

\begin{ack}{\rm Research supported by grants from  MICINN (Spain) projects\\ MTM2007-65852 and
MTM2009-13383.}
\end{ack}

\end{document}